\documentclass[12pt]{article}

\usepackage[colorlinks, linkcolor=blue, citecolor=blue]{hyperref}           
\usepackage{color}
\usepackage{graphicx,subfigure,amsmath,amssymb,amsfonts,bm,epsfig,epsf,url,dsfont}
\usepackage{amsthm}

\newtheorem{thm}{Theorem}[section]
\newtheorem{lem}[thm]{Lemma}
\newtheorem{assum}[thm]{Assumption}
\newtheorem{proposition}[thm]{Proposition}
\newtheorem{remark}[thm]{Remark}
\newtheorem{definition}[thm]{Definition}
\newtheorem{example}[thm]{Example}
\newtheorem{corollary}[thm]{Corollary}

\usepackage{tikz}
\usepackage{bbm}
\usepackage[small,bf]{caption}
\usepackage{natbib}
\usepackage[top=1in,bottom=1in,left=1in,right=1in]{geometry}
\usepackage{fancybox}
\usepackage{hyperref}
\usepackage{algorithm}
\usepackage{verbatim}

\usepackage{threeparttable}
\usepackage[titletoc,toc]{appendix}

\usepackage{mathtools}

\usepackage{scalerel,stackengine}
\stackMath
\newcommand\reallywidehat[1]{%
\savestack{\tmpbox}{\stretchto{%
  \scaleto{%
    \scalereL_[\widthof{\ensuremath{#1}}]{\kern-.6pt\bigwedge\kern-.6pt}%
    {\rule[-\textheight/2]{1ex}{\textheight}}
  }{\textheight}%
}{0.5ex}}%
\stackon[1pt]{#1}{\tmpbox}%
}

\newcommand*{\rom}[1]{\expandafter\@slowromancap\romannumeral #1@}

\newcommand{\Cov}{\mathrm{Cov}}

\DeclareMathOperator*{\argmax}{arg max}
\DeclareMathOperator*{\argmin}{arg min}

\usepackage{xspace}

\numberwithin{equation}{section}

\allowdisplaybreaks
\usepackage[utf8]{inputenc}
\usepackage{authblk}
\usepackage{url}
\usepackage[left=1in,right=1in,top=1in,bottom=1in]{geometry}

\usepackage{titlesec}
\titleformat{\part}[display]
  {\centering\bfseries\LARGE}
  {}{0pt}{}
\titlespacing*{\part}{0pt}{1.0em}{1.0em}

\usepackage{booktabs}
\usepackage{array}
\usepackage{enumitem}
\newcolumntype{L}[1]{>{\raggedright\arraybackslash}p{#1}}
\setlist[itemize]{leftmargin=*,nosep,topsep=0pt,parsep=0pt,partopsep=0pt}

\usepackage{tablefootnote}

\makeatletter
\renewcommand\paragraph{\@startsection{paragraph}{4}{\z@}%
  {1ex}
  {-1em}
  {\normalfont\normalsize\bfseries}}
\makeatother

\begin{document}

\title{The generalized method of moments is (almost) statistically efficient in low-SNR Gaussian latent-variable models }

\author[1]{Amnon Balanov\thanks{Corresponding author: \url{samnonb@campus.technion.ac.il}}}
\author[1]{Tamir Bendory}
\author[2]{Dan Edidin}

\affil[1]{\normalsize Department of Electrical and Computer Engineering, Technion---Israel Institute of Technology, Haifa 3200003, Israel}

\affil[2]{Department of Mathematics, University of Missouri, Columbia, MO 65211, USA}

\maketitle

\begin{abstract}
We study estimation in the low signal-to-noise ratio (SNR) regime for a broad class of Gaussian latent-variable models, including Gaussian mixtures and orbit-recovery problems. 
We show that, in this regime, the generalized method-of-moments (GMoM) matches the first-order asymptotic efficiency of maximum likelihood.  In particular, if the moment features are chosen up to the minimal local order required for identification and are weighted optimally, then the resulting GMoM estimator has the same leading asymptotic covariance as the maximum-likelihood estimator. We further show that this efficiency can be achieved using practical, data-driven weighting schemes and is preserved under a broad class of simpler weighting choices, including the commonly used identity-weighted method of moments, although these may incur larger higher-order efficiency losses. On the other hand, inappropriate weighting or discarding moment entries can lead to persistent efficiency losses even as SNR tends to zero, despite preserving identifiability. These results distinguish the moments required for identifiability from those needed for statistical efficiency, and characterize when moment-based estimation can match the performance of maximum likelihood. 
\end{abstract}

\newpage
\tableofcontents

\newpage

\section{Introduction}
Gaussian latent-variable models provide a common framework for statistical inference from noisy observations whose distribution depends on the unknown signal through an unobserved label, transformation, or nuisance variable. They are central in statistics, signal processing, and machine learning~\cite{bishop2006pattern}, and include as special cases Gaussian mixture models~\cite{lindsay1995mixture,heinrich2018strong,fruhwirth2006finite, mclachlan2019finite}, and orbit-recovery models under compact group actions~\cite{bandeira2023estimation,romanov2021multi}. These models also arise in important scientific applications, notably in structural biology, where single-particle cryo-electron microscopy (cryo-EM) and cryo-electron tomography (cryo-ET) can be viewed as latent-variable inference problems with extremely low signal-to-noise ratio (SNR) \cite{bendory2020single}. In these low-SNR regimes, each observation carries only weak information about the underlying signal, making the latent label or transformation effectively unrecoverable at the single-sample level~\cite{aguerrebere2016fundamental,serov2026interplay}.

\paragraph{Likelihood-based versus moment-based estimation.}
This raises a basic statistical question: how efficiently can the underlying signal be estimated when the latent variables themselves cannot be reliably recovered? Maximum likelihood provides the classical benchmark for such questions in parametric estimation: under standard regularity conditions, the maximum-likelihood estimator (MLE) attains the Cram\'er-Rao lower bound determined by the Fisher information~\cite{cramer1999mathematical}. In latent-variable models, this benchmark is formulated in terms of the observed, or marginal, likelihood, obtained after integrating out the unobserved nuisance variable~\cite{dempster1977maximum,mclachlan2008algorithm}. This distinction is particularly important in the low-SNR regime: estimation must aggregate information across many observations while marginalizing over the latent structure, rather than relying on sample-by-sample recovery of the latent variables. However, in low SNR, likelihood-based analysis and optimization are often difficult: the marginal likelihood is typically analytically intractable, and numerical procedures may be computationally demanding and converge very slowly~\cite{balanov2026expectation}.
These difficulties motivate moment-based methods, which replace the full likelihood by empirical moments.
Such methods often lead to polynomial estimation problems that are more amenable to analysis and computation, and their relevant statistics can typically be computed in a single pass over the data, with dimension determined by the signal dimension and the chosen moment order rather than by the sample size~\cite{hsu2013learning,anandkumar2014tensor,abbe2018multireference,bendory2017bispectrum}. 

Within this class, the generalized method-of-moments\footnote{The generalized method of moments is often abbreviated as GMM; here, we use GMoM to avoid confusion with Gaussian mixture models.} (GMoM) is the classical moment-based framework, originating in econometrics: for a fixed collection of moment conditions, it matches model-implied moments to their empirical counterparts through a weighted quadratic criterion, and, with the optimal weighting, achieves the best asymptotic efficiency available within that class \cite{hansen1982large,newey1994large,hall2004generalized}. Its usual drawback relative to maximum likelihood is statistical: moment-based estimators are generally not asymptotically efficient
~\cite{van2000asymptotic,carrasco2014asymptotic}. 
The main result of this paper shows that this is not the case in the low-SNR regime for latent Gaussian models: once moments are taken up to the local identifiability order, the optimally weighted GMoM estimator achieves the same first-order asymptotic efficiency as maximum likelihood.

\subsection{Problem formulation}

In this work we study a unified Gaussian \emph{mean-scaling} framework that covers both finite-mixture models and continuous group-action models. Specifically, we observe i.i.d.\ samples $y_1,\ldots,y_n\in\mathbb{R}^d$ generated as
\begin{align}
    \label{eq:meanscaling-model}
    (\text{Gaussian latent-variable model})
    \qquad
    y_i = \beta\,A(Z_i)\theta^\star \;+\; \xi_i, 
\end{align}
where $\theta^\star\in\Theta\subseteq\mathbb{R}^m$ is the unknown parameter, $\beta\ge 0$ is a signal-strength parameter, $\xi_i \stackrel{\text{i.i.d.}}{\sim}\mathcal{N}(0,I_d)$ is additive normal noise, and $Z_i\sim\mu$ is a latent variable taking values in a measurable space $\mathcal{Z}$ (possibly discrete, continuous, or hybrid). The map $A:\mathcal{Z}\to\mathbb{R}^{d\times m}$ is known and encodes the latent structure (see Table~\ref{tab:models} and Section~\ref{subsec:representative_models} for different applications); e.g., for Gaussian mixtures, $Z_i$ is a discrete label and $A(Z_i)$ selects the corresponding mean; for orbit recovery, $Z_i$ is a group element and $A(Z_i)$ is the corresponding (typically orthogonal) action on the ground-truth; and for mixture-orbit hybrids, $Z_i$ contains both a label and a transformation. 
Although we write $\xi_i\sim\mathcal{N}(0,I_d)$ for notational simplicity, all results extend verbatim to Gaussian noise with a known, nonsingular covariance matrix $\Sigma$ (see Remark~\ref{remark:general_covariance_matrix}); we focus on the isotropic case throughout to streamline the presentation.

Our goal is to estimate $\theta^\star$ from $\{y_i\}_{i=1}^n$. In general, however, $\theta^\star$ is identifiable only up to a model-intrinsic equivalence symmetry, so the proper inferential target is the equivalence class $[\theta^\star]$. For example, in a $K$-component Gaussian mixture model, the distribution is invariant under permutation of the component labels, while in orbit-recovery models the distribution is invariant under the corresponding group action \cite{bandeira2023estimation,fan2023likelihood}. 
Accordingly, we formulate estimation and efficiency on the quotient space induced by this equivalence relation. We denote this subspace by $W$, which is the normal space at $\theta^\star$. Directions in $W$ are those that are statistically distinguishable from the ground truth, whereas tangent directions to $[\theta^\star]$ merely change the representative within the same equivalence class; see Definition~\ref{def:tangent_normal_pi}.

\paragraph{Moment identifiability and sample complexity.}
We focus on the low-SNR regime $\mathrm{SNR}\triangleq\beta^2\to 0$, where individual observations carry vanishing information about the unknown parameter, yet consistent estimation remains possible by aggregating sufficiently many samples. A central insight from previous works is that the sample complexity of recovery is governed by the first moment order that identifies the parameter up to its intrinsic symmetries~\cite{bandeira2023estimation,katsevich2023likelihood}. More precisely, the global identifiability order $r_{\mathrm{glob}}$ is the smallest moment order that distinguishes the equivalence class $[\theta^\star]$ from all other classes. Recovering the parameter to a fixed accuracy requires $n\gtrsim\mathrm{SNR}^{-r_{\mathrm{glob}}}$ observations, with matching upper bounds under suitable regularity conditions~\cite{balanov2026sample, bandeira2023estimation, abbe2018estimation}. In important examples, including multireference alignment (MRA), orbit recovery, and Gaussian mixture models, this order is often $2$ or $3$ for generic signals~\cite{hsu2013learning,anandkumar2014tensor,abbe2018multireference,bandeira2023estimation,bendory2022dihedral}.

However, global identifiability does not  characterize the difficulty of local estimation. Recent results distinguish $r_{\mathrm{glob}}$ from the local identifiability order $r_{\mathrm{loc}}$, defined as the smallest moment order whose differential detects every infinitesimal perturbation of $\theta^\star$~\cite{balanov2026sample}. These orders govern different statistical scales: global localization and local refinement, respectively. Since statistical efficiency concerns the asymptotic covariance of estimators near $\theta^\star$, our analysis focuses on the local structure and the corresponding Fisher information. In particular, we analyze whether GMoM attains the same asymptotic efficiency as MLE in low SNR, and how this equivalence depends on the moment cutoff and weighting matrix.

\paragraph{The GMoM estimator.}
Unlike maximum likelihood, which uses the full observed-data density, GMoM replaces the likelihood by a weighted discrepancy between empirical moments and their corresponding population values:
\begin{align}
    \label{eqn:gmomGeneral_efficiency} \widehat{\theta}_{\mathrm{GMoM}}(\mathcal{Y}) \in \argmin_{\theta\in\Theta} \; \Bigl\{ \bigl(M_n(\mathcal{Y}) -M(\theta;\beta)\bigr)^\top \Omega_n \bigl(M_n(\mathcal{Y})-M(\theta;\beta)\bigr)\Bigr\},
\end{align}
where $M_n(\mathcal{Y})$ is the empirical moment vector built from the observations $\mathcal{Y}$, $M(\theta;\beta)$ is the corresponding population moment vector, and $\Omega_n$ is a positive semi-definite weighting matrix. In this work, we consider GMoM estimators built from empirical moments of the observations up to a prescribed cutoff order $L_{\mathrm{mom}}$. For this moment cutoff, the vectors $M_n(\mathcal{Y})$ and $M(\theta;\beta)$ are obtained by stacking the empirical and population moments of orders $1,\dots,L_{\mathrm{mom}}$; see Section~\ref{subsec:empirical_features_main} for formal definition.

The choice of the limiting weighting matrix $\Omega_n\xrightarrow{\mathbb P}\Omega$ plays a central role in determining the statistical efficiency of GMoM relative to MLE. As established in classical GMoM theory~\cite{hansen1982large,newey1994large}, the optimal choice is the inverse moment covariance, $\Omega_{\mathrm{opt}}=\Sigma_{L_{\mathrm{mom}}}(\theta^\star;\beta)^{-1}$, where $\Sigma_{L_{\mathrm{mom}}}(\theta^\star;\beta)$ denotes the covariance of the stacked moment features. This choice minimizes the asymptotic covariance among GMoM estimators using the same moments. Beyond optimal weighting, other alternatives are frequently used: block-isotropic weighting, which assigns a common positive weight to all moment coordinates of a given order; and selective moment weighting, which retains only a subset of the available moments, as in frequency-marching methods~\cite{bendory2017bispectrum, bendory2025orbit,balanov2026provable}.

Two particularly simple block-isotropic choices are the identity weighting, $\Omega_I=I$, and the inverse zero-signal covariance weighting, $\Omega_0=(\Sigma_{L_{\mathrm{mom}}}^{(0)})^{-1}$, where $\Sigma_{L_{\mathrm{mom}}}^{(0)}=\lim_{\beta\to 0}\Sigma_{L_{\mathrm{mom}}}(\theta^\star;\beta)$.  Explicitly, the zero-signal covariance has the form
\begin{align}
    \Sigma_{L_{\mathrm{mom}}}^{(0)}=\operatorname{diag}\left(1!I_{N_1},2!I_{N_2},\ldots, L_{\mathrm{mom}}!I_{N_{L_{\mathrm{mom}}}} \right),    \label{eq:zero_signal_covariance}
\end{align}
where $N_k=\binom{d+k-1}{k}$ is the number of independent symmetric moment coordinates of order $k$.
Both identity and inverse zero-signal covariance weighting are independent of the unknown signal and avoid covariance estimation, but the latter has the additional advantage of approaching the optimal weighting as $\beta \to 0$.

\paragraph{Motivating example.} Figure~\ref{fig:mainFigure}(a) illustrates two key observations concerning the statistical efficiency of GMoM and the choice of the weighting matrix. We consider a Gaussian mixture model with local identifiability moment order $r_{\mathrm{loc}}=3$, using moments up to $L_{\mathrm{mom}}=3$, and measure the relative excess MSE compared with MLE, $\mathrm{MSE}_{\mathrm{GMoM}}/\mathrm{MSE}_{\mathrm{MLE}}-1$, for $n=10^5$ observations. First, the efficiency gap decreases as SNR decreases for several weighting schemes, suggesting that GMoM can approach 
the statistical efficiency of MLE in the low-SNR regime. Second, the behavior depends strongly on the choice of weighting matrix: optimal covariance weighting achieves the smallest error, followed by inverse zero-signal covariance and identity weighting, whereas anisotropic weighting exhibits a substantially larger gap that does not vanish as SNR decreases. These observations motivate our analysis of the conditions under which GMoM achieves leading MLE efficiency and the additional efficiency loss induced by different weighting matrices.

\begin{figure*}[t!]
    \centering
    \includegraphics[width=\linewidth]{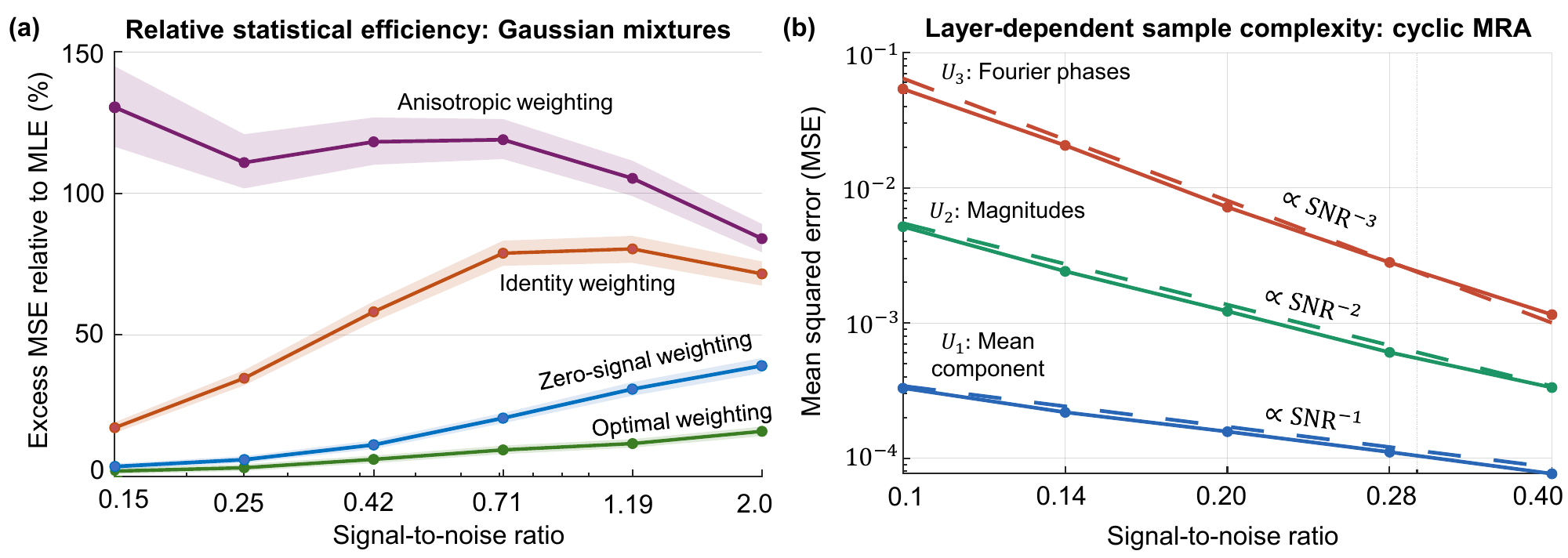}  
    \caption{\textbf{Statistical efficiency and layer-dependent estimation difficulty in the low-SNR regime.} \textbf{(a)} 
    Relative excess mean-squared error (MSE) compared with MLE for optimal-weighted, inverse zero-signal covariance-weighted~\eqref{eq:zero_signal_covariance}, identity-weighted, and anisotropically weighted GMoM in a three-component Gaussian mixture model in $\mathbb{R}^4$, using moments through order $L=3$.  Shaded bands indicate pointwise 95\% paired-bootstrap confidence intervals based on 1000 independent Monte Carlo trials, with $n=10^5$ observations per trial.
    \textbf{(b)} MLE estimation errors along different parameter directions in uniform cyclic MRA with signal length $d=5$. The MSE, after optimal cyclic-shift alignment, are projected onto the three informative subspaces of the cyclic MRA model~\cite{balanov2026expectation}: the DC component (dimension $1$), Fourier-magnitude directions (dimension $2$), and Fourier-phase directions (dimension $2$). Results are averaged over $400$ independent Monte Carlo trials for sample sizes $n=32,000$. The reference lines indicate the predicted $\mathrm{SNR}^{-1}$, $\mathrm{SNR}^{-2}$, and $\mathrm{SNR}^{-3}$ scaling of the respective layer-wise MSEs.}    \label{fig:mainFigure}
    \end{figure*}

\paragraph{Measures for relative statistical efficiency.}
We quantify statistical efficiency relative to the inverse observed Fisher information, which determines the classical Cram\'er--Rao lower bound. We introduce two complementary measures based on asymptotic covariance matrices. Consider a consistent estimator $\widehat{\theta}_n$ such that, for each fixed $\beta>0$, its estimation error projected onto the normal space $W$ satisfies
\begin{align}
    \sqrt{n}\,\Pi_W(\widehat{\theta}_n-\theta^\star)
    \xrightarrow{\mathcal D}
    \mathcal{N}(0,C),
    \label{eq:general_estimator_asymptotic_normality}
\end{align}
where $\Pi_W$ denotes the projection onto the normal subspace $W$, and $C=C(\theta^\star;\beta)\succ0$ denotes its asymptotic covariance matrix on $W$. In particular, for MLE, $C=\mathcal{I}_\beta^{-1}$, where $\mathcal{I}_\beta\triangleq\mathcal{I}_{\mathrm{obs}}^{(W)}(\theta^\star;\beta)$ is the observed Fisher information restricted to $W$; see Definition~\ref{def:I_obs_meanscaling}. For any such covariance matrix $C$, we define
\begin{align}
    \mathcal{E}_{\mathrm{op}}(C;\beta) &\triangleq \left\| \mathcal{I}_\beta^{1/2} \left( C  - \mathcal{I}_\beta^{-1} \right) \mathcal{I}_\beta^{1/2} \right\|_{\mathrm{op}},    \label{eq:relative_efficiency_op}
    \\    
    \mathcal{E}_{\mathrm{tr}}(C;\beta)
    &\triangleq \frac{\operatorname{tr}(C)}{\operatorname{tr}(\mathcal{I}_\beta^{-1})}-1.    \label{eq:relative_efficiency_tr}
\end{align}
The first measure quantifies the worst-direction relative asymptotic variance inflation, whereas the second measures the relative increase in total asymptotic mean-squared error (MSE), since $\mathbb E\|\Pi_W(\widehat{\theta}_n-\theta^\star)\|^2\sim n^{-1}\operatorname{tr}(C)$. For any estimator satisfying $C\succeq\mathcal{I}_\beta^{-1}$, both measures are nonnegative and satisfy
\begin{align}
    0\le \mathcal{E}_{\mathrm{tr}}(C;\beta) \le \mathcal{E}_{\mathrm{op}}(C;\beta).    \label{eq:relative_efficiency_comparison}
\end{align}
A central question of this work is how these efficiency losses behave as $\beta\to 0$ for GMoM under different weighting matrices, and in particular, under what conditions they vanish or remain bounded away from zero.

\subsection{Main results}
\label{subsec:main_contributions}

Our results address two closely related questions: how efficiently moment-based estimators use the available observations, and how the difficulty of estimating different parameter directions depends on SNR. Both phenomena are illustrated in Figure~\ref{fig:mainFigure}.

\subsubsection{Low-SNR efficiency of generalized method of moments}

Our main result establishes that, for each fixed sufficiently small SNR, both MLE and GMoM are asymptotically normal on the normal space $W$ and their information matrices agree to leading order as SNR tends to zero. Recall that $L_{\mathrm{mom}}$ is the user-chosen moment cutoff, whereas $r_{\mathrm{loc}}$ is the smallest moment order that locally identifies all directions in $W$ (Definition~\ref{def:rloc}), and $r_{\mathrm{glob}}$ is the smallest moment order that globally identifies the equivalence class of the signal~\cite{balanov2026sample}. While $L_{\mathrm{mom}}\ge r_{\mathrm{glob}}$ ensures global consistency of GMoM for generic signals, $L_{\mathrm{mom}}\ge r_{\mathrm{loc}}$ suffices for local asymptotic efficiency. 

\begin{thm}[Informal statement: efficiency equivalence of MLE and GMoM in low SNR]
\label{thm:informal1}
Fix a sufficiently small $\mathrm{SNR}=\beta^2$, and let $L_{\mathrm{mom}}\ge r_{\mathrm{loc}}$. Suppose that the GMoM estimator~\eqref{eqn:gmomGeneral_efficiency} is constructed from all empirical moments up to order $L_{\mathrm{mom}}$, and is weighted optimally by the inverse covariance of the corresponding empirical moment vector. Then:
\begin{enumerate}
    \item \emph{(Asymptotic normality).} Assume that $L_{\mathrm{mom}}\ge r_{\mathrm{glob}}$. Then, as $n\to\infty$,
    \begin{align}
        \sqrt{n}\,\Pi_W\bigl(\widehat{\theta}_{\mathrm{MLE}}-\theta^\star\bigr)
        &\xrightarrow[]{\mathcal{D}} \mathcal{N}\!\left(0,\mathcal{I}_{\mathrm{obs}}^{(W)}(\theta^\star;\beta)^{-1}\right),\\ \sqrt{n}\,\Pi_W\bigl(\widehat{\theta}_{\mathrm{GMoM}}-\theta^\star\bigr)
        &\xrightarrow[]{\mathcal{D}} \mathcal{N}\!\left(0,\mathcal{I}_{\mathrm{GMoM}}^{(W)}(\theta^\star;\beta)^{-1}\right),
    \end{align}
    where $\widehat{\theta}_{\mathrm{MLE}}$ denotes the maximum-likelihood estimator, $\mathcal{I}_{\mathrm{obs}}^{(W)}(\theta^\star;\beta)$ denotes the observed Fisher information matrix restricted to $W$, as defined in~\eqref{eq:restricted_obs_info_positive_mle}, and $\mathcal{I}_{\mathrm{GMoM}}^{(W)}(\theta^\star;\beta)$ denotes the corresponding GMoM information matrix on $W$, as defined in~\eqref{eq:def_GMoM_information}.

    \item \emph{(Higher-order information discrepancy and asymptotic efficiency).} The observed Fisher information and the GMoM information agree to leading order in low SNR, with $R_\beta = \mathcal{I}_{\mathrm{obs}}^{(W)}(\theta^\star;\beta) -\mathcal{I}_{\mathrm{GMoM}}^{(W)}(\theta^\star;\beta) \succeq 0$, and $\|R_\beta\|_{\mathrm{op}} = O\left(\mathrm{SNR}^{L_{\mathrm{mom}}+1}\right)$. In addition, the relative covariance losses of optimally weighted GMoM satisfy
    \begin{align}
        0 \le \mathcal{E}_{\mathrm{tr}} \left(\Sigma_{\mathrm{GMoM}};\beta\right) \le \mathcal{E}_{\mathrm{op}} \left(\Sigma_{\mathrm{GMoM}};\beta\right) = O\left(\mathrm{SNR}^{L_{\mathrm{mom}}+1-r_{\mathrm{loc}}}\right),
    \end{align}
    where $ \Sigma_{\mathrm{GMoM}}^{-1} (\theta^\star;\beta)\triangleq \mathcal{I}_{\mathrm{GMoM}}^{(W)}(\theta^\star;\beta)$.
    
\end{enumerate}
\end{thm}

The asymptotic normality of MLE and GMoM, for each fixed $\beta>0$, follows from standard asymptotic estimation theory, adapted to account for observational equivalence and the corresponding quotient geometry, as established in Appendices~\ref{subsec:cons_efficiency_MLE} and~\ref{subsec:GMoM_consistency_iormality}, respectively. The second result, concerning the higher-order information discrepancy and the resulting relative efficiency bound, is established in Theorem~\ref{thm:H_fixed_beta_fisher_moment_identity}.
We note that the large-sample and low-SNR limits in this work are taken sequentially. We first fix a small positive $\beta>0$ and study the regime $n\to\infty$, and then analyze the resulting information matrices in the limit $\beta \to 0$.  In Section~\ref{sec:discussion}, we numerically analyze the joint behavior as both $n$ increases and $\beta$ decreases. However, a rigorous analysis of the joint asymptotic regime would require additional tools to control the estimation error uniformly as the Fisher information degenerates, and is left for future work.

Beyond optimal weighting,
Sections~\ref{subsec:nonoptimal_weighting} and~\ref{subsec:selective_weighting} examine how alternative weighting schemes affect statistical efficiency. Remarkably, we show that leading MLE efficiency in the low-SNR regime does not require optimal weighting. In particular, for $L_{\mathrm{mom}}\geq r_{\mathrm{loc}}(\theta^\star)$, any fixed positive-definite block-isotropic weighting matrix $\Omega$ achieves a relative efficiency loss of at most $O(\mathrm{SNR})$:
\begin{align}
    0 \leq \mathcal{E}_{\mathrm{tr}} \big(\Sigma_{\mathrm{GMoM}}^{(\Omega)};\beta\big) \leq \mathcal{E}_{\mathrm{op}} \big(\Sigma_{\mathrm{GMoM}}^{(\Omega)};\beta\big) = O(\mathrm{SNR}).    \label{eq:intro_block_isotropic_efficiency}
\end{align}
This includes unweighted moment matching ($\Omega_I=I$), as well as inverse zero-signal covariance weighting $\Omega_0=(\Sigma_{L_{\mathrm{mom}}}^{(0)})^{-1}$; see Proposition~\ref{prop:block_isotropic_weighting}. Thus, first-order statistical efficiency is a considerably broader property of moment-based estimation than optimal GMoM weighting alone.
Nevertheless, the choice of weighting remains important for higher-order efficiency and finite-SNR performance. In contrast, we show that anisotropic weighting or selective moment estimation may incur a persistent efficiency loss even as $\mathrm{SNR}\to 0$, despite preserving local identifiability.

\subsubsection{Layered local geometry of the Fisher information}

The mechanism behind the preceding theorem is a layered local geometry induced by the differential moment maps. For the Gaussian latent-variable model studied here, the signal moments are $T_k(\theta)\triangleq \mathbb{E}_{Z\sim\mu}[(A(Z)\theta)^{\otimes k}]$ for $k\ge 1$. These tensors capture the signal-dependent part of the moment structure, with the additive Gaussian noise separated out. Their derivatives at the ground truth, $DT_k(\theta^\star)[h]$, describe how an infinitesimal perturbation $h$ changes the $k$-th signal moment. The first order at which this derivative is nonzero determines the low-SNR scale at which the perturbation $h$ first affects the local information.

Let 
\begin{align}
    V_k \triangleq  \{h\in W : DT_j(\theta^\star)[h]=0 \text{ for all } j<k\},
\end{align}
which defines a decreasing filtration of the normal space $W = V_1 \supseteq V_2 \supseteq \cdots$.
Thus, $V_k$ consists of directions that are invisible to all differential moments of order strictly less than $k$.
The quotient $V_k/V_{k+1}$ captures the genuinely new directions that become detectable at order $k$.
To represent these quotient layers by concrete linear subspaces, we define $U_k \triangleq  V_k \cap V_{k+1}^{\perp}$. Then, $V_k = U_k \oplus V_{k+1}$, and, when the filtration terminates at $r_{\mathrm{loc}}$, iteration yields the orthogonal decomposition $W = \bigoplus_{k=1}^{r_{\mathrm{loc}}} U_k$.
In this sense, $U_k$ is the $k$-th informative layer: it is a linear representative of the new directions that first become visible in the moment order $k$. In particular, every nonzero $h\in U_k$ satisfies
\begin{align}
    DT_j(\theta^\star)[h]=0 \quad \text{for all } j<k,
    \qquad \text{and} \quad DT_k(\theta^\star)[h]\neq 0.
\end{align}
Thus, different layers correspond to directions that emerge at different moment orders.

To describe the layered low-SNR geometry, we analyze the bilinear form of the restricted observed Fisher information across the informative layers. Specifically, for $1\leq k,\ell\leq r_{\mathrm{loc}}$ and vectors $h\in U_k$, $g\in U_\ell$, the bilinear form $\big\langle h,\mathcal{I}_{\mathrm{obs}}^{(W)}(\theta^\star;\beta)g\big\rangle$ admits a low-SNR expansion whose leading behavior reflects the moment orders at which the corresponding layers first become visible. Since $\mathcal{I}_{\mathrm{obs}}^{(W)}(\theta^\star;\beta)$ is self-adjoint and positive semidefinite, this bilinear form measures the information coupling between the two directions. 

\begin{thm}[Informal statement: layered low-SNR geometry]
Fix a sufficiently small  $\mathrm{SNR} = \beta^2$.
Then, for every $1\leq k,\ell\leq r_{\mathrm{loc}}$, every $h\in U_k$, and every $g\in U_\ell$,
\begin{align}
    \label{eq:block_fisher_layered_informal_snr}
    \big\langle h,\mathcal{I}_{\mathrm{obs}}^{(W)}(\theta^\star;\beta)g \big\rangle =
    \begin{cases}
        \displaystyle      \frac{\mathrm{SNR}^{k}}{k!}\,  \big\langle DT_k(\theta^\star)[h],DT_k(\theta^\star)[g]\big\rangle_F + O\bigl(\mathrm{SNR}^{k+1}\bigr),
        & k=\ell,\\[2ex]   O\bigl(\mathrm{SNR}^{(k+\ell+1)/2}\bigr),
        & |k-\ell|=1,\\[1.25ex]   O\bigl(\mathrm{SNR}^{(k+\ell+2)/2}\bigr),
        & |k-\ell|\geq 2.
    \end{cases}
\end{align}
An analogous expansion holds for the restricted GMoM information matrix. In particular, on each layer $U_k$, both operators have the same leading-order term, appearing at order $\mathrm{SNR}^k$.
\end{thm}

The expansion shows that directions within the same layer interact at their leading SNR scale, whereas directions belonging to different layers are coupled only at higher order. Thus, in the low-SNR limit, the Fisher information is approximately block-diagonal with respect to the informative-layer decomposition, and its leading eigendirections align with these layers. The same layered structure holds for the restricted GMoM information matrix. See Proposition~\ref{prop:fisher_layer_block_asymptotics} for the formal statement.
Figure~\ref{fig:mainFigure}(b) illustrates this layered structure in cyclic MRA, where $U_1$ corresponds to the mean component, $U_2$ to the Fourier-magnitude directions, and $U_3$ to the Fourier-phase directions~\cite{balanov2026expectation}. The estimation errors along these layers exhibit the predicted $\mathrm{SNR}^{-1}$, $\mathrm{SNR}^{-2}$, and $\mathrm{SNR}^{-3}$ scaling, respectively.

\subsection{Related work}
\label{subsec:prior_work_efficiency_polished}

Maximum likelihood is the classical benchmark for asymptotic efficiency. By contrast, the standard notion of efficiency in the GMoM literature is conditional on the chosen moment conditions: among GMoM estimators constructed from a given set of valid moment conditions, the optimally weighted estimator achieves the smallest asymptotic covariance within that class~\cite{van2000asymptotic,hansen1982large,newey1994large,hall2004generalized,carrasco2014asymptotic}. However, this notion of optimality is internal to the selected moment family and does not generally imply efficiency relative to the full likelihood. In particular, attaining the MLE efficiency bound requires the likelihood score to belong to the closed linear span of the selected moment functions~\cite[Proposition~4.1]{carrasco2014asymptotic}.

A substantial literature in the last decade on orbit recovery, invariants, and related latent-variable models has established that low-order moments govern identifiability and sample complexity in the low-SNR regime~\cite{bandeira2023estimation,perry2019sample,balanov2026sample}. Moment-based estimation procedures have also been developed for specific models, including multi-reference alignment and Gaussian mixtures~\cite{abas2022generalized,zhang2025diagonally}. However, identifying the moments sufficient for recovery does not determine whether an estimator based on those moments attains the statistical efficiency of MLE.

Related work has investigated the low-SNR structure of the likelihood and Fisher information. In Gaussian mixture models,~\cite{katsevich2023likelihood} established a low-SNR expansion of the population likelihood, relating likelihood optimization to successive moment-matching problems. In orbit-recovery models,~\cite{fan2023likelihood,fan2024maximum} characterized the high-noise likelihood landscape and its relationship to moment invariants. In particular,~\cite{fan2024maximum} established that the moment order required for generic local identification determines the scaling of the Fisher-information eigenvalues and their limiting informative subspaces, including in projected group-orbit models. This analysis also distinguished local identification from unique global recovery.  

Building on these results, we extend the analysis of the local Fisher-information structure to a broader class of Gaussian latent-variable models, without requiring the latent operators to arise from group actions. Our approach relies on Hermite expansions of the likelihood, together with the differential structure of the signal moments, to derive the information expansions needed for the main statistical-efficiency results.
In particular, we harness these tools to answer the main question of this work: under what conditions does a moment-based estimator capture the same leading statistical information as the full likelihood? 

\subsection{Numerical simulations}
\label{subsec:numerical_simulations}

The numerical experiments throughout the manuscript include both population-level information comparisons and finite-sample estimation experiments. In the population-level comparisons, Gaussian expectations are evaluated using Gauss--Hermite quadrature or quasi-Monte Carlo integration. In the finite-sample experiments, the MLE and GMoM objectives are optimized using MATLAB's \texttt{fminunc} with its quasi-Newton algorithm, initialized in a neighborhood of the ground truth to isolate local statistical efficiency from global optimization challenges. Frequency marching is implemented using selected Fourier invariants. Model-specific parameters and numerical settings are provided in the corresponding figure captions, with additional details given in Appendix~\ref{app:numerical_details}. The implementation is available in the accompanying code repository: \href{https://github.com/AmnonBa/statistical-efficiency-GMoM}{https://github.com/AmnonBa/statistical-efficiency-GMoM}.

\section{Statistical model and geometric preliminaries}
\label{sec:preliminaries}

In this section, we develop the geometric and statistical foundations for the efficiency analysis. 

\paragraph{Notation.} Throughout, $\xrightarrow[]{\mathcal{D}}$, $\xrightarrow[]{\mathbb{P}}$, $\xrightarrow[]{\mathrm{a.s.}}$, and $\xrightarrow[]{\mathcal{L}^p}$ denote convergence in distribution, in probability, almost surely, and in $\mathcal{L}^p$, respectively.
We write $\|\cdot\|$ for the Euclidean norm on vectors and $\|\cdot\|_F$ for the Frobenius norm on matrices. 
Let $V$ and $U$ be finite-dimensional normed vector spaces, and let $F:V\to U$ be continuously differentiable in a neighborhood of $\theta^\star\in V$. We denote by $DF(\theta^\star):V\to U$ the \emph{derivative} (or \emph{Jacobian}) of $F$ at $\theta^\star$, namely, the unique linear map satisfying
\begin{align}
    F(\theta^\star+h)-F(\theta^\star) = DF(\theta^\star)[h]+o(\|h\|),
    \qquad \|h\|\to 0.
\end{align}

\subsection{Representative model classes}
\label{subsec:representative_models}

Model~\eqref{eq:meanscaling-model} provides a unified Gaussian mean-scaling framework for a broad class of latent-variable models. Table~\ref{tab:models} lists representative examples, together with the latent space, the associated linear map $A(\cdot)$, and the induced observational equivalence relation. The examples illustrate that the same formalism covers finite latent classes, continuous latent group actions, projection models, and hybrid constructions within a single notation. 

\begin{table*}[t]
\centering
\begingroup
\setlength{\tabcolsep}{3.5pt}
\renewcommand{\arraystretch}{1.05}
\scriptsize

\caption{\textbf{Representative models covered by the mean-scaling Gaussian latent-variable framework.}
Each row specifies the latent space $\mathcal{Z}$, the parameter, the linear map $A(\cdot)$, and the intrinsic observational equivalence relation. We use $\mathcal{T}_\ell$ for circular shifts, $J$ for reflection, $\Pi$ for projection, $P_\pi$ for permutation matrices, and $\rho(g)$ for a linear representation of a group element $g$. In particular, if $R\in\mathrm{SO}(n)$ or $R\in\mathrm{SO}(3)$, then $R$ is a rotation and $\rho(R)$ is its induced action on the parameter space. For cryo-EM with translations, $S_t$ denotes the in-plane shift operator.}
\label{tab:models}

\begin{tabular}{|L{2.7cm}||L{1.80cm}|L{2.5cm}|L{2.65cm}|L{3.25cm}|L{2.35cm}|}
\hline
\textbf{Application} &
\textbf{Latent $Z$} &
\textbf{Parameter $\theta^\star$} &
\textbf{Linear map $A(Z)$} &
\textbf{Equivalence / symmetry ($[\theta^\star]$)} &
\textbf{Min.\ moments $r$ (typical)} \\
\hline\hline

\textbf{Spherical GMM} &
$\{1,\ldots,K\}$ &
$(\theta^\star_1,\ldots,\theta^\star_K)\in(\mathbb{R}^d)^K$ &
$A(z)\theta=\theta_z$ &
Label permutation $\theta\sim(\theta_{\pi(1)},\ldots,\theta_{\pi(K)})$ &
$K \geq 3:\; 3$ \cite{hsu2013learning,anandkumar2014tensor} $\qquad K = 2:\; 2$ \\
\hline

\textbf{Mixture of linear regressions $^\dagger$} &
$\{1,\ldots,K\}$ &
$(\theta^\star_1,\ldots,\theta^\star_K)\in(\mathbb{R}^d)^K$ &
$A(z;x)\theta=\langle x,\theta_z\rangle$  &
Label permutation $\theta\sim(\theta_{\pi(1)},\ldots,\theta_{\pi(K)})$ &
$K \geq 3:\; 3$ \cite{sedghi2016provable,chaganty2013spectral} $\qquad K = 2:\; 2$ \\
\hline

\textbf{Cyclic MRA} &
$\mathbb{Z}_d$ &
$\theta^\star\in\mathbb{R}^d$ &
$A(\ell)=\mathcal{T}_\ell$ &
$\mathbb{Z}_d\cdot\theta$ &
$3$ \cite{bendory2017bispectrum,bandeira2023estimation} \\
\hline

\textbf{Dihedral MRA} &
$D_{2d}$ &
$\theta^\star\in\mathbb{R}^d$ &
$A(g)\in\{\mathcal{T}_\ell,\mathcal{T}_\ell J\}$ &
$D_{2d}\cdot\theta$ &
$3$ \cite{bendory2022dihedral,edidin2025orbit} \\
\hline

\textbf{Projected MRA} &
$\mathbb{Z}_d$ &
$\theta^\star\in\mathbb{R}^d$, $d$ odd &
$A(\ell)=\Pi\mathcal{T}_\ell$ &
$D_{2d}\cdot\theta$ &
$3$ \cite{bandeira2023estimation,balanov2026projected}  
\\
\hline

\textbf{Coordinate sign-flip model} &
$\{\pm1\}^d$ &
$\theta^\star\in\mathbb{R}^d$ &
$A(s)=\mathrm{diag}(s)$ &
$\theta\sim s\odot\theta$ &
$2$ \\
\hline

\textbf{Permutation group} &
$S_d$ &
$\theta^\star\in\mathbb{R}^d$ &
$A(\pi)=P_\pi$ &
$\theta\sim P_\pi\theta$ &
$d$ \cite{bandeira2023estimation} \\
\hline

\textbf{Orbit recovery under $\mathrm{SO}(n)$} &
$\mathrm{SO}(n)$ &
bandlimited $n$-dimensional volume &
$A(R)=\rho(R)$ &
$\theta\sim\rho(R)\theta$ &
$3$ \cite{bandeira2023estimation,fan2024maximum,bendory2025orbit} \\
\hline

\textbf{$\mathbb{S}^2$ registration} &
$\mathrm{SO}(3)$ &
bandlimited $f^\star:\mathbb{S}^2\to\mathbb{R}$ &
$A(R)=\rho(R)$ &
$f\sim\rho(R)f$ &
$3$ \cite{bandeira2023estimation,fan2024maximum} \\
\hline

\textbf{Rigid motion} &
$\mathrm{SE}(n)$ &
bandlimited $n$-dimensional volume  &
$A(g) = g \cdot f$,  $g \in \mathrm{SE}(n)$ &
Global rigid motion &
$3$--$5$ \cite{balanov2025orbit} \\
\hline

\textbf{Cryo-EM without translations} &
$\mathrm{SO}(3)$ &
bandlimited 3D volume &
$A(R)=\Pi\rho(R)$ &
Global $\mathrm{O}(3)$ ambiguity &
$3$ \cite{bandeira2023estimation,fan2024maximum, balanov2026provable} \\
\hline

\textbf{Cryo-EM with translations} &
$\mathrm{SO}(3)\times\mathbb{R}^2$ &
bandlimited 3D volume &
$A(R,t)=S_t\Pi\rho(R)$ &
Global $\mathrm{E}(3)$ ambiguity &
Unknown \\
\hline

\textbf{Heterogeneous orbit recovery} &
$\{1,\ldots,K\}\times G$ &
$(\theta^\star_1,\ldots,\theta^\star_K)$ &
$A(c,g)\theta=\rho(g)\theta_c$ &
$\theta\sim(\rho(g)\theta_{\pi(1)},\ldots,\rho(g)\theta_{\pi(K)})$ &
Unknown \\
\hline
\end{tabular}

\vspace{2pt}
\parbox{\textwidth}{\scriptsize
$^\dagger$ The framework in \eqref{eq:meanscaling-model} is written without observed covariates to simplify notation. Mixtures of linear regressions fit a straightforward covariate-dependent extension in which $A(z)$ is replaced by $A(z;x)$, under standard moment and nonsingularity assumptions on the covariate distribution. }

\endgroup
\end{table*}

\paragraph{Finite mixtures.}
Let $\mathcal{Z}=\{1,\ldots,K\}$, let $\mu$ be the uniform distribution on $\mathcal{Z}$, let $\theta=(\theta_1,\ldots,\theta_K)\in (\mathbb{R}^d)^K$, and define $A(z)\theta=\theta_z$. Then, for $Z \sim \mathrm{Unif}\{1,2, \ldots, K\}$,
\begin{align}
    Y=\beta \theta_Z + \xi.
\end{align}
This is the standard equal-weight spherical Gaussian mixture model with known covariance. The latent variable selects the active component, and the parameter is identifiable only up to permutation of the component labels:
$[\theta^\star]  =  \bigl\{(\theta^\star_{\pi(1)},\ldots,\theta^\star_{\pi(K)}):\pi\in S_K\bigr\}$.

\paragraph{Multi-reference alignment.}
Let $\theta^\star\in\mathbb{R}^d$, let $\mathcal{Z}=\mathbb{Z}_d$ with $\mu$ uniform, let $\mathcal{T}_\ell$ denote the circular-shift operator on $\mathbb{R}^d$, and set $A(\ell)=\mathcal{T}_\ell$. Then, 
\begin{align}
    Y=\beta \mathcal{T}_\ell \theta^\star+ \xi.
\end{align}
This is the standard discrete MRA model, in which one observes many randomly shifted and highly noisy copies of an underlying signal \cite{bendory2017bispectrum,abbe2018multireference,bandeira2023estimation}. Since a global shift of the signal does not change the observation law, the parameter is identifiable only up to the orbit
$[\theta^\star] = \{\mathcal{T}_\ell\theta^\star:\ell\in\mathbb{Z}_d\}$.
MRA is a canonical orbit-recovery problem and a central test case for the low-SNR moment methods. 

\paragraph{Orbit recovery under compact group actions.}
Let $\mathcal{Z}=G$ be a compact group equipped with a uniform Haar measure $\mu$, let $\theta^\star\in\mathbb{R}^d$, let $\rho:G\to O(d)$ be a known orthogonal representation, and define $A(g)=\rho(g)$. Then, for $g \in G$,
\begin{align}
    Y=\beta \rho(g)\theta^\star+ \xi.
\end{align}
The signal is identifiable only up to its orbit, $[\theta^\star] = \{\rho(g)\theta^\star:g\in G\}$.
This formulation includes discrete actions, such as cyclic and dihedral shifts (which are finite subgroups of $O(d)$), as well as continuous actions, such as rotations. 

\paragraph{Projection models.}
A further important class is obtained by composing a group action with a known projection operator. Let $\mathcal{Z}=G$ be a compact group with normalized Haar measure, let $\rho:G\to O(m)$ be a known representation acting on the signal space, and let $\Pi:\mathbb{R}^m\to\mathbb{R}^d$ be a known linear projection or observation operator. Defining $A(g)=\Pi\rho(g)$ gives
\begin{align}
    Y=\beta \Pi\rho(g)\theta^\star+ \xi .
\end{align}
This model covers observation schemes in which the signal is first randomly transformed and then projected to a lower-dimensional measurement space.  A canonical example is cryo-EM without in-plane translations, where $G=\mathrm{SO}(3)$, $\theta^\star$ is a 3D volume in a finite basis, $\rho(R)$ rotates the volume, and $\Pi$ maps the rotated volume to a 2D projection image. The observational equivalence class is generally induced by global rotations and, in projection models such as cryo-EM, may also include a reflection or handedness ambiguity. For some projected orbit-recovery models, recent results show that  moments of the projected signal can be related to the corresponding full group-orbit moments under suitable structural assumptions~\cite{balanov2026group}. 

\begin{remark}
\label{remark:general_covariance_matrix}
Although we write $\xi_i\sim\mathcal{N}(0,I_d)$ for simplicity, all results extend verbatim to the case of Gaussian noise with a known, nonsingular covariance matrix $\Sigma$. Indeed, multiplying the observations by $\Sigma^{-1/2}$ and defining $\tilde{y}_i = \Sigma^{-1/2} y_i$, with $\widetilde{A}(z)=\Sigma^{-1/2}A(z)$, yields the whitened model $\tilde{y}_i=\beta\,\widetilde{A}(Z_i)\theta^\star+\,\tilde{\xi}_i$, with $\tilde{\xi}_i\sim\mathcal{N}(0,I_d)$, which is again of the form \eqref{eq:meanscaling-model}, with a modified known map $\widetilde{A}(\cdot)$.
\end{remark}

\subsection{Observational equivalence and quotient geometry}

For simplicity of exposition, we formulate the main results in the case $\dim(\Theta)=m$, but the same arguments extend, with routine modifications, to the case where $\Theta\subseteq\mathbb{R}^m$ is a lower-dimensional $C^2$ embedded submanifold; see, e.g.,~\cite{balanov2026sample}.
For a candidate parameter $\theta\in\Theta\subseteq\mathbb{R}^m$ from \eqref{eq:meanscaling-model}, the observed (marginal) density of a single sample is 
\begin{align}
    \label{eqn:obsDensityGeneral}
    p_{\theta,\beta}(y) = \int_{\mathcal{Z}}\phi \, \!\big(y-\beta A(z)\theta\big)\,\mu(dz),
\end{align}
where $\phi$ denotes the standard Gaussian density of $\mathbb{R}^d$.
The next assumption specifies the regularity conditions imposed on the map $A(\cdot)$.

\begin{assum}[Regularity and boundedness of the action $A(z)$]
\label{ass:integrability_max}
Assume that $\Theta\subset\mathbb{R}^m$ is compact. Assume further that the latent action is  bounded, i.e., there exists a constant $a_{\max}<\infty$ such that $\|A(z)\|_{\mathrm{op}} \le a_{\max}$, for a.e. $z\in\mathcal{Z}$. 
In addition, assume that the matrix
\begin{align}
    \label{eqn:I-positive}
    \mathcal{I} \;\triangleq\; \mathbb{E}_{Z\sim\mu}\big[A(Z)^\top A(Z)\big] \;\in\; \mathbb{R}^{m\times m}
\end{align}
is positive-definite, i.e.,\ $\mathcal{I}\succ 0$. 
\end{assum}

In the model classes considered in this work and in Table~\ref{tab:models}, the boundedness condition is immediate. For mixtures and finite latent spaces, $A(z)$ is a block-selection or finite collection of linear maps. For compact group-action models, $A(z)$ is  orthogonal, and in projection models $A(z)=\Pi\rho(z)$, where $\rho(z)$ is orthogonal and $\Pi$ is a fixed bounded linear operator. Hence, $\|A(z)\|_{\mathrm{op}}$ is uniformly bounded in all these examples. The condition $\mathcal{I} \succ 0$ is necessary for parameter recovery, requiring that every nonzero direction in parameter space affects the observed signal.
Since the parameter $\theta^\star$ is generally identifiable only up to a model-intrinsic observational equivalence relation, the statistical target is the equivalence class $[\theta^\star]$ rather than a specific representative, leading to the next definition.

\begin{definition}[Observational equivalence, equivalence class, and quotient space]
\label{def:equivalence_class_observational}
Fix $\beta>0$ and consider \eqref{eq:meanscaling-model}. For each $\theta\in\Theta$, let $p_{\theta,\beta}$ denote the observed-data density in~\eqref{eqn:obsDensityGeneral}. We say that $\theta,\theta'\in\Theta$ are \emph{observationally equivalent}, and write $\theta\sim\theta'$, if
\begin{align}
    \theta\sim\theta'    \quad\Longleftrightarrow\quad
    p_{\theta,\beta}(y)=p_{\theta',\beta}(y)
    \quad\text{for a.e. } y\in\mathbb{R}^d .    \label{eq:obs_equivalence_density}
\end{align}
Equivalently, this holds if and only if the induced latent means have the same distribution, that is, 
$A(Z)\theta \stackrel{d}{=} A(Z)\theta'$ for $Z\sim\mu$. The equivalence class of $\theta$ is $[\theta]\triangleq\{\theta'\in\Theta:\theta'\sim\theta\}$, and the quotient space is $\Theta/{\sim}\triangleq\{[\theta]:\theta\in\Theta\}$, with canonical projection $\pi:\Theta\to\Theta/{\sim}$, $\pi(\theta)=[\theta]$.
\end{definition}

In discrete models, such as finite mixtures, the equivalent representatives are isolated, whereas in continuous group-action models they vary along a smooth family; see Appendix~\ref{app:subsec_quotient_geometry} for representative examples of observational equivalence classes in the models considered in this work. We treat both cases through the quotient-space viewpoint. To formulate local asymptotic statements, we assume that the quotient admits a regular local manifold structure near $[\theta^\star]$~\cite{absil2008optimization,lee2003smooth}.

\begin{assum}[Local quotient manifold structure]
\label{ass:equiv_class_submanifold}
Fix $\theta^\star\in\Theta$. Assume that there exists an open neighborhood $U_\star\subset\Theta$ of $\theta^\star$ such that the quotient $U_\star/{\sim}$ admits a $C^2$ manifold structure for which the canonical projection $\pi:U_\star\to U_\star/{\sim}$, $\pi(\theta)=[\theta]$, is a $C^2$ submersion. 
\end{assum}

Under Assumption~\ref{ass:equiv_class_submanifold}, the equivalence class $[\theta^\star]$ is a $C^2$ embedded submanifold of $\Theta$ near $\theta^\star$, and its tangent space at $\theta^\star$ is therefore well defined. This tangent space describes the infinitesimal directions that remain within the same equivalence class and are thus locally non-identifiable. Its Euclidean orthogonal complement captures the identifiable local directions transverse to the equivalence class.
\begin{definition}[Tangent and normal space]
\label{def:tangent_normal_pi}
Fix $\theta^\star\in\Theta$ and assume Assumption~\ref{ass:equiv_class_submanifold}. Define
\begin{align}
    T(\theta^\star) \triangleq T_{\theta^\star}[\theta^\star] = \Bigl\{ \dot{\gamma}(0): \exists \varepsilon>0,\  \gamma:(-\varepsilon,\varepsilon)\to[\theta^\star]\text{ is }C^1,\ \gamma(0)=\theta^\star \Bigr\}.
    \label{eq:tangent_space_def}
\end{align}
Let
\begin{align}
    W \triangleq T(\theta^\star)^\perp = \{w\in\mathbb{R}^m:w^\top t=0 \ \text{for all } t\in T(\theta^\star)\}, \label{eq:direct_sum_T_W}
\end{align}
where orthogonality is with respect to the standard Euclidean inner product on $\mathbb{R}^m$. Then,  $\mathbb{R}^m = T(\theta^\star)\oplus W$.
We denote by $\Pi_W:\mathbb{R}^m\to W$ the Euclidean orthogonal projector onto $W$.
\end{definition}

The subspace $W$ is the \emph{normal space} at $\theta^\star$. It provides the local linear model for the quotient near $[\theta^\star]$: any sufficiently small perturbation of $\theta^\star$ decomposes into a tangential component in $T(\theta^\star)$, which changes only the representative inside the same class, and a normal component in $W$, which captures the identifiable local displacement. For this reason, all local asymptotic statements below are formulated on  the quotient $\Theta/{\sim}$. In Appendix~\ref{app:subsec_quotient_geometry}, we explain why Assumption~\ref{ass:equiv_class_submanifold} holds for the representative models considered in this work, and characterize the associated tangent and normal spaces for both discrete models and continuous group-action models.

\subsection{Maximum-likelihood estimation and Fisher information}
\label{subsec:likelihood}

Let $\mathcal{Y} \triangleq \{y_i\}_{i=1}^n$ denote the observed dataset, where
$y_1,\ldots,y_n$ are independent samples drawn according to model
\eqref{eq:meanscaling-model}. For $\theta\in\Theta$, the observed log-likelihood is
\begin{align}
    \label{eqn:logLikelihoodGeneral_efficiency}   \mathcal{L}_n(\theta;\mathcal{Y}) \triangleq \frac{1}{n}\sum_{i=1}^n \log p_{\theta,\beta}(y_i),
\end{align}
where $p_{\theta,\beta}$ is the observed-data density under the Gaussian latent-variable model~\eqref{eq:meanscaling-model}.
The MLE is the maximizer of the observed log-likelihood,
\begin{align}
    \label{eqn:mleGeneral_efficiency}    \widehat{\theta}_{\mathrm{MLE}}(\mathcal{Y}) \in \argmax_{\theta\in\Theta}\mathcal{L}_n(\theta;\mathcal{Y}).
\end{align}

The statistical efficiency of the MLE is governed by the \emph{observed Fisher information}, namely, the Fisher information of the marginal density $p_{\theta,\beta}$~\eqref{eqn:obsDensityGeneral} obtained after integrating out the latent variable.

\begin{definition}[Observed Fisher information]
\label{def:I_obs_meanscaling}
For fixed $\beta$, the observed Fisher information matrix at $\theta\in\Theta$ is defined by
\begin{align}
    \mathcal{I}_{\mathrm{obs}}(\theta;\beta) \triangleq  \mathbb{E}_{Y\sim p_{\theta,\beta}}\!\left[\nabla_\theta \log p_{\theta,\beta}(Y)\, \nabla_\theta \log p_{\theta,\beta}(Y)^\top \right] \in \mathbb{R}^{m\times m},
    \label{eq:def_I_obs}
\end{align}
where $Y\sim p_{\theta,\beta}$ is distributed according to the observed model~\eqref{eqn:obsDensityGeneral}.
\end{definition}

Since the parameter is identifiable only up to observational equivalence, we restrict the Fisher information to the normal space $W$ at $\theta^\star$. Specifically, we define the self-adjoint operator $\mathcal{I}_{\mathrm{obs}}^{(W)}(\theta^\star;\beta):W\to W$ by
\begin{align}
    \label{eq:restricted_obs_info_positive_mle}
    \big\langle h,\mathcal{I}_{\mathrm{obs}}^{(W)}(\theta^\star;\beta)g\big\rangle = h^\top \mathcal{I}_{\mathrm{obs}}(\theta^\star;\beta)g,
    \qquad h,g\in W.
\end{align}
In Appendix~\ref{subsec:cons_efficiency_MLE}, we establish that, under suitable regularity and identifiability conditions, the MLE is consistent, asymptotically normal, and statistically efficient on the quotient space. In particular, for each fixed $\beta>0$,
\begin{align}
    \sqrt{n}\,\Pi_W\!\bigl(\widehat{\theta}_{\mathrm{MLE}}-\theta^\star\bigr) \xrightarrow[]{\mathcal{D}} \mathcal{N}\!\bigl(0,\mathcal{I}_{\mathrm{obs}}^{(W)}(\theta^\star;\beta)^{-1}\bigr).
\end{align}

\subsection{Local identifiability and informative layers}
\label{subsec:residual_subspaces_layers}

We now refine the local picture by organizing the identifiable normal space according to the lowest moment order at which a direction becomes visible. We first introduce the moment quantities that organize the low-SNR analysis throughout the paper.

For each $k\ge1$, let $N_k\triangleq\binom{d+k-1}{k}$ denote the dimension of the space of order-$k$ symmetric tensors on $\mathbb{R}^d$. Throughout, we identify this space with $\mathbb{R}^{N_k}$ through an orthonormal basis, so that the Frobenius inner product corresponds to the standard Euclidean inner product.

\begin{definition}[Signal population moments]
\label{def:signal_population_moment_tensor}
For each $k\ge 1$, we define the order-$k$ signal moment tensor by
\begin{align}
    \label{eq:def_Tk_signal}
    T_k(\theta) \triangleq \mathbb{E}_{Z\sim\mu}\!\big[(A(Z)\theta)^{\otimes k}\big] \in \mathbb{R}^{N_k}.
\end{align}
We also define the \emph{stacked signal moment map up to order $k$} by
\begin{align}
    \label{eq:def_T_le_k}
    T_{\le k}(\theta) \triangleq    \bigl(T_1(\theta),\ldots,T_k(\theta)\bigr)
    \in\mathbb{R}^{N_{\le k}},
\end{align}
where $N_{\le k} \triangleq \sum_{j=1}^k N_j = \binom{d+k}{k}-1$. 
\end{definition}

Asymptotic efficiency is governed by the first order at which the derivative of the moment map at the ground-truth, $DT_r(\theta^\star)$, detects all infinitesimal identifiable perturbations, namely, by injectivity of the differential moment map on the normal space at the ground truth.

\begin{definition}[Local identifiability order]
\label{def:rloc}
Fix $\theta^\star\in\Theta$, and for each $h\in W\setminus\{0\}$ define the first order at which the differential moment map detects $h$,
\begin{align}
    r(h)\triangleq \min\{j\ge 1:\ DT_j(\theta^\star)[h]\neq 0\},  \label{eqn:def_r_h}
\end{align}
with the convention $r(h)=\infty$ if no such $j$ exists. The \emph{local identifiability order} at $\theta^\star$ is \begin{align}
    \label{eq:rloc_def}
    r_{\mathrm{loc}}(\theta^\star) \triangleq \max_{h\in W\setminus\{0\}} r(h) \in \{1,2,\ldots\}\cup\{\infty\}.
\end{align}
\end{definition}

We now decompose $W$ according to the moment order at which its directions first become detectable.

\begin{definition}[Moment filtration and informative layers]
\label{def:residual_subspaces_layers}
Define the moment-filtration subspaces by
\begin{align}
    \label{eq:Vk_def}
    V_1 &\triangleq W,
    \\
    V_k &\triangleq    \Big(\bigcap_{j<k}\ker\bigl(DT_j(\theta^\star)\bigr)\Big)\cap W \;=\; \{h\in W:\ r(h)\ge k\},
    \qquad k\ge 2.
\end{align}
Each $V_k$ is a linear subspace of $W$, since it is an intersection of linear subspaces.
The $k$-th informative layer is defined by
\begin{align}
    \label{eq:Uk_def}
    U_k \triangleq V_k\cap V_{k+1}^{\perp},
\end{align}
where orthogonality is taken with respect to the standard Euclidean inner product on $\mathbb{R}^m$. In particular, each $U_k$ is also a linear subspace, being the intersection of two linear subspaces.
\end{definition}

The family $\{V_k\}_{k\ge 1}$ forms a decreasing filtration $W = V_1 \supseteq V_2 \supseteq \cdots$, where $V_k$ consists of directions that are invisible to all differential moments of order strictly less than $k$. Thus, the quotient $V_k/V_{k+1}$ captures the genuinely new directions that become detectable at order $k$, and $U_k$ is a concrete orthogonal representative of this quotient layer.

By construction, $V_k = U_k \oplus V_{k+1}$, and hence every $h\in V_k$ admits a unique decomposition into a component in $U_k$ and a component in $V_{k+1}$.
In particular, when $r_{\mathrm{loc}} \triangleq  r_{\mathrm{loc}}(\theta^\star) < \infty$, the filtration terminates at $V_{r_{\mathrm{loc}}+1}=\{0\}$, and iterating the above decomposition yields
\begin{align}
    \label{eq:W_layer_decomposition} W=\bigoplus_{k=1}^{r_{\mathrm{loc}}} U_k.
\end{align}
This decomposition organizes the identifiable parameter directions by their first informative moment order, which will be used in Proposition~\ref{prop:fisher_layer_block_asymptotics} to establish the connection between local identifiability and statistical estimation difficulty.

\paragraph{Example: a two-component Gaussian mixture.}
To illustrate the informative-layers decomposition, we use a simple example of a symmetric two-component Gaussian mixture
\begin{align}
    Y \sim \frac{1}{2}\mathcal{N}(\beta(m+\mu),1) + \frac{1}{2}\mathcal{N}(\beta(m-\mu),1),
    \label{eq:example_gmm_model}
\end{align}
where $\theta=(m,\mu)\in\mathbb{R}^2$, with $\mu\neq 0$, and $\beta>0$ is the known signal-strength parameter. The parameter is identifiable up to the exchange of the two mixture components. The first two signal moments are $T_1(m,\mu)=m$, and $ T_2(m,\mu)=m^2+\mu^2$. The first moment determines the mixture center $m$, while the second moment  determines $|\mu|$. The Jacobian of the stacked moment map is
\begin{align}
    DT_{\leq 2}(m,\mu)=
    \begin{pmatrix}
    1&0\\
    2m&2\mu
    \end{pmatrix}.
    \label{eq:example_gmm_jacobian}
\end{align}
By definition, $V_1=W=\mathbb{R}^2$. Since $DT_1(\theta^\star)=
\begin{pmatrix}1&0\end{pmatrix}$, we obtain $V_2=\ker DT_1(\theta^\star) =\operatorname{span}\{(0,1)\}$.
Furthermore, for $\mu^\star\neq0$, the stacked Jacobian $DT_{\leq2}(\theta^\star)$ has full rank, so $V_3=\ker DT_{\leq2}(\theta^\star)=\{0\}$.
Consequently, the informative layers are
\begin{align}
    U_1&=V_1\cap V_2^\perp     =\operatorname{span}\{(1,0)\},
    \label{eq:example_gmm_U1}\\
    U_2&=V_2\cap V_3^\perp =\operatorname{span}\{(0,1)\}.
    \label{eq:example_gmm_U2}
\end{align}
Thus, $W=U_1\oplus U_2$, $\dim(U_1)=\dim(U_2)=1$, and $r_{\mathrm{loc}}=2$.

The two layers have a direct statistical interpretation. The direction $U_1$ corresponds to changes in the mixture center $m$, which are detected by the first moment. In contrast, the first moment is insensitive to changes in the separation $\mu$, so directions in $U_2$ become detectable only through the second moment. Accordingly, in the weak-signal regime, the Fisher information scales as $\mathrm{SNR}$ along $U_1$ and as $\mathrm{SNR}^2$ along $U_2$. In Appendix~\ref{subsec:layer_decompoisition_MRA}, we further illustrate the informative-layer decomposition for cyclic MRA. In this model, $U_1$ corresponds to the mean component, detected by the first moment; $U_2$ corresponds to Fourier-magnitude directions, detected by the power spectrum; and $U_3$ corresponds to Fourier-phase directions, which first become detectable through the bispectrum~\cite{balanov2026expectation}.

\subsection{Layerwise low-SNR asymptotics of the Fisher information}
\label{subsec:main-results-fisher-layers}

We now connect the moment-layer decomposition of the normal space $W$ to the local likelihood geometry. The proof is provided in Appendix~\ref{app:proof_fisher_layer_block_asymptotics}.

\begin{proposition}[Layerwise block asymptotics of the Fisher information]
\label{prop:fisher_layer_block_asymptotics}
For every $1\leq k,\ell\leq r_{\mathrm{loc}}$, every $h\in U_k$, and every $g \in U_\ell$,
\begin{align}
    \label{eq:block_fisher_layered}
    \big\langle h,\mathcal{I}_{\mathrm{obs}}^{(W)}(\theta^\star;\beta)g\big\rangle =
    \begin{cases}
        \displaystyle
        \frac{\mathrm{SNR}^{k}}{k!}\,\big\langle DT_k(\theta^\star)[h],DT_k(\theta^\star)[g]\big\rangle_F + O\bigl(\mathrm{SNR}^{k+1}\bigr),
        & k=\ell,\\[2ex]       O\bigl(\mathrm{SNR}^{(k+\ell+1)/2}\bigr),
        & |k-\ell|=1,\\[1.25ex]  O\bigl(\mathrm{SNR}^{(k+\ell+2)/2}\bigr),
        & |k-\ell|\geq 2.
    \end{cases}
\end{align}
as $\mathrm{SNR}\to 0$.
\end{proposition}

Proposition~\ref{prop:fisher_layer_block_asymptotics} shows that the observed Fisher information is asymptotically block diagonal in low SNR, with the $k$-th diagonal block first appearing at order $\mathrm{SNR}^k$. Its leading term is governed by the first nonvanishing differential signal moment on $U_k$. Thus, the moment order that detects a local direction also determines the SNR scale of its Fisher information. 
In particular, for $h\in W$, the first nonzero contribution to the quadratic form of the observed Fisher information appears at the first moment order at which the differential moment map detects $h$, namely, at order $r(h)$. The next corollary, which is proved in Appendix~\ref{app:proof_dir_exponents_Iobs}, states this observation and records the corresponding consequence for the smallest eigenvalue of the restricted Fisher information.

\begin{corollary}
[Directional scaling and smallest-eigenvalue of the observed Fisher information]
\label{cor:dir_exponents_Iobs}
Let $r(h)$ be defined as in~\eqref{eqn:def_r_h}. Then, for every $h\in W\setminus\{0\}$ with $r(h)<\infty$,
\begin{align}
    \label{eq:lmin_Iobs_scale_main}
    h^\top \mathcal{I}_{\mathrm{obs}}^{(W)}(\theta^\star;\beta)\,h = \frac{\|DT_{r(h)}(\theta^\star)[h]\|_F^2}{r(h)!}\, \mathrm{SNR}^{\,r(h)} + O\bigl(\mathrm{SNR}^{\,r(h)+1}\bigr),
\end{align}
as $\mathrm{SNR}\to 0$.
Consequently, if $r_{\mathrm{loc}}(\theta^\star)<\infty$, then $\lambda_{\min}\!\big( \mathcal{I}_{\mathrm{obs}}^{(W)}(\theta^\star;\beta) \big) = \Theta\bigl(\mathrm{SNR}^{\,r_{\mathrm{loc}}(\theta^\star)}\bigr)$.
\end{corollary}

\section{The statistical efficiency of the generalized method of moments}
\label{sec:GMoM_efficiency}
In this section, we establish the main result of the paper: in the low-SNR regime, the observed Fisher information and the GMoM information matrix agree to leading order. 

\subsection{The generalized method of moments estimator}
\label{subsec:empirical_features_main}

We first recall the standard GMoM formulation. Let $g(Y,\theta)$ be a vector of moment conditions satisfying $\mathbb{E}_{Y\sim p_{\theta^\star , \beta}}[g(Y,\theta^\star)]=0$. Given i.i.d. observations $\mathcal{Y}=\{y_i\}_{i=1}^n$, the GMoM estimator minimizes the weighted quadratic criterion
\begin{align}
    \widehat\theta \in \argmin_{\theta\in\Theta} \left( \frac{1}{n}\sum_{i=1}^n g(y_i,\theta) \right)^\top \Omega_n \left(  \frac{1}{n}\sum_{i=1}^n g(y_i,\theta) \right), \label{eq:standard_GMoM_estimator}
\end{align}
where $\Omega_n$ is a positive semidefinite weighting matrix. 

Let $\Omega_n\xrightarrow{\mathbb P}\Omega\succ0$, and write $G=\mathbb{E}[\nabla_\theta g(Y,\theta^\star)]$ and $S=\mathrm{Cov}(g(Y,\theta^\star))$. Under regularity conditions, the corresponding GMoM estimator has asymptotic covariance~\cite{newey1994large}
\begin{align}
    \Sigma_{\mathrm{GMoM}}^{(\Omega)} = (G^\top\Omega G)^{-1} G^\top\Omega S\Omega G     (G^\top\Omega G)^{-1}.    \label{eq:general_GMoM_covariance}
\end{align}
The optimal weighting within the fixed class of moment conditions is $\Omega_n\xrightarrow{\mathbb{P}}S^{-1}$, and with this choice, the limiting asymptotic covariance is (see, for example,~\cite{hansen1982large,newey1994large})
\begin{align}
    \Sigma_{\mathrm{GMoM}}^{(\mathrm{opt})} \triangleq \bigl(G^\top S^{-1}G\bigr)^{-1}.\label{eqn:limiting_covariance}
\end{align}

We now specialize this construction to the moment features used in this paper. For each $j\ge 1$, we define the empirical raw moment tensor
\begin{align}
    \widehat{R}_{j,n}(\mathcal{Y}) \triangleq \frac{1}{n}\sum_{i=1}^n y_i^{\otimes j}  \in\mathbb{R}^{N_j},
\end{align}
and its population counterpart
$R_j(\theta;\beta) \triangleq \mathbb E_{Y\sim p_{\theta,\beta}} [Y^{\otimes j}] \in\mathbb{R}^{N_j}$, where $Y$ denotes an observation from \eqref{eq:meanscaling-model}, and symmetric tensors are represented in the orthonormal coordinates introduced in Definition~\ref{def:signal_population_moment_tensor}.

Raw empirical moments contain both the latent signal and additive Gaussian noise. Since the noise law is known, the Gaussian terms can be removed explicitly. We therefore construct, for each order $j$, a noise-debiased empirical moment $\widehat{T}_{j,n}(\mathcal{Y})$ whose population target is the signal moment $T_j(\theta)=\mathbb{E}_{Z \sim \mu}[(A(Z)\theta)^{\otimes j}]$, namely $\mathbb{E}[\widehat{T}_{j,n}(\mathcal{Y})]=T_j(\theta)$. A convenient way to implement this debiasing is through multivariate Hermite tensors, which form the natural Gaussian-orthogonal feature system for the low-SNR expansion~\cite{geng2020wiener}. The detailed construction and the relevant Hermite identities are deferred to Appendix~\ref{subsec:empirical_moments}.

We collect these debiased moments into a single feature vector. Let $\psi(y)$ denote the the stacked noise-debiased moment feature vector of a single observation $y$ up to order $L \triangleq L_{\mathrm{mom}}$, after a low-SNR normalization (see~\eqref{def:stacked_Hermite_feature_map} for a formal definition). With this notation, the empirical feature vector is
\begin{align}
    M_n(\mathcal{Y}) \triangleq \frac{1}{n}\sum_{i=1}^n \psi(y_i) =
    \begin{bmatrix}
        \beta\,\widehat{T}_{1,n}(\mathcal{Y})\\ \beta^2\,\widehat{T}_{2,n}(\mathcal{Y})\\
        \vdots\\    \beta^L\,\widehat{T}_{L,n}(\mathcal{Y})
    \end{bmatrix},  \label{eq:def_Mn_maintext}
\end{align}
and its population counterpart is
\begin{align}
    M(\theta;\beta) \triangleq \mathbb{E}_{Y\sim p_{\theta,\beta}}[\psi(Y)] =
    \begin{bmatrix}
        \beta\,T_1(\theta)\\    \beta^2\,T_2(\theta)\\
        \vdots\\        \beta^L\,T_L(\theta)
    \end{bmatrix}.    \label{eq:def_Mtheta_maintext}
\end{align}

Defining $g(Y,\theta)=\psi(Y)-M(\theta;\beta)$, we have $\mathbb{E}_{Y \sim p_{\theta^\star, \beta}}[g(Y,\theta^\star)]=0$, and the GMoM criterion becomes
\begin{align}
    Q_n(\theta;\mathcal{Y}) \triangleq \bigl(M_n(\mathcal{Y})-M(\theta;\beta)\bigr)^\top \Omega_n \bigl(M_n(\mathcal{Y})-M(\theta;\beta)\bigr). \label{eq:def_GMoM_criterion_main}
\end{align}
The corresponding estimator is
\begin{align}
    \widehat{\theta}_{\mathrm{GMoM}}(\mathcal{Y}) \in \argmin_{\theta\in\Theta} Q_n(\theta;\mathcal{Y}). \label{eq:def_GMoM_estimator_main}
\end{align}

\subsection{Efficiency of the optimal weighting matrix}
\label{subsec:GMoM_asymptotic_efficiency}

To obtain the main result, we specialize the general optimal GMoM covariance formula~\eqref{eqn:limiting_covariance} into our model. Recall that the moment conditions are given by $g(Y,\theta)=\psi(Y)-M(\theta;\beta)$. Accordingly, the covariance matrix in the general formulation becomes
\begin{align}
    \Sigma_L(\theta^\star;\beta)
    &\triangleq
    \operatorname{Cov}_{Y\sim p_{\theta^\star,\beta}} \bigl(\psi(Y)\bigr)
    \notag\\ &=
    \mathbb{E}_{Y\sim p_{\theta^\star,\beta}} \Big[ \bigl(\psi(Y)-M(\theta^\star;\beta)\bigr) \bigl(\psi(Y)-M(\theta^\star;\beta)\bigr)^\top \Big],  \label{eqn:def_Sigma_L}
\end{align}
and the optimal weighting matrix in our setting is $\Omega_n\xrightarrow{\mathbb P}\Sigma_L(\theta^\star;\beta)^{-1}$.
Similarly, since local estimation is performed on the normal space $W$ at $\theta^\star$, the derivative $G$ in the general formulation~\eqref{eqn:limiting_covariance}, restricted to $W$, becomes $-D_WM(\theta^\star;\beta)$, where
\begin{align}
    \label{eqn:def_Jacobian_moments}
    D_WM(\theta^\star;\beta) \triangleq DM(\theta^\star;\beta)\big|_W.
\end{align}
Consequently, when $D_WM(\theta^\star;\beta)$ has full column rank on $W$, the asymptotic covariance of optimally weighted GMoM is
\begin{align}
    \label{eqn:asymptotic_covariance_GMoM}
    \Sigma_{\mathrm{GMoM}}^{(\mathrm{opt})}(\theta^\star;\beta) \triangleq \Big(D_WM(\theta^\star;\beta)^\top \Sigma_L(\theta^\star;\beta)^{-1} D_WM(\theta^\star;\beta) \Big)^{-1}.
\end{align}
We define the corresponding GMoM information matrix on $W$ by
\begin{align}
    \mathcal{I}_{\mathrm{GMoM}}^{(W)}(\theta^\star;\beta)
    &\triangleq    D_WM(\theta^\star;\beta)^\top \Sigma_L(\theta^\star;\beta)^{-1} D_WM(\theta^\star;\beta). \label{eq:def_GMoM_information}
\end{align}
This is the GMoM analogue of the restricted observed Fisher information $\mathcal{I}_{\mathrm{obs}}^{(W)}(\theta^\star;\beta)$ defined in Section~\ref{subsec:likelihood}.

The following theorem is the main result of the paper. It establishes that the information captured by optimally weighted GMoM agrees with the observed Fisher information to leading order in low SNR. The proof is provided in Appendix~\ref{app:proof_H_fixed_beta_fisher_mo}.

\begin{thm}[Low-SNR Fisher-GMoM information discrepancy and relative efficiency]
\label{thm:H_fixed_beta_fisher_moment_identity}
Let $L\ge1$. Then, for $\beta>0$ sufficiently small, the following hold.

\begin{enumerate}
    \item \emph{(Information discrepancy).}
    The observed Fisher information and the GMoM information satisfy $\mathcal{I}_{\mathrm{obs}}^{(W)}(\theta^\star;\beta) = \mathcal{I}_{\mathrm{GMoM}}^{(W)}(\theta^\star;\beta) +R_\beta$, where $R_\beta:W\to W$ is self-adjoint, positive semidefinite, and satisfies $\|R_\beta\|_{\mathrm{op}} =O\bigl(\mathrm{SNR}^{L+1}\bigr)$.

    \item \emph{(Relative inverse-information equivalence).}
    If, in addition, $L\ge r_{\mathrm{loc}}(\theta^\star)$, then
    \begin{align}
        &\bigl(\mathcal{I}_{\mathrm{obs}}^{(W)}        (\theta^\star;\beta)\bigr)^{1/2} \bigl(\mathcal{I}_{\mathrm{GMoM}}^{(W)}        (\theta^\star;\beta)\bigr)^{-1} \bigl(\mathcal{I}_{\mathrm{obs}}^{(W)}        (\theta^\star;\beta)\bigr)^{1/2} =I_W+\widetilde R_\beta,
    \end{align}
    where $\widetilde R_\beta:W\to W$ is self-adjoint and satisfies $\|\widetilde{R}_\beta\|_{\mathrm{op}} = O\bigl( \mathrm{SNR}^{L+1-r_{\mathrm{loc}}(\theta^\star)} \bigr) =o(1)$.
\end{enumerate}
\end{thm}

Figure~\ref{fig:1} illustrates the predicted information discrepancy across three representative Gaussian latent-variable models, showing that each additional moment order improves the decay rate by one power of SNR.
Combined with Proposition~\ref{prop:fisher_layer_block_asymptotics}, Theorem~\ref{thm:H_fixed_beta_fisher_moment_identity} implies that the GMoM information matrix has the same leading layerwise structure as the observed Fisher information~\eqref{eq:block_fisher_layered}.

\begin{figure*}[t!]
    \centering
    \includegraphics[width=1.0 \linewidth]{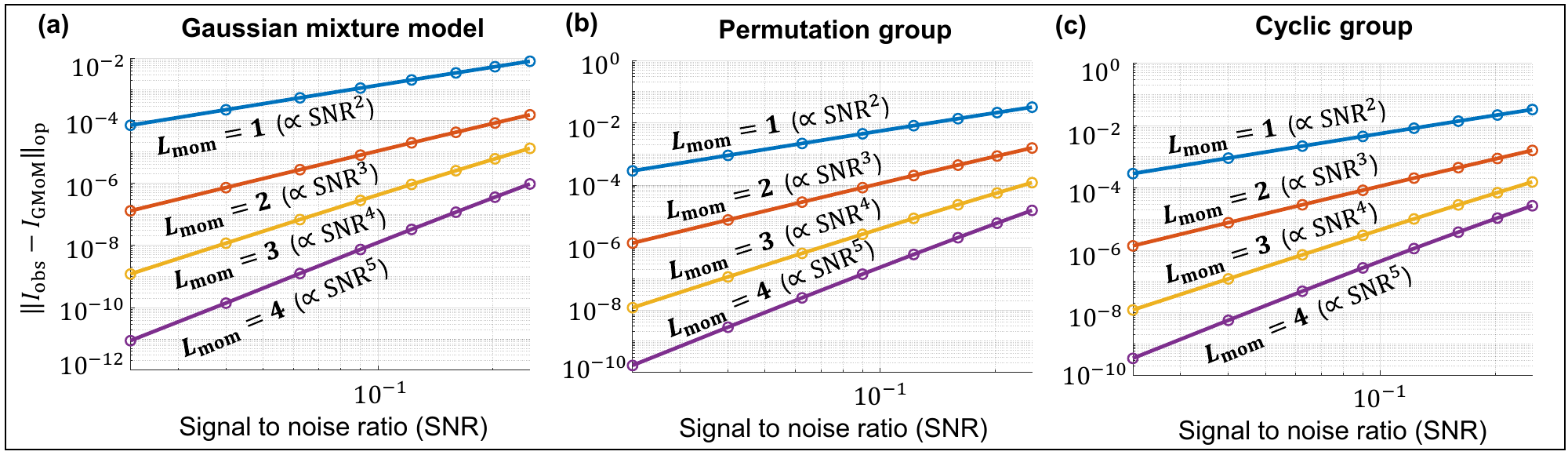}
    \caption{\textbf{Low-SNR Fisher--GMoM discrepancy across representative Gaussian latent-variable models.}
    Operator-norm discrepancy $\|\mathcal{I}_{\mathrm{obs}}^{(W)}(\theta^\star;\beta)-\mathcal{I}_{\mathrm{GMoM}}^{(W)}(\theta^\star;\beta)\|_{\mathrm{op}}$ as a function of SNR for (a) an equal-weight Gaussian mixture with $K=3$ components in dimension $d=3$~\cite{hsu2013learning,anandkumar2014tensor}, (b) the permutation-group model with $d=3$ and six permutations~\cite{bandeira2023estimation}, and (c) cyclic MRA with signal length $d=3$~\cite{bendory2017bispectrum}, using moments up to order $L_{\mathrm{mom}}=1,2,3,4$. In all three models, the discrepancy follows the predicted $O(\mathrm{SNR}^{L_{\mathrm{mom}}+1})$ scaling. Population quantities were evaluated by averaging over the latent states and using Gauss--Hermite quadrature with 11 nodes per dimension for the Gaussian mixture and cyclic MRA, and 9 nodes for the permutation model. Moment Jacobians were computed using centered finite differences with step size $10^{-7}$.
    }
    \label{fig:1} 
\end{figure*}

\subsection{Construction of a data-driven weighting matrix}
\label{subsec:weighting_matrix}

Theorem~\ref{thm:H_fixed_beta_fisher_moment_identity} establishes the low-SNR efficiency of GMoM under the population-optimal weighting matrix $\Sigma_L(\theta^\star;\beta)^{-1}$. Since this covariance depends on the unknown parameter, we may construct a feasible estimator by estimating it directly from the observed feature vectors $\psi(y_i)$, without requiring preliminary estimation of $\theta^\star$:
\begin{align}
    \widehat{\Sigma}_{L,n}(\mathcal{Y}) \triangleq \frac{1}{n}\sum_{i=1}^n \bigl(\psi(y_i)-M_n(\mathcal{Y})\bigr) \bigl(\psi(y_i)-M_n(\mathcal{Y})\bigr)^\top. \label{eqn:def_sample_covariance}
\end{align}

For each fixed $\beta>0$, the empirical covariance $\widehat{\Sigma}_{L,n}(\mathcal{Y})$ converges to $\Sigma_L(\theta^\star;\beta)$ as $n\to\infty$.
Moreover, Lemma~\ref{lem:sigma_beta_positive_definite_small_beta} establishes that the population covariance is positive definite for sufficiently small $\beta$. Nevertheless, the empirical covariance may be singular or poorly conditioned at finite sample sizes. To ensure invertibility and numerical stability, we introduce a ridge parameter $\lambda_n>0$ and define the feasible weighting matrix as
\begin{align}
    \Omega_n \triangleq \left( \widehat{\Sigma}_{L,n}(\mathcal{Y}) +\lambda_n I \right)^{-1}.
    \label{eq:def_omega_n}
\end{align}
The following result, which is proved in Appendix~\ref{app:proof_lem_weighting_matrix}, establishes that, for an appropriate choice of $\lambda_n$, this weighting matrix consistently estimates the population-optimal weighting matrix and characterizes its convergence rate.

\begin{lem}[Consistency of empirical GMoM weighting]
\label{lem:weighting_matrix}
Assume Assumption~\ref{ass:integrability_max} holds. For fixed $L$ and $d$, let $\lambda_n>0$ satisfy $\lambda_n\to 0$. Then, uniformly for sufficiently small $\beta$,
\begin{align}
    \left\| \Omega_n -\Sigma_L(\theta^\star;\beta)^{-1} \right\|_{\mathrm{op}}
    &=O_{\mathbb P}(n^{-1/2}+\lambda_n), \label{eq:empirical_weighting_rate}
\end{align}
In particular, choosing $\lambda_n=c/\sqrt n$ for any fixed $c>0$ yields,
\begin{align}
    \left\| \Omega_n-\Sigma_L(\theta^\star;\beta)^{-1} \right\|_{\mathrm{op}}= O_{\mathbb P}(n^{-1/2}).
\end{align}
\end{lem}

The ridge parameter also guarantees invertibility in high-dimensional settings. Indeed, the dimension of the stacked moment feature vector is $N_{\leq L}=\binom{d+L}{L}-1$, while the rank of the empirical covariance $\widehat{\Sigma}_{L,n}$ is at most $n-1$. Consequently, the unregularized covariance is necessarily singular whenever $n\leq N_{\leq L}$. Choosing $\lambda_n=c/\sqrt{n}$ with $c>0$ ensures invertibility for every sample size, while preserving the $O_{\mathbb{P}}(n^{-1/2})$ convergence rate established for fixed $d$ and $L$. The resulting feasible GMoM estimator is summarized in Algorithm~\ref{alg:feasible_GMoM}.

\begin{algorithm}[t]
\caption{Feasible low-SNR GMoM}
\label{alg:feasible_GMoM}
\textbf{Input:} Observations $y_1,\ldots,y_n$, signal level $\beta$, moment order $L=L_{\mathrm{mom}}\ge r_{\mathrm{loc}}$, and ridge constant $c>0$.

\textbf{Output:} Feasible GMoM estimator $\widehat{\theta}_{\mathrm{GMoM}}$.

\begin{enumerate}
    \item Compute the empirical feature vector    $M_n(\mathcal{Y})=\frac1n\sum_{i=1}^n\psi(y_i)$.
    \item Estimate the feature covariance $\widehat{\Sigma}_{L,n}(\mathcal{Y})$
    using~\eqref{eqn:def_sample_covariance}.
    \item Set $\lambda_n=c/\sqrt n$ and compute
    $\Omega_n=(\widehat{\Sigma}_{L,n}+\lambda_n I)^{-1}$.
    \item Compute $\widehat{\theta}_{\mathrm{GMoM}}$
    according to~\eqref{eq:def_GMoM_criterion_main}--\eqref{eq:def_GMoM_estimator_main}.
\end{enumerate}
\end{algorithm}

\subsection{Efficiency loss under nonoptimal positive-definite weighting}
\label{subsec:nonoptimal_weighting}

The preceding results establish the low-SNR efficiency of optimally weighted GMoM. We now examine the efficiency loss induced by nonoptimal weighting matrices. Recall the general GMoM covariance formula~\eqref{eq:general_GMoM_covariance}. Let $\Omega_n\xrightarrow{\mathbb P}\Omega=\Omega(\beta)\succ0$ for each fixed $\beta>0$, and write $D_WM=D_WM(\theta^\star;\beta)$ and $\Sigma_L=\Sigma_L(\theta^\star;\beta)$. Assuming that $D_WM$ has full column rank on $W$, the asymptotic covariance of~\eqref{eq:general_GMoM_covariance} becomes
\begin{align}
    \Sigma_{\mathrm{GMoM}}^{(\Omega)}
    &=
    (D_WM^\top\Omega D_WM)^{-1}
    D_WM^\top\Omega\Sigma_L\Omega D_WM
    (D_WM^\top\Omega D_WM)^{-1}.
    \label{eq:nonoptimal_GMoM_covariance}
\end{align}
The optimal choice $\Omega_{\mathrm{opt}}=\Sigma_L^{-1}$ recovers $\Sigma_{\mathrm{GMoM}}^{(\mathrm{opt})}$ from~\eqref{eqn:asymptotic_covariance_GMoM}. By the classical GMoM theory~\cite{hansen1982large,newey1994large}, any positive-definite weighting matrix satisfies $\Sigma_{\mathrm{GMoM}}^{(\Omega)}\succeq\Sigma_{\mathrm{GMoM}}^{(\mathrm{opt})}$. Consequently, the covariance loss relative to MLE decomposes as
\begin{align}
    \Sigma_{\mathrm{GMoM}}^{(\Omega)}-\mathcal{I}_\beta^{-1} =  \underbrace{\Sigma_{\mathrm{GMoM}}^{(\mathrm{opt})}-\mathcal{I}_\beta^{-1}}_{\text{moment-truncation loss}}    +\underbrace{\Delta_\Omega}_{\text{weighting loss}},    \label{eq:nonoptimal_total_loss}
\end{align}
where $\Delta_\Omega\triangleq \Sigma_{\mathrm{GMoM}}^{(\Omega)}-\Sigma_{\mathrm{GMoM}} ^{(\mathrm{opt})} \succeq 0$, and both terms are positive semidefinite. Thus, nonoptimal weighting introduces an additional efficiency loss beyond that of moment truncation. We next examine how the weighting loss $\Delta_\Omega$ behaves in the low-SNR regime for different choices of $\Omega$. 

\paragraph{Weighting matrices preserving informative subspaces.}
A natural way to preserve leading efficiency is to choose weights that approach the optimal inverse covariance as SNR decreases. However, this is not necessary: even fixed, nonoptimal weighting matrices can retain leading efficiency if they preserve the informative moment subspaces. The following proposition, which is proved in Appendix~\ref{subsec:proof_of_block_isotropic_weighting}, establishes this property.

\begin{proposition}[Low-SNR efficiency under invariant-subspace weighting]
\label{prop:block_isotropic_weighting}
Let $L\ge r_{\mathrm{loc}}(\theta^\star)$, and consider a fixed positive-definite weighting matrix that is block diagonal across moment orders, $\Omega=\operatorname{diag}(\Omega_1,\ldots,\Omega_L)$, with $\Omega_k\succ0$. Suppose that each informative moment subspace is invariant under its corresponding weighting block,
\begin{align}
    \Omega_k\operatorname{Im}\bigl(DT_k(\theta^\star)|_{U_k}\bigr) \subseteq \operatorname{Im}\bigl(DT_k(\theta^\star)|_{U_k}\bigr).    \label{eq:weight_invariance_condition}
\end{align}
for $ 1\le k\le r_{\mathrm{loc}}(\theta^\star)$. Then, under the same assumptions of Theorem~\ref{thm:H_fixed_beta_fisher_moment_identity},
\begin{align}
    0\le\mathcal{E}_{\mathrm{tr}} (\Sigma_{\mathrm{GMoM}}^{(\Omega)};\beta) \le\mathcal{E}_{\mathrm{op}}    (\Sigma_{\mathrm{GMoM}}^{(\Omega)};\beta) =O(\mathrm{SNR}).    \label{eq:block_isotropic_efficiency}
\end{align}
In particular, this holds for every block-isotropic weighting matrix $\Omega=\operatorname{diag}(c_1I_{N_1},\ldots,c_LI_{N_L})$ with $c_k>0$.
\end{proposition}

Two particularly simple block-isotropic choices are identity weighting, $\Omega_I=I$, and inverse zero-signal covariance weighting, $\Omega_0=(\Sigma_L^{(0)})^{-1}$. Importantly, even identity-weighted GMoM, corresponding to the common unweighted method of moments, achieves the same leading asymptotic efficiency as MLE in the low-SNR limit, with a relative efficiency loss of at most $O(\mathrm{SNR})$. 
As established in Proposition~\ref{prop:sigma0_closed_form_updated}, the inverse zero-signal covariance weighting has the explicit form
\begin{align}
    \Omega_0 =\operatorname{diag}\left( I_{N_1}, \frac{1}{2!}I_{N_2}, \ldots, \frac{1}{L!}I_{N_L} \right). \label{eq:noise_only_block_weights}
\end{align}
This weighting coincides with the moment-order weighting proposed in~\cite{zhang2025diagonally}, which has demonstrated practical benefits in numerical experiments on Gaussian mixture estimation.

\paragraph{Anisotropic weighting and persistent efficiency loss.}
The invariant-subspace condition~\eqref{eq:weight_invariance_condition} is essential: positive-definite weighting does not necessarily preserve leading statistical efficiency. In particular, when a weighting block mixes informative moment directions with their orthogonal complement, the relative efficiency loss may remain strictly positive as SNR vanishes, even when all moments are retained and the parameter remains identifiable.

Figure~\ref{fig:3}(a) illustrates this phenomenon for a three-component Gaussian mixture in $\mathbb{R}^2$ with equal mixing weights, using moments through order $L=3$. We consider the anisotropic weighting matrix $\Omega_A = \operatorname{diag}\left( I_2,\, I_3,\, \operatorname{diag}(1,10,1,1) \right)$.
In this example, the third informative layer is one-dimensional, and the limiting worst-direction efficiency loss is
\begin{align}
    \lim_{\beta\to 0} \mathcal{E}_{\mathrm{op}} \bigl(\Sigma_{\mathrm{GMoM}}^{(\Omega_A)};\beta\bigr) =\frac{2673}{1681}>0. \label{eq:anisotropic_GMM_efficiency_limit}
\end{align}
The derivation and further details are provided in
Appendix~\ref{app:anisotropic_GMM_example}.
Thus, anisotropic weighting incurs a persistent variance inflation relative to MLE, while preserving the optimal low-SNR sample-complexity exponent.

\begin{figure*}[t!]
    \centering
    \includegraphics[width=\linewidth]{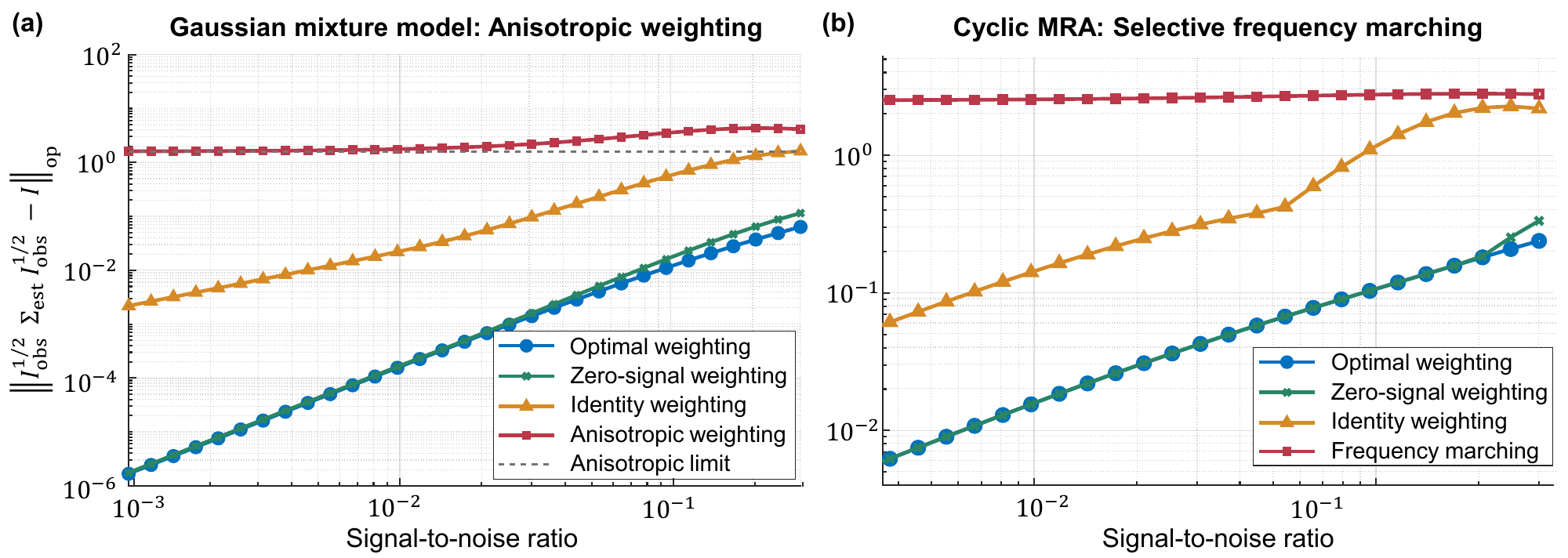}
    \caption{\textbf{Worst-direction relative covariance loss compared with MLE}, $\mathcal{E}_{\mathrm{op}} (\Sigma_{\mathrm{est}};\beta)$ ~\eqref{eq:relative_efficiency_op}, as a function of SNR. \textbf{(a)} Three-component Gaussian mixture model in $\mathbb{R}^2$, with equal mixing weights and component means $\mu_1=(-1,0)$, $\mu_2=(0,1)$, and $\mu_3=(1,-1)$. All Hermite moments through order $L=3$ are retained. The anisotropic positive-definite weighting is $\Omega_A=\operatorname{diag}(I_2,I_3,\operatorname{diag}(1,10,1,1))$ in orthonormal symmetric-tensor coordinates. The dashed horizontal line indicates the analytical low-SNR limit given in~\eqref{eq:anisotropic_GMM_efficiency_limit}. \textbf{(b)} Uniform cyclic MRA with $d=7$ and signal $\theta^\star=(\sqrt{6},-1/\sqrt{6},\ldots,-1/\sqrt{6})$. Full GMoM uses all Hermite moments through order $L=3$, whereas frequency marching uses the Fourier mean, three power-spectrum entries, and three selected bispectrum entries. The observed Fisher information is computed using a 32-point Gauss-Hermite rule per coordinate in (a) and scrambled Sobol quadrature with $2^{14}$ points in (b).}
    \label{fig:3}
\end{figure*}

\subsection{Efficiency loss under selective moment weighting}
\label{subsec:selective_weighting}

A prominent example of constructive moment recovery is frequency marching, which sequentially reconstructs the unknown signal using selected moment relations. Such methods have been developed for cyclic MRA~\cite{perry2019sample,bendory2017bispectrum}, dihedral MRA~\cite{edidin2026reflection}, orbit recovery over $\mathrm{SO}(3)$~\cite{bendory2025orbit}, and more recently for projected MRA~\cite{balanov2026projected} and cryo-EM~\cite{balanov2026provable}, providing provable population-level recovery guarantees. Although these methods suffice for signal recovery, they need not retain all the statistical information available in the full moment tensors.

To quantify explicitly the efficiency loss induced by selective moment recovery, we analyze a simplified  model, which permits an explicit characterization of the information discarded by frequency marching. We consider the uniform cyclic MRA model in odd dimensions $d=2m+1\ge 7$, with a centered flat-spectrum signal. For $k\in\mathbb{Z}_d$, let
\begin{align}
    \widehat{\theta}_k = \frac{1}{\sqrt{d}} \sum_{j=0}^{d-1}\theta_j e^{-2\pi i jk/d}
\end{align}
denote the discrete Fourier coefficients. We assume $\widehat{\theta}_0^\star=0$ and $|\widehat{\theta}_k^\star|=a>0$ for all $k\neq0$. Since $\theta^\star\in\mathbb{R}^d$, the conjugate symmetry $\widehat{\theta}_{d-k}^\star=\overline{\widehat{\theta}_k^\star}$ leaves $m$ independent Fourier phases. 

The second moment determines the Fourier magnitudes, while the third-order bispectrum, defined by $B(k,\ell)=\widehat{\theta}_k\widehat{\theta}_\ell\overline{\widehat{\theta}_{k+\ell}}$, provides relations among these phases. Frequency marching recovers them sequentially using the selected entries $\{B(1,k)\}_{k=1}^{m-1} \cup\{B(m,m)\}$~\cite{bendory2017bispectrum}. Together with the first and second moments, these entries suffice for recovery, although they use only a subset of the available third-order information.

In Appendix~\ref{app:FM_efficiency_derivation}, we show that for centered flat-spectrum signals in odd dimensions $d\ge7$, the low-SNR efficiency losses of frequency marching satisfy
\begin{align}
    \lim_{\beta\to 0}    \mathcal{E}_{\mathrm{op}}(\Sigma_{\mathrm{FM}};\beta)
    &\sim\frac{d^3}{8\pi^2},    \label{eq:FM_general_op_gap}\\
    \lim_{\beta\to 0}    \mathcal{E}_{\mathrm{tr}}(\Sigma_{\mathrm{FM}};\beta)
    &=\frac{(d-5)(d-3)(d+8)}{24(d-1)}
    \sim\frac{d^2}{24}.
    \label{eq:FM_general_trace_gap}
\end{align}
Figure~\ref{fig:5} illustrates  this persistent efficiency gap for $d=7,9,11$. The efficiency penalty increases substantially with signal dimension, scaling as $O(d^2)$ in total MSE and $O(d^3)$ in the worst direction. This growth reflects not only the increasing number of discarded bispectrum relations but also the deteriorating conditioning of the phase-recovery system induced by frequency marching.

\begin{figure*}[t!]
    \centering
    \includegraphics[width=\linewidth]{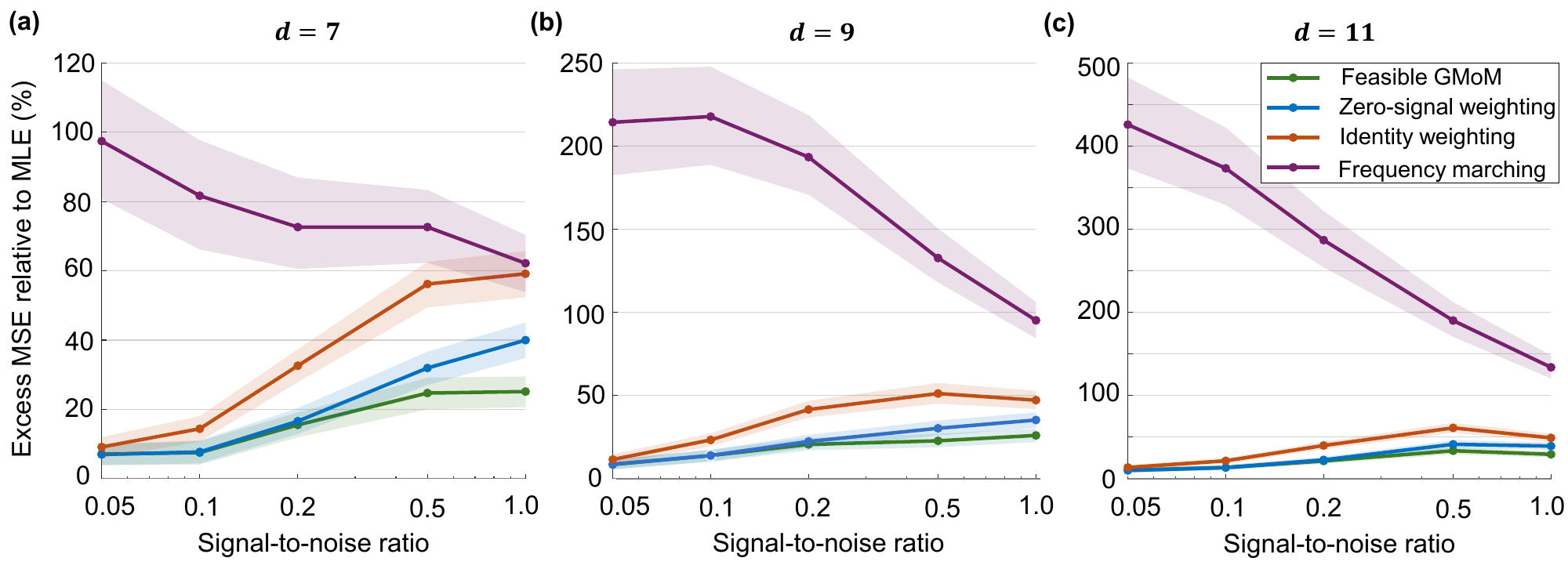}
    \caption{\textbf{Efficiency loss under different GMoM weighting schemes in cyclic MRA.}
    Relative excess MSE compared with MLE, $\mathrm{MSE}_{\mathrm{GMoM}}/\mathrm{MSE}_{\mathrm{MLE}}-1$, as a function of SNR for uniform cyclic MRA with signal lengths (a) $d=7$, (b) $d=9$, and (c) $d=11$. We compare feasible GMoM (Algorithm~\ref{alg:feasible_GMoM}),
    inverse zero-signal covariance weighting, identity weighting, and frequency marching based on selected bispectrum entries. All estimators are evaluated after optimal cyclic-shift alignment, using $n=100{,}000$ observations and 400 independent Monte Carlo trials per configuration. Shaded regions indicate pointwise 95\% paired-bootstrap confidence intervals. While full-moment GMoM exhibits decreasing relative efficiency losses as SNR vanishes, frequency marching retains a substantial, dimension-dependent penalty, consistent with the theoretical low-SNR predictions~\eqref{eq:FM_general_trace_gap}.}
    
    \label{fig:5}
\end{figure*}

\section{Discussion and outlook}
\label{sec:discussion}

The main statistical conclusion of this work is that, in the low-SNR regime, GMoM can attain the same leading asymptotic efficiency as MLE, providing an  answer to the fundamental question of when moment-based estimation can match the statistical efficiency of the full likelihood. In particular, when the moment cutoff captures all locally informative layers, optimally weighted GMoM achieves the same leading layerwise covariance as MLE. Moreover, this efficiency equivalence extends to simpler block-isotropic weighting schemes, including identity and inverse zero-signal covariance weighting, whereas anisotropic or selective weighting may incur a persistent efficiency loss even when the parameter remains identifiable. These findings highlight the importance of both moment selection and weighting in statistical efficiency. They are particularly relevant to low-SNR imaging problems such as cryo-EM and cryo-ET, where moment-based methods offer computationally attractive alternatives to likelihood-based estimation. Our conclusions concern local asymptotic efficiency and do not imply equivalence of the global optimization landscapes or computational guarantees of the two approaches.

We next outline some open questions that remain for future work.

\paragraph{Joint sample-size, SNR, and high-dimensional regimes.}
An important direction for future work is to characterize the statistical efficiency of GMoM in the joint regime where $n\to\infty$ and $\mathrm{SNR}_n\to 0$ simultaneously. Our theoretical results consider sequential limits: asymptotic normality is established for fixed SNR as $n\to\infty$, followed by an analysis of the resulting information operators as SNR vanishes. 

To investigate the joint regime numerically, Figure~\ref{fig:4} compares the empirical covariance of MLE and optimally weighted GMoM with their respective fixed-SNR asymptotic covariance predictions in the cyclic MRA case.
For each estimator, the covariance discrepancy is measured by
\begin{align}
    D(n,\mathrm{SNR}) = \left\| V^{-1/2} \left( n\widehat{\operatorname{Cov}}(\widehat\theta)-V \right) V^{-1/2} \right\|_{\mathrm{op}},    \label{eqn:covaraince_disp_def}
\end{align}
where $\widehat{\theta}$ denotes the MLE or GMoM estimator, $V$ denotes the corresponding asymptotic covariance matrix, given by $V=\mathcal{I}_{\mathrm{obs}}^{-1}$ for MLE and $V=\mathcal{I}_{\mathrm{GMoM}}^{-1}$ for GMoM, and ~$\widehat{\operatorname{Cov}}$ denotes the empirical covariance estimate across Monte Carlo trials. This normalization measures the worst-direction relative deviation from the predicted asymptotic covariance.
Both estimators exhibit similar transitions toward their asymptotic predictions, with an empirical boundary approximately following $n\propto\mathrm{SNR}^{-3}$. These numerical results suggest that the asymptotic covariance predictions derived for fixed SNR remain accurate when SNR decreases with the sample size, provided the sample size is sufficiently large. However, establishing joint asymptotic normality, and identifying the precise sample-size conditions for efficiency equivalence remain open questions. A further direction is to extend the analysis to high-dimensional settings, where the parameter dimension grows with the sample size~\cite{romanov2021multi,dou2024rates}. 

\begin{figure*}[t!]
    \centering
    \includegraphics[width=0.9\linewidth]{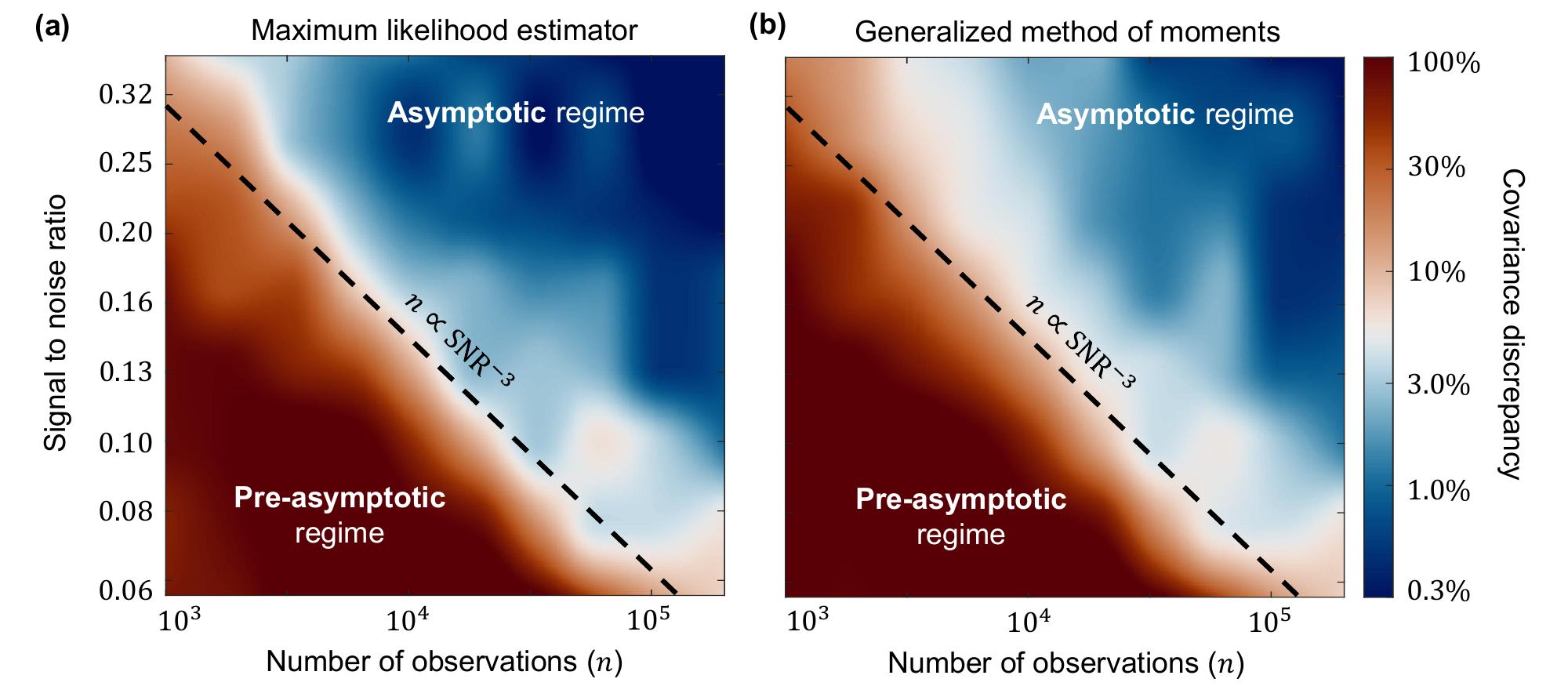}
    \caption{ \textbf{Finite-sample agreement with asymptotic covariance predictions in the joint sample-size and SNR regime for uniform cyclic MRA.} The heatmaps show the relative discrepancy, as defined in~\eqref{eqn:covaraince_disp_def}, between the empirical covariance and the corresponding asymptotic covariance for \textbf{(a)} the MLE and \textbf{(b)} the optimally weighted GMoM. Empirical covariances are estimated using 1,000 independent Monte Carlo trials per configuration. The dashed black line indicates the fitted scaling $n \propto \mathrm{SNR}^{-3}$. Both estimators exhibit similar transitions toward accurate asymptotic covariance approximations as the sample size increases. } 
    \label{fig:4}
\end{figure*}

\paragraph{Statistical efficiency of practical recovery algorithms.}
Another important direction is to characterize the statistical efficiency of practical recovery algorithms. While constructive methods such as frequency marching provide provable recovery guarantees, our results demonstrate that they may incur a persistent efficiency loss relative to MLE. 
Other approaches, including tensor-based reconstruction~\cite{anandkumar2014tensor,hsu2013learning}, semidefinite relaxations~\cite{bendory2017bispectrum}, and fast low-SNR MLE~\cite{kreymer2026fast}, offer computationally efficient recovery methods under suitable assumptions, while their asymptotic efficiency relative to MLE remains less understood.
Thus, an important goal is to develop algorithms that combine provable recovery guarantees with asymptotic statistical efficiency. 
One promising approach is to construct hybrid estimators that use a consistent reconstruction method for initialization, followed by local refinement based on MLE or optimally weighted GMoM. 

\paragraph{Toward a general theory of statistical efficiency.}
A natural direction is to extend the present statistical efficiency theory beyond Gaussian mean-scaling models. Recent work~\cite{balanov2026sample} develops a general framework for weak-signal latent-variable models, characterizing sample complexity through likelihood-ratio expansions and encompassing nonlinear latent signals, non-Gaussian additive-noise channels, and exponential-family observation models. An important open question is whether the layered Fisher-information structure, and the leading efficiency equivalence between GMoM and MLE extend to these broader settings, reflecting fundamental properties of statistical estimation in the weak-signal regime rather than features specific to Gaussian models.

\section*{Acknowledgment}

T.B. is supported in part by NSF-BSF under Grant 2024791, and in part by ISF under Grant 1046/26. D.E. was supported by NSF DMS2205626. 

\bibliographystyle{plain}

\begin{appendices}

{\centering{\section*{Appendix}}}

\paragraph{Appendix organization.}
Appendix~\ref{app:prelim} provides additional preliminaries on quotient geometry and global identifiability. Appendix~\ref{app:fundamental_properties} establishes fundamental properties of the Gaussian latent-variable model, including likelihood regularity and MLE asymptotics. Appendix~\ref{sec:hermite_low_snr_expansions} develops the Hermite expansions of the likelihood  underlying the low-SNR Fisher-information analysis. Appendix~\ref{app:gmom_efficiency} establishes the consistency, asymptotic normality, and statistical efficiency of optimally weighted GMoM.
Appendix~\ref{app:eff_loss} analyzes the efficiency loss under nonoptimal weighting, including anisotropic weighting and selective frequency marching. Finally, Appendix~\ref{app:numerical_details} provides additional details on the numerical experiments.

\section{Preliminaries}
\label{app:prelim}

\subsection{Quotient geometry }
\label{app:subsec_quotient_geometry}

We provide additional details on the quotient-space formulation introduced in Section~\ref{sec:preliminaries}.

\paragraph{On the relation between the latent space and the equivalence class.}
By Definition~\ref{def:equivalence_class_observational}, two parameters $\theta \sim \theta'$ are equivalent if the random latent means $A(Z)\theta \stackrel{d}{=} A(Z)\theta'$ have the same distribution. Thus, the equivalence class $[\theta^\star]$ is determined by the distribution of the random measurement $A(Z)\theta^\star$, rather than by the set of latent variables $\mathcal{Z}$ alone.
In many models, the observational equivalence relation coincides with the symmetry generated by the latent action. For example, in cyclic  MRA, the latent variable is a cyclic shift and the parameter is identifiable up to a global cyclic shift. However, this coincidence does not generally hold. For example, in projected models, the projection may  introduce additional symmetries. In projected MRA, the
latent variable may consist only of cyclic shifts, while the projection makes the observation
law invariant also under reflection; consequently, the identifiable object is the dihedral
orbit rather than only the cyclic-shift orbit~\cite{bandeira2023estimation,balanov2026projected}. Similarly, in cryo-EM, the latent viewing direction lies in $\mathrm{SO}(3)$, but the projection model
introduces the usual handedness ambiguity, so the observational equivalence class contains both rotated and reflected copies of the underlying volume~\cite{bendory2020single}. Thus, $[\theta^\star]$ should always be understood as a property of the marginal model, rather than as a symmetry determined solely by the latent space.

\paragraph{Quotient geometry in representative examples.}
Next, we show that in all the examples considered in this work, the quotient-space formulation is natural and the local manifold structure required in Assumption~\ref{ass:equiv_class_submanifold} holds:
\begin{enumerate}
    \item In finite-symmetry models, such as $K$-component Gaussian mixtures or cyclic/dihedral MRA, the equivalence class $[\theta^\star]$ is a finite set and hence a $0$-dimensional embedded smooth submanifold of $\Theta$. In this case the tangent space is trivial, $T(\theta^\star)=\{0\}$, and therefore the normal space is the whole ambient space, $W=\mathbb{R}^m$.

    \item In continuous group-action models, such as orbit recovery under $\mathrm{SO}(n)$, the equivalence class is the group orbit $[\theta^\star]=\{\rho(g)\theta^\star:g\in G\}$. This orbit is an embedded smooth submanifold and is diffeomorphic to the homogeneous space $G/\mathrm{Stab}(\theta^\star)$, where $\mathrm{Stab}(\theta^\star)\triangleq \{g\in G:\rho(g)\theta^\star=\theta^\star\}$ is the stabilizer of $\theta^\star$.
    Its tangent space is the tangent space to the orbit, $T(\theta^\star)=T_{\theta^\star}(G\cdot\theta^\star)$, namely, the span of the infinitesimal generators of the group action at $\theta^\star$, and the normal space
    $W=T(\theta^\star)^\perp$ collects the locally identifiable directions transverse to the orbit. 

    For example, in orbit recovery under $\mathrm{SO}(3)$, for a generic signal one has $\mathrm{Stab}(\theta^\star)=\{I\}$, so the orbit is $3$-dimensional and $\dim T(\theta^\star)=3$, with $\dim W=m-3$. More generally, $\dim T(\theta^\star)=3-\dim \mathrm{Stab}(\theta^\star)$, so nongeneric signals with rotational symmetries give lower-dimensional orbits.

    \item Heterogeneous models combine a finite symmetry with a compact Lie-group action; accordingly, $[\theta^\star]$ is a finite union of embedded smooth group orbits, so locally near $\theta^\star$ it is again an embedded smooth submanifold with well-defined tangent and normal spaces. 

    \item In projected models, such as projected MRA~\cite{bandeira2023estimation,balanov2026projected} and cryo-EM, the observational equivalence class may be strictly larger than the latent $\mathrm{SO}(3)$ orbit, since the projection model also introduces a handedness ambiguity, as described above.  Nevertheless, locally near a fixed representative $\theta^\star$, the  component of $[\theta^\star]$ containing $\theta^\star$ is an embedded smooth submanifold, so the tangent space $T(\theta^\star)$ and normal space $W=T(\theta^\star)^\perp$ are well defined.

\end{enumerate}

\paragraph{Loss modulo observational equivalence.}
To establish global consistency of the estimators, we need a notion of distance that accounts for observational equivalence, since the parameter is identifiable only up to its equivalence class. Accordingly, for every $\theta_0,\theta_1\in\Theta$, we define
\begin{align}
    \label{eq:deq_pairwise}
    d_{\mathrm{eq}}(\theta_0,\theta_1) \triangleq \inf_{\theta_0'\in[\theta_0],\,\theta_1'\in[\theta_1]} \|\theta_0'-\theta_1'\|.
\end{align}
This quantity measures the Euclidean separation between the equivalence classes of $\theta_0$ and $\theta_1$. It satisfies $d_{\mathrm{eq}}(\theta_0,\theta_1)=0$ whenever $\theta_0\sim\theta_1$. If the equivalence classes are closed in $\Theta$, as in the finite- and compact-group action examples, then the converse also holds: $d_{\mathrm{eq}}(\theta_0,\theta_1)=0$ implies $\theta_0\sim\theta_1$. In most of the examples considered below, the equivalence relation is induced by a group $G$ acting on $\Theta$ by isometries, so that $[\theta]=\{g\cdot\theta:g\in G\}$. In this case, $d_{\mathrm{eq}}$ coincides with the standard orbit metric on the orbit space and can be written as
\begin{align}
    d_{\mathrm{eq}}(\theta_0,\theta_1) = \inf_{g\in G} \|\theta_0-g\cdot\theta_1\|.
\end{align}

\subsection{Global identifiability order}
\label{subsec:global_identifiability}

Moments play two related roles in latent-variable models. They govern global distinguishability, determining when the population moments uniquely identify the underlying parameter up to observational equivalence, and local sensitivity, determining when infinitesimal identifiable perturbations become detectable. These notions underlie previous results on moment identifiability, low-SNR sample complexity, and Fisher-information geometry~\cite{bandeira2023estimation,fan2024maximum, balanov2026sample}. We briefly recall the distinction to clarify its role in our statistical-efficiency analysis.

\begin{definition}[Generic signal]
\label{def:generic_signal}
A property $P(\theta)$ is said to hold for a \emph{generic signal} in $\Theta$ if there exists an exceptional set $\Theta_{\mathrm{exc}}\subseteq\Theta$ of measure zero such that $P(\theta)$ holds for every $\theta\in\Theta\setminus\Theta_{\mathrm{exc}}$. 
\end{definition}

The exceptional set may depend on the property under consideration. For example, constructive recovery methods may require nonvanishing Fourier coefficients or full-rank harmonic coupling matrices~\cite{bendory2017bispectrum,bendory2025orbit}. Such assumptions should be distinguished from the generic regularity conditions needed for local identifiability.

\begin{definition}[Generic global identifiability order]
\label{def:rglob}
We say that the model is \emph{globally moment-identifiable of order $k$ for a generic signal} if there exists an exceptional set $\Theta_{\mathrm{exc}}\subseteq\Theta$ of measure zero such that, for every $\theta^\star\in\Theta\setminus\Theta_{\mathrm{exc}}$,
\begin{align}
    \label{eq:global_ident_k}
    T_{\le k}(\theta)=T_{\le k}(\theta^\star)
    \Longrightarrow
    \theta\sim\theta^\star,
    \qquad \forall\theta\in\Theta.
\end{align}
The corresponding generic global identifiability order is
\begin{align}
    \label{eq:rglob_def}
    r_{\mathrm{glob}}
    \triangleq
    \min\bigl\{
    k\ge 1:
    \eqref{eq:global_ident_k}
    \text{ holds for generic }\theta^\star
    \bigr\}
    \in\{1,2,\ldots\}\cup\{\infty\}.
\end{align}
\end{definition}

The global order $r_{\mathrm{glob}}$ concerns the values of the moment maps, whereas the local order $r_{\mathrm{loc}}(\theta^\star)$, as defined in Section~\ref{subsec:residual_subspaces_layers}, concerns their differentials on the identifiable normal space $W$. The relationship between these notions was previously studied in~\cite{fan2024maximum} and~\cite{balanov2026sample}. In particular, since unique global identification implies finite-list identification, it follows that
\begin{align}
    \label{eq:generic_equality}
    r_{\mathrm{loc}}(\theta^\star)
    \leq r_{\mathrm{glob}}
\end{align}
generically in these models. Analogous generic-rank reasoning applies to broader polynomial moment models under suitable regularity assumptions~\cite{balanov2026sample}.

The inequality can be strict, as also illustrated in~\cite[Example~2.14]{fan2024maximum}. The following simple example demonstrates this possibility within our Gaussian latent-variable framework.

\begin{example}[Local identifiability can precede global identifiability] Consider a one-dimensional parameter space $\Theta=\mathbb{R}$, with trivial quotient, and let $Z$ take three values such that
\begin{align}
    A(Z)=
    \begin{cases}
    1, & \text{with probability } \frac{1}{3},\\
    2, & \text{with probability } \frac{1}{3},\\
    -3 & \text{with probability }  \frac{1}{3}.
    \end{cases}
\end{align}
Then, a direct calculation shows that $\mathbb{E}[A(Z)] = 0$, so $T_1(\theta)=0$.
On the other hand, $T_2(\theta) = \frac{14}{3}\theta^2$, and $T_3(\theta) = -6\theta^3$.
Now, $T_2(\theta)$ does not globally identify $\theta$, since $T_2(\theta)=T_2(-\theta)$.
However, for  generic $\theta^\star\neq 0$,
\begin{align}
    DT_2(\theta^\star)[h]=\frac{28}{3}\theta^\star h,
\end{align}
so the second-order moment map is already locally injective there. Thus, $r_{\mathrm{loc}}(\theta^\star)=2$,  for generic  $\theta^\star\neq 0$, whereas $r_{\mathrm{glob}}=3$.

\end{example}

The distinction between local and global identifiability is relevant to the interpretation of our results: global moment identifiability concerns recovery of the correct equivalence class, whereas local moment identifiability determines which perturbations can be estimated at each low-SNR information scale. Our efficiency analysis concerns the latter, conditional on consistent localization, and characterizes when GMoM attains the leading asymptotic covariance of MLE.

\subsection{Informative-layer decomposition in cyclic MRA}
\label{subsec:layer_decompoisition_MRA}

We illustrate the informative-layer decomposition for the cyclic MRA model with $\mathbb{Z}_3$ acting on $\mathbb{R}^3$ by circular shifts. Since the orbit of a generic signal is finite, the equivalence class is discrete, and therefore the tangent space is trivial, and the normal space satisfies $W=T(\theta^\star)^\perp=\mathbb{R}^3$.
For generic signals in this model one has $r_{\mathrm{glob}}=3$~\cite{bendory2017bispectrum}. Up to order three, the cyclic-invariant signal moments can be represented, after removing redundant entries, by the map
\begin{align}
    T_{\leq 3}(x) = \Big(&x_1+x_2+x_3,\;
    x_1^2+x_2^2+x_3^2,\;
    x_1x_2+x_2x_3+x_3x_1, \notag\\
    &x_1^3+x_2^3+x_3^3,\;
    x_1^2x_2+x_2^2x_3+x_3^2x_1,\;
    x_1x_2^2+x_2x_3^2+x_3x_1^2
    \Big).
\label{eq:example_Z3_moments}
\end{align}
Its Jacobian is
\begin{align}
    D T_{\leq 3}(x)=
    \begin{bmatrix}
    1 & 1 & 1\\
    2x_1 & 2x_2 & 2x_3\\
    x_2+x_3 & x_1+x_3 & x_1+x_2\\
    3x_1^2 & 3x_2^2 & 3x_3^2\\
    2x_1x_2+x_3^2 & x_1^2+2x_2x_3 & x_2^2+2x_3x_1\\
    x_2^2+2x_1x_3 & 2x_1x_2+x_3^2 & x_1^2+2x_2x_3
    \end{bmatrix}.   \label{eq:example_Z3_jacobian}
\end{align}
We evaluate this Jacobian at $\theta^\star=(1,2,-4)$, which gives
\begin{align}
    D T_{\leq 3}(\theta^\star)=
    \begin{pmatrix}
    1 & 1 & 1\\
    2 & 4 & -8\\
    -2 & -3 & 3\\
    3 & 12 & 48\\
    20 & -15 & -4\\
    -4 & 20 & -15
    \end{pmatrix}.    \label{eq:example_Z3_jacobian_theta}
\end{align}
We now compute the moment filtration. By definition, $V_1=W=\mathbb{R}^3$. Next,
\begin{align}
    V_2 = \ker DT_1(\theta^\star) = \ker
    \begin{pmatrix}
    1 & 1 & 1
    \end{pmatrix}
    =
    \langle (1,1,1)\rangle^\perp.
    \label{eq:example_Z3_V2}
\end{align}
Similarly,
\begin{align}
    V_3 &=  \ker D(T_{\leq 2})(\theta^\star) \notag\\
    &=
    \ker
    \begin{pmatrix}
    1 & 1 & 1\\
    2 & 4 & -8\\
    -2 & -3 & 3
    \end{pmatrix}
    =
    \operatorname{span}\{(-6,5,1)\}.
    \label{eq:example_Z3_V3}
\end{align}
Finally, $V_4 = \ker D(T_{\leq 3})(\theta^\star) = \{0\}$, since the full stacked Jacobian in \eqref{eq:example_Z3_jacobian_theta} has rank $3$. Therefore, the informative layers are
\begin{align}
    U_1 &= V_1\cap V_2^\perp = \operatorname{span}\{(1,1,1)\},
    \label{eq:example_Z3_U1}\\
    U_2 &= V_2\cap V_3^\perp = \operatorname{span}\{(4,7,-11)\},
    \label{eq:example_Z3_U2}\\
    U_3 &= V_3\cap V_4^\perp = V_3 = \operatorname{span}\{(-6,5,1)\}.
\label{eq:example_Z3_U3}
\end{align}
Hence
\begin{align}
    W = U_1\oplus U_2\oplus U_3, \qquad \dim(U_1)=\dim(U_2)=\dim(U_3)=1. \label{eq:example_Z3_decomposition}
\end{align}
The three layers correspond to directions first detected by the first-, second-, and third-order invariant moments, respectively. In cyclic MRA, these are the mean, power-spectrum, and bispectrum levels: $U_1$ is visible from the mean, $U_2$ from the power spectrum after removing the mean direction, and $U_3$ only from bispectrum information~\cite{balanov2026expectation}.

\section{Fundamental properties of the Gaussian latent-variable model}
\label{app:fundamental_properties}

This appendix collects foundational material for the Gaussian latent-variable model (\eqref{eq:meanscaling-model}) used throughout the paper. 
We define $\phi(x)$ to be the standard Gaussian density on $\mathbb{R}^d$,
\begin{align}    
    \phi(x) \triangleq (2\pi)^{-d/2} \exp\!\left(-\frac{\|x\|_2^2}{2}\right).   \label{eq:def_standard_gaussian_density_phi}
\end{align}

\subsection{Regularity of the Fisher-information identities}
\label{subsec:Fisher_regular_meanscaling}

The next lemma collects the regularity properties of the marginal likelihood needed for the Fisher-information analysis below. These facts provide the analytic foundation for the low-SNR expansions derived in the following subsections.

\begin{lem}[Analytic regularity and Fisher-information identities]
\label{lem:likelihood_fisher_regular}
Consider \eqref{eq:meanscaling-model} and fix $\beta>0$, and $\theta^\star\in\Theta$. Assume Assumption~\ref{ass:integrability_max} holds. Let $U\subset\Theta$ be a neighborhood of $\theta^\star$. Then, the following hold.

\begin{enumerate}
    \item For every $y\in\mathbb{R}^d$, the map $\theta\mapsto p_{\theta,\beta}(y) = \int_{\mathcal{Z}}\phi\!\big(y-\beta A(z)\theta\big)\,\mu(dz)$ is $C^\infty$ on $U$.

    \item There exists a measurable function $G_U:\mathbb{R}^d\to[0,\infty)$ such that $ \mathbb{E}_{Y\sim p_{\theta^\star,\beta}}[G_U(Y)]<\infty$ and, for all $y\in\mathbb{R}^d$,
    \begin{align}
        \sup_{\theta\in U}|\log p_{\theta,\beta}(y)| + \sup_{\theta\in U}\|\nabla_\theta\log p_{\theta,\beta}(y)\| + \sup_{\theta\in U}\|\nabla_\theta^2\log p_{\theta,\beta}(y)\|_{\mathrm{op}} \le G_U(y).
    \end{align}

    \item The observed-data Fisher information matrix satisfies
    \begin{align}     \mathcal{I}_{\mathrm{obs}}(\theta;\beta)
        &=         \mathbb{E}_{Y\sim p_{\theta,\beta}}\!\big[\nabla_\theta \log p_{\theta,\beta}(Y)\,        \nabla_\theta \log p_{\theta,\beta}(Y)^\top \big]
        \\ & = - \mathbb{E}_{Y\sim p_{\theta,\beta}}\!\big[\nabla_\theta^2 \log p_{\theta,\beta}(Y)\big].
    \end{align}
\end{enumerate}
\end{lem}

\begin{proof}[Proof of Lemma~\ref{lem:likelihood_fisher_regular}]
By Assumption~\ref{ass:integrability_max}, $\Theta$ is compact and $\|A(z)\|_{\mathrm{op}}\le a_{\max}$ for a.e. $z\in\mathcal{Z}$. Let $B_\Theta\triangleq \sup_{\theta\in\Theta}\|\theta\|<\infty$ and $R\triangleq \beta a_{\max}B_\Theta$. Then $\|\beta A(z)\theta\|\le R$ for all $\theta\in\Theta$ and a.e. $z\in\mathcal{Z}$, so $\beta A(z)\theta$ lie in the fixed compact ball $\mathcal{B}(0,R)\subset\mathbb{R}^d$.

For fixed $y\in\mathbb{R}^d$ and $z\in\mathcal{Z}$, define $g(\theta;y,z)\triangleq \phi\!\big(y-\beta A(z)\theta\big)$.
Since $g$ is a Gaussian density composed with an affine map, it is $C^\infty$ in $\theta$. Its first two derivatives are
\begin{align}
    \nabla_\theta g(\theta;y,z)
    &=
    \beta A(z)^\top\!\big(y-\beta A(z)\theta\big)\,g(\theta;y,z),    \label{eq:grad_g_likelihood_regular}
    \\
    \nabla_\theta^2 g(\theta;y,z)
    &=  \beta^2  \left[  A(z)^\top\!\big(y-\beta A(z)\theta\big) \big(y-\beta A(z)\theta\big)^\top A(z) -  A(z)^\top A(z) \right] g(\theta;y,z).    \label{eq:hess_g_likelihood_regular}
\end{align}
Hence, uniformly in $\theta\in\Theta$,
\begin{align}
    \|\nabla_\theta g(\theta;y,z)\|
    &\le C_1(1+\|y\|)\,g(\theta;y,z),  \label{eq:grad_g_bound_likelihood_regular}
    \\
    \|\nabla_\theta^2 g(\theta;y,z)\|_{\mathrm{op}}
    &\le C_2(1+\|y\|^2)\,g(\theta;y,z),\label{eq:hess_g_bound_likelihood_regular}
\end{align}
for constants $C_1,C_2<\infty$.

Similarly, for higher orders, we define for a multi-index $\alpha=(\alpha_1,\ldots,\alpha_m)\in\mathbb{N}^m$, with $|\alpha|=\sum_{j=1}^m \alpha_j$, the partial derivative of order $|\alpha |$ by $\partial_\theta^\alpha = \frac{\partial^{|\alpha|}}{\partial\theta_1^{\alpha_1}\cdots\partial\theta_m^{\alpha_m}}$.
For each fixed $\alpha$, $\partial_\theta^\alpha g(\theta;y,z)$ is a polynomial of degree at most $|\alpha|$ in $A(z)$ and $y-\beta A(z)\theta$, multiplied by $g(\theta;y,z)$. Therefore, using the bounds on $A(z)$ and $\theta$, there exists $C_\alpha<\infty$ such that, uniformly over $\theta\in\Theta$ and a.e. $z\in\mathcal{Z}$,
\begin{align}
    \left|\partial_\theta^\alpha g(\theta;y,z)\right|
    \le C_\alpha(1+\|y\|)^{|\alpha|}g(\theta;y,z).
    \label{eq:multi_derivative_g_bound_likelihood_regular}
\end{align}
Since the means $\beta A(z)\theta$ are uniformly bounded in $\mathcal{B}(0,R)$, there exist constants $C_\alpha'<\infty$ and $c>0$ such that
\begin{align}
    \left|\partial_\theta^\alpha g(\theta;y,z)\right|
    \le C_\alpha'(1+\|y\|)^{|\alpha|}e^{-c\|y\|^2},
    \label{eq:multi_derivative_g_integrable_bound_likelihood_regular}
\end{align}
again uniformly over $\theta\in\Theta$ and a.e. $z\in\mathcal{Z}$. For each fixed $y$, this bound is independent of $z$ and hence is $\mu$-integrable; thus dominated convergence justifies differentiation under the $\mu$-integral at every order, proving item~(1).

We next establish the domination bounds for the log-likelihood derivatives. Since all means lie in $\mathcal{B}(0,R)$, the mixture density admits uniform Gaussian upper and lower tail bounds: there exists $C<\infty$ such that, for all $y\in\mathbb{R}^d$,
\begin{align}
    \sup_{\theta\in U}|\log p_{\theta,\beta}(y)|
    \le C(1+\|y\|^2).
    \label{eq:log_p_growth_bound_likelihood_regular}
\end{align}
Moreover, the first two derivatives of $g$ satisfy~\eqref{eq:grad_g_bound_likelihood_regular}--\eqref{eq:hess_g_bound_likelihood_regular}, uniformly over $\theta\in\Theta$ and a.e. $z\in\mathcal{Z}$.
Integrating these bounds over $z$ gives
\begin{align}
    \|\nabla_\theta p_{\theta,\beta}(y)\|
    &\le C(1+\|y\|)p_{\theta,\beta}(y),
    \label{eq:grad_p_bound_likelihood_regular}
    \\
    \|\nabla_\theta^2 p_{\theta,\beta}(y)\|_{\mathrm{op}}
    &\le C(1+\|y\|^2)p_{\theta,\beta}(y),
    \label{eq:hess_p_bound_likelihood_regular}
\end{align}
uniformly over $\theta\in U$. Hence
\begin{align}
    \sup_{\theta\in U}\|\nabla_\theta\log p_{\theta,\beta}(y)\|
    &\le C(1+\|y\|),
    \label{eq:score_growth_bound_likelihood_regular}
    \\
    \sup_{\theta\in U}\|\nabla_\theta^2\log p_{\theta,\beta}(y)\|_{\mathrm{op}}
    &\le C(1+\|y\|^2),
    \label{eq:hessian_growth_bound_likelihood_regular}
\end{align}
where the second bound follows from $\nabla_\theta^2\log p = \nabla_\theta^2p / p -  (\nabla_\theta p)(\nabla_\theta p)^\top/ p^2$.
Combining \eqref{eq:log_p_growth_bound_likelihood_regular}, \eqref{eq:score_growth_bound_likelihood_regular}, and \eqref{eq:hessian_growth_bound_likelihood_regular}, item~(2) follows with $G_U(y)=C(1+\|y\|^2)$. This function is integrable under $p_{\theta^\star,\beta}$, since $Y=\beta A(Z)\theta^\star+\xi$ has finite second moment.

It remains to prove the Fisher-information identity. The bounds above on $\nabla_\theta p_{\theta,\beta}(y)$ and $\nabla_\theta^2 p_{\theta,\beta}(y)$ provide integrable Lebesgue-dominating functions locally uniformly in $\theta$, so differentiation may be passed under the integral in $\int_{\mathbb{R}^d}p_{\theta,\beta}(y)\,dy=1$. Differentiating once gives $0 = \int_{\mathbb{R}^d}\nabla_\theta p_{\theta,\beta}(y)\,dy = \mathbb{E}_{Y\sim p_{\theta,\beta}} \big[\nabla_\theta\log p_{\theta,\beta}(Y)\big]$. Differentiating this identity once more yields
\begin{align}
    0 = \mathbb{E}_{Y\sim p_{\theta,\beta}} \big[\nabla_\theta^2\log p_{\theta,\beta}(Y)\big] + \mathbb{E}_{Y\sim p_{\theta,\beta}} \big[ \nabla_\theta\log p_{\theta,\beta}(Y) \nabla_\theta\log p_{\theta,\beta}(Y)^\top \big]. \label{eq:fisher_identity_likelihood_regular}
\end{align}
Rearranging gives
\begin{align}
    \mathcal{I}_{\mathrm{obs}}(\theta;\beta)
    &=
    \mathbb{E}_{Y\sim p_{\theta,\beta}}
    \big[
    \nabla_\theta\log p_{\theta,\beta}(Y)
    \nabla_\theta\log p_{\theta,\beta}(Y)^\top
    \big]
    \notag\\
    &=
    -
    \mathbb{E}_{Y\sim p_{\theta,\beta}}
    \big[
    \nabla_\theta^2\log p_{\theta,\beta}(Y)
    \big],
    \label{eq:observed_fisher_identity_likelihood_regular}
\end{align}
which proves item~(3).
\end{proof}

\subsection{Consistency and asymptotic efficiency of MLE}
\label{subsec:cons_efficiency_MLE}

We next formulate the conditions under which maximum likelihood provides the appropriate consistency and local efficiency guarantees on the quotient. The assumption in the following has two roles. First, it requires the population likelihood to identify the equivalence class of the ground truth.
Second, it requires the observed Fisher information to be nondegenerate in the normal space $W$. Together with the regularity conditions in Assumption~\ref{ass:integrability_max} and the local quotient structure in Assumption~\ref{ass:equiv_class_submanifold}, these conditions yield the standard consistency and asymptotic efficiency of the MLE, formulated in normal coordinates on the quotient. 

\begin{assum}[Quotient identifiability and first-order nondegeneracy]
\label{ass:mle_identifiable_nondegenerate}
Fix $\theta^\star\in\Theta$ and $\beta>0$. Assume that:
\begin{enumerate}
    \item The population log-likelihood criterion $\mathcal{L}(\theta;\theta^\star,\beta)\triangleq  \mathbb{E}_{Y\sim p_{\theta^\star,\beta}}\!\big[\log p_{\theta,\beta}(Y)\big]$
    is uniquely maximized on $\Theta$ at the equivalence class $[\theta^\star]$, i.e., 
    \begin{align}
        \mathcal{L}(\theta;\theta^\star,\beta) = \mathcal{L}(\theta^\star;\theta^\star,\beta)        \quad\Longleftrightarrow\quad \theta\sim\theta^\star,
        \qquad \theta\in\Theta.
    \end{align}

    \item The restricted observed Fisher information $\mathcal{I}_{\mathrm{obs}}^{(W)}(\theta^\star;\beta)$~\eqref{eq:restricted_obs_info_positive_mle} is positive-definite on $W$, that is $w^\top \mathcal{I}_{\mathrm{obs}}(\theta^\star;\beta) w>0$, for every $0\neq w\in W$.
\end{enumerate}
\end{assum}


\begin{proposition}[Consistency and asymptotic efficiency of the MLE on the quotient]
\label{prop:mle_quotient_efficiency}
Fix $\beta>0$. Assume Assumptions~\ref{ass:integrability_max}, \ref{ass:equiv_class_submanifold}, and~\ref{ass:mle_identifiable_nondegenerate}. Let $\widehat\theta_{\mathrm{MLE}}$ be any  maximizer of~\eqref{eqn:mleGeneral_efficiency}. Then:
\begin{enumerate}
    \item \emph{(Consistency on the quotient).} As $n \to \infty$,
    \begin{align}
        d_{\mathrm{eq}}\bigl(\widehat\theta_{\mathrm{MLE}},\theta^\star\bigr) \xrightarrow[]{\mathbb{P}}0.
    \end{align}

    \item \emph{(Asymptotic normality and efficiency).} There exists a choice of representatives $\widetilde{\theta}_{\mathrm{MLE}}\in[\widehat\theta_{\mathrm{MLE}}]$ such that $\widetilde{\theta}_{\mathrm{MLE}}\to\theta^\star$ in probability and
    \begin{align}
        \sqrt{n}\,\Pi_W\!\bigl(\widetilde{\theta}_{\mathrm{MLE}}-\theta^\star\bigr) \xrightarrow[]{\mathcal{D}} \mathcal{N}\!\bigl(0,\mathcal{I}_{\mathrm{obs}}^{(W)}(\theta^\star;\beta)^{-1}\bigr).
    \end{align}
\end{enumerate}
\end{proposition}

\begin{proof}
We prove in two steps.

\paragraph{Step 1: Consistency on the quotient.}
For the consistency claim, we observe that
\begin{align}
    \mathcal{L}(\theta^\star;\theta^\star,\beta)-\mathcal{L}(\theta;\theta^\star,\beta) = \mathbb{E}_{Y\sim p_{\theta^\star,\beta}}\!\left[\log\frac{p_{\theta^\star,\beta}(Y)}{p_{\theta,\beta}(Y)} \right] = D_{\mathrm{KL}}\!\bigl(p_{\theta^\star,\beta}\,\|\,p_{\theta,\beta}\bigr) \ge 0 . \label{eq:mle_consistency_KL_gap}
\end{align}
where $D_{\mathrm{KL}}(p\|q)$ denotes the Kullback--Leibler divergence from $p$ to $q$. Equality holds if and only if $p_{\theta,\beta}=p_{\theta^\star,\beta}$ almost everywhere, that is, if and only if $\theta\sim\theta^\star$. Thus, by Assumption~\ref{ass:mle_identifiable_nondegenerate}(1), the population log-likelihood is uniquely maximized on the quotient at $[\theta^\star]$. 

By the uniform law of large numbers for parametric classes with compact parameter space, pointwise continuity in the parameter, and an integrable envelope, we have
\begin{align}
    \sup_{\theta\in\Theta} \bigl| \mathcal{L}_n(\theta;\mathcal{Y}) - \mathcal{L}(\theta;\theta^\star,\beta) \bigr| \xrightarrow[]{\mathbb{P}}0;
\end{align}
see, for example,~\cite[Sec.~2]{newey1994large}. Since the population criterion is maximized precisely on the equivalence class $[\theta^\star]$, the argmax theorem~\cite{van2000asymptotic} yields consistency in quotient loss:
\begin{align}
    d_{\mathrm{eq}}\!\bigl(\widehat\theta_{\mathrm{MLE}},\theta^\star\bigr) \xrightarrow[]{\mathbb{P}}0. \label{eq:mle_quotient_consistency_proof}
\end{align}

\paragraph{Step 2: Asymptotic normality and efficiency.}
For the asymptotic normality claim, Assumption~\ref{ass:equiv_class_submanifold} implies that, after shrinking neighborhoods if necessary, there exists a $C^2$ local parametrization of identifiable perturbations, $\Psi:U_W\subset W\to\Theta$, such that $\Psi(0)=\theta^\star$ and for every $w\in W$, $D\Psi(0)[w]=w$. 
We choose $\Psi$ as a local transversal parametrization, so that each nearby equivalence class has a unique representative in $\Psi(U_W)$. By the quotient consistency in~\eqref{eq:mle_quotient_consistency_proof}, we may therefore choose representatives $\widetilde{\theta}_{\mathrm{MLE}}\in[\widehat\theta_{\mathrm{MLE}}]$ such that $\widetilde{\theta}_{\mathrm{MLE}} = \Psi(\widehat{w}_n)$, with $\widehat{w}_n\in W$ and $\widehat{w}_n\xrightarrow[]{\mathbb{P}}0$.

Now consider the locally parametrized likelihood $\mathcal{L}_n(\Psi(w);\mathcal{Y})$, with $w\in W$. By construction, $\widehat w_n$ is, with probability tending to one, a local maximizer of this criterion in a neighborhood of $0$. The local model $w\mapsto p_{\Psi(w),\beta}$ is differentiable in quadratic mean at $w=0$. This follows from the standard differentiability-in-quadratic-mean theory for smooth dominated parametric models, together with the chain rule for differentiable-in-quadratic-mean families; see, for example,~\cite[Sec.~7]{le2000asymptotics} and~\cite[Chs.~7--8]{van2000asymptotic}. The required domination and smoothness conditions are verified in Lemma~\ref{lem:likelihood_fisher_regular}(1)--(2). Its Fisher information is
\begin{align}
    D\Psi(0)^\top \mathcal{I}_{\mathrm{obs}}(\theta^\star;\beta) D\Psi(0) = \Pi_W \mathcal{I}_{\mathrm{obs}}(\theta^\star;\beta) \Pi_W\big|_W = \mathcal{I}_{\mathrm{obs}}^{(W)}(\theta^\star;\beta).    \label{eq:local_fisher_information_equals_restricted}
\end{align}
Since $\mathcal{I}_{\mathrm{obs}}^{(W)}(\theta^\star;\beta)$ is positive-definite by Assumption~\ref{ass:mle_identifiable_nondegenerate}(2), the standard asymptotic normality theorem for MLEs in differentiable-in-quadratic-mean models applies to the locally parametrized model. Hence
\begin{align}
    \sqrt{n}\,\widehat{w}_n \xrightarrow[]{\mathcal{D}} \mathcal{N}\!\left(0,\, \mathcal{I}_{\mathrm{obs}}^{(W)}(\theta^\star;\beta)^{-1} \right). \label{eq:mle_local_w_asymptotic_normality}
\end{align}

Finally, since $\Psi$ is $C^2$ and satisfies $D\Psi(0)[w]=w$, then as $w \to 0$,
\begin{align}
    \Pi_W\bigl(\Psi(w)-\theta^\star\bigr) = w+O(\|w\|^2). \label{eq:local_param_projection_expansion}
\end{align}
Applying~\eqref{eq:local_param_projection_expansion} with $w=\widehat{w}_n$, and using $\widehat{w}_n=O_{\mathbb{P}}(n^{-1/2})$ from~\eqref{eq:mle_local_w_asymptotic_normality}, gives
\begin{align}
    \Pi_W\bigl(\widetilde{\theta}_{\mathrm{MLE}}-\theta^\star\bigr) = \widehat{w}_n+o_{\mathbb{P}}(n^{-1/2}). \label{eq:mle_projection_equals_local_coordinate}
\end{align}
Combining~\eqref{eq:mle_local_w_asymptotic_normality} and~\eqref{eq:mle_projection_equals_local_coordinate}, we obtain
\begin{align}
    \sqrt{n}\, \Pi_W\!\bigl(\widetilde{\theta}_{\mathrm{MLE}}-\theta^\star\bigr) \xrightarrow[]{\mathcal{D}} \mathcal{N}\!\left( 0,\, \mathcal{I}_{\mathrm{obs}}^{(W)}(\theta^\star;\beta)^{-1} \right). \label{eq:mle_projected_asymptotic_normality}
\end{align}
This proves asymptotic normality and efficiency on the quotient.

\end{proof}

\section{Hermite expansions and low-SNR likelihood asymptotics}
\label{sec:hermite_low_snr_expansions}

In this section we develop the analytic expansions around the Gaussian noise baseline that underlie the low-SNR theory. We first recall the Hermite-tensor identities, then derive the Hermite-moment expansion of the likelihood ratio and its uniform low-SNR consequences. These results yield the leading score and local log-likelihood expansions.

\subsection{Hermite tensors for Gaussian expansions}
\label{subsec:hermite_prelims}

We collect standard definitions and identities for multivariate Hermite polynomials. These tools will be used repeatedly to expand likelihood ratios and score functions around a Gaussian reference measure. The presentation below follows standard treatments of multivariate Hermite polynomials and Wiener-It\^o chaos; see, e.g., \cite{nualart2006malliavin,peccati2011wiener,geng2020wiener,nualart2019malliavin}.

\begin{definition}[$k$-fold tensor derivative]
\label{def:kfold_tensor_derivative}
Let $f:\mathbb{R}^d\to\mathbb{R}$ be $k$ times differentiable. The $k$-fold tensor derivative of $f$ at $x\in\mathbb{R}^d$, denoted $\nabla^{\otimes k} f(x)$, is the order-$k$ tensor in $(\mathbb{R}^d)^{\otimes k}$ whose entries are
\begin{align}
    \label{eq:kfold_tensor_derivative_entries}
    \big(\nabla^{\otimes k} f(x)\big)_{i_1,\ldots,i_k} \;=\; \frac{\partial^k f(x)}{\partial x_{i_1}\cdots \partial x_{i_k}},
    \qquad i_1,\ldots,i_k\in\{1,\ldots,d\}.
\end{align}
If $f\in C^k$, then $\nabla^{\otimes k} f(x)$ is symmetric in its indices (i.e., mixed partial derivatives commute).
\end{definition}

\begin{definition}[Hermite tensors]\label{def:hermite-tensors}
Define the \emph{Hermite tensor} $H_k:\mathbb{R}^d\to(\mathbb{R}^d)^{\otimes k}$ by the Rodrigues formula~\cite{geng2020wiener}
\begin{align}
    \label{eq:Hermite_tensor_def_prelim}
    H_k(x)\;\triangleq\;(-1)^k\,e^{\|x\|^2/2}\,\nabla_x^{\otimes k}\,e^{-\|x\|^2/2}.
\end{align}
Here $(\mathbb{R}^d)^{\otimes k}$ is equipped with the natural tensor inner product $\langle\cdot,\cdot\rangle$ and Frobenius norm $\|\cdot\|_F$.
    
\end{definition}

The next result is a basic identity in the theory of Hermite polynomials and Wiener-Ito chaos; it is the exponential generating function of the Hermite tensors~\cite{geng2020wiener}.

\begin{lem}[Exponential generating function]
\label{lem:Wick_identity_prelim}
For every $u\in\mathbb{R}^d$, $x\in\mathbb{R}^d$, and $\beta\in\mathbb{R}$,
\begin{align}
    \label{eq:Wick_expansion_prelim}
    \exp\!\left( \beta \langle x,u\rangle-\frac{\beta^2}{2}\|u\|^2 \right) = \sum_{k=0}^\infty \frac{\beta^k}{k!}\,\big\langle H_k(x),\,u^{\otimes k}\big\rangle,
\end{align}
where the series converges absolutely for every fixed $(x,u,\beta)$.
\end{lem}

\begin{remark}[Convention for symmetric tensor representatives]
\label{rem:symmetric_tensor_representatives}
Throughout, order-$k$ tensors are understood as symmetric tensors whenever they are paired with Hermite tensors or moment tensors, i.e., all  inner products involving $H_k(x)$, $T_k(\theta)$, or $DT_k(\theta^\star)[h]$ depend only on the symmetric part of the corresponding tensor. 
\end{remark}

To use Hermite tensors for Gaussian $L^2$ expansions, it is useful to group the scalar polynomial functions obtained by contracting $H_k(x)$ with order-$k$ tensors according to their Hermite order.

\begin{definition}[Hermite chaoses]
\label{def:Hermite_chaoses}
Recall that $\phi(x)=(2\pi)^{-d/2}\exp(-\|x\|^2/2)$ is the standard Gaussian density on $\mathbb{R}^d$~\eqref{eq:def_standard_gaussian_density_phi}, and let $L^2(\phi) \triangleq  L^2(\mathbb{R}^d,\phi(x)\,dx)$.
For each $k\ge 0$, the \emph{$k$-th Hermite chaos} is the finite-dimensional subspace
$\mathcal{H}_k\subset L^2(\phi)$ defined by
\begin{align}
    \mathcal{H}_k \triangleq \left\{ x\mapsto \big\langle H_k(x),a\big\rangle \;:\; a\in(\mathbb{R}^d)^{\otimes k} \right\}.
\end{align}
For $k=0$, we use the convention that $\mathcal{H}_0$ is the space of constant functions.
\end{definition}

The elements of $\mathcal{H}_k$ are scalar-valued polynomial functions on $\mathbb{R}^d$. For example, the zeroth Hermite chaos is $\mathcal{H}_0 = \{x\mapsto c:\ c\in\mathbb{R}\}$, namely, the space of all constant functions.
Since $H_1(x)=x$, the first Hermite chaos is $\mathcal{H}_1  = \{x\mapsto a^\top x:\ a\in\mathbb{R}^d\}$, namely, the space of all linear functions.
Moreover, since $H_2(x)=x^{\otimes 2}-I_d$, the second Hermite chaos is $\mathcal{H}_2 = \{x\mapsto x^\top A x-\operatorname{tr}(A):\ A=A^\top\}$, namely, the space of centered quadratic functions.

The following example illustrates that the Hermite-chaos decomposition separates a function into orthogonal components of different chaos orders. 

\begin{example}[A simple Hermite-chaos decomposition]
\label{ex:simple_Hermite_decomposition}
The polynomial $x_1^2$ has degree two, but it is not purely second-order in the Hermite-chaos decomposition. Indeed, $x_1^2 = (x_1^2-1)+1$, where $x_1^2-1 \in \mathcal{H}_2$, and $1\in\mathcal{H}_0$.
Thus, $x_1^2$ has a second-order Hermite component and a zeroth-order Hermite
component.
\end{example}

\begin{lem}[Orthogonality of Hermite chaoses]
\label{lem:Hermite_orthogonality_prelim}
Let $X\sim\mathcal{N}(0,I_d)$. For any $k,\ell\ge 0$, $a\in(\mathbb{R}^d)^{\otimes k}$, and $b\in(\mathbb{R}^d)^{\otimes \ell}$,
\begin{align}
    \label{eq:Hermite_orthogonality_prelim}    \mathbb{E}\Big[\big\langle H_k(X),a\big\rangle\,\big\langle H_\ell(X),b\big\rangle\Big] = \delta_{k,\ell}\,k!\,\langle a,b\rangle.
\end{align}
\end{lem}

The next standard result states that the Hermite chaoses defined above form a complete orthogonal decomposition of the Gaussian Hilbert space $L^2(\phi)$; see, e.g., \cite{janson1997gaussian, nualart2006malliavin}.

\begin{lem}[Gaussian Hermite chaos decomposition]
\label{lem:L2_Hermite_expansion_prelim}
The Hilbert space $L^2(\phi)$ admits the orthogonal decomposition $L^2(\phi) = \bigoplus_{k=0}^{\infty} \mathcal{H}_k$. In particular, every $R\in L^2(\phi)$ admits a unique expansion
\begin{align}
    R(x) = \sum_{k=0}^{\infty} \frac{1}{k!} \big\langle H_k(x),\alpha_k\big\rangle,
    \qquad \alpha_k\in(\mathbb{R}^d)^{\otimes k},
\end{align}
with convergence in $L^2(\phi)$. The coefficients are given by
\begin{align}
    \alpha_k = \mathbb{E}_\phi\!\left[R(X)H_k(X)\right],
    \qquad X\sim \mathcal{N}(0,I_d),
\end{align}
and Parseval's identity takes the form $\|R\|_{L^2(\phi)}^2 = \sum_{k=0}^{\infty} \frac{1}{k!}\|\alpha_k\|_F^2$.
\end{lem}

The following elementary lemma records a simple bookkeeping fact that will be used below: products of finite-order Hermite components cannot generate chaos orders larger than their total ordinary polynomial degree.

\begin{lem}[Polynomial degree and Hermite chaoses]
\label{lem:polynomial_degree_chaos}
Let $p:\mathbb{R}^d\to\mathbb{R}$ be a polynomial of degree at most $m$.
Then $p \in \bigoplus_{q=0}^{m}\mathcal{H}_q$.
Consequently, if $f_j\in\mathcal{H}_j$ for $j=1,\ldots,m$, then every finite product $f_{j_1}\cdots f_{j_r}$ with $j_1+\cdots+j_r\le m$ satisfies $f_{j_1}\cdots f_{j_r} \in \bigoplus_{q=0}^{m}\mathcal{H}_q$.
\end{lem}

\begin{proof}[Proof of Lemma~\ref{lem:polynomial_degree_chaos}]
The multivariate Hermite polynomials form a triangular basis for ordinary polynomials: the $q$-th Hermite chaos $\mathcal{H}_q$ consists of polynomial functions whose leading degree is $q$. Hence every polynomial of degree at most $m$ admits a finite Hermite expansion involving only chaos orders $0,\ldots,m$, which proves the first claim. For the second claim, each $f_j\in\mathcal{H}_j$ is a polynomial of degree at most $j$. Therefore $f_{j_1}\cdots f_{j_r}$ is a polynomial of degree at most $j_1+\cdots+j_r\le m$, and the first claim applies.
\end{proof}

\subsection{Hermite-moment expansion of the likelihood}
\label{subsec:hermite_moment_lr}

We now record a convenient Hermite expansion for the likelihood ratio with respect to the Gaussian noise baseline.
This expansion expresses the likelihood ratio coefficients in terms of the signal moment tensors
$\{T_k(\theta)\}_{k\ge 1}$, and will be used repeatedly to derive low-SNR asymptotics.

For each $k\ge 1$, recall the order-$k$ signal moment tensor $T_k(\theta)$ from Definition~\ref{def:signal_population_moment_tensor}, and let $p_{\theta,\beta}$ be the marginal density of $Y$ under \eqref{eq:meanscaling-model}. 
Define the likelihood ratio
\begin{align}
    \label{eq:lr_def}
    R_{\theta,\beta}(y) \;\triangleq\; \frac{p_{\theta,\beta}(y)}{\phi(y)},
    \qquad y\in\mathbb{R}^d.
\end{align}
Marginalizing over $Z\sim\mu$ gives
\begin{align}
    \label{eq:lr_exp_mixture_identity}
    R_{\theta,\beta}(y) &= \int_{\mathcal{Z}} \exp\!\left( \beta\langle y,A(z)\theta\rangle -\frac{\beta^2}{2}\|A(z)\theta\|^2 \right)\,\mu(dz) \nonumber\\
    &=
    \mathbb{E}_{Z\sim\mu} \left[ \exp\!\left( \beta\langle y,A(Z)\theta\rangle -\frac{\beta^2}{2}\|A(Z)\theta\|^2 \right) \right].
\end{align}

\begin{lem}[Hermite-moment expansion of the likelihood ratio]
\label{lem:hermite_moment_lr_expansion}
Consider \eqref{eq:meanscaling-model} and assume Assumption~\ref{ass:integrability_max} holds.
Let $\{H_k\}_{k\ge 0}$ be the Hermite tensors defined in \eqref{eq:Hermite_tensor_def_prelim}.
Then for every $y\in\mathbb{R}^d$,
\begin{align}
    \label{eq:lr_hermite_series}
    R_{\theta,\beta}(y) \;=\; \sum_{k=0}^\infty \frac{\beta^k}{k!}\, \big\langle H_k(y),\,T_k(\theta)\big\rangle,
\end{align}
where $T_0(\theta)\equiv 1$ and the series is absolutely convergent for each fixed $y$.
\end{lem}

\begin{proof}[Proof of Lemma~\ref{lem:hermite_moment_lr_expansion}]
Applying the Wick identity  (Lemma~\ref{lem:Wick_identity_prelim}) on the expression within the expectation in~\eqref{eq:lr_exp_mixture_identity} with $u=A(Z)\theta$ gives, for each realization of the $Z$,
\begin{align}
    \exp\!\Big(\beta \langle y,A(Z)\theta\rangle-\tfrac{\beta^2}{2}\|A(Z)\theta\|^2\Big) = \sum_{k=0}^\infty \frac{\beta^k}{k!}\, \big\langle H_k(y),\,(A(Z)\theta)^{\otimes k}\big\rangle .
\end{align}
Taking $\mathbb{E}_{Z\sim\mu}[\cdot]$ term-by-term and using $T_k(\theta)=\mathbb{E}_{Z\sim\mu}[(A(Z)\theta)^{\otimes k}]$ yields \eqref{eq:lr_hermite_series}. The term-by-term interchange is justified under Assumption~\ref{ass:integrability_max} (uniform moment control of $\|A(Z)\theta\|$ on compact $\Theta$), which ensures absolute convergence of the resulting power series for each fixed $y$.
\end{proof}

We now pass from the exact Hermite expansion of the likelihood ratio to estimates that are uniform on compact sets. 

\begin{lem}[Uniform Hermite expansion of the likelihood ratio]
\label{lem:uniform_Hermite_likelihood_ratio}
Assume Assumption~\ref{ass:integrability_max}, and let $U\subset\Theta$ be compact. Then, for every integer $L\ge 0$ and every $p<\infty$, there exist constants $C_{U,L,p}<\infty$ and $\beta_0>0$ such that, for all $|\beta|\le \beta_0$,
\begin{align}
    \sup_{\theta\in U} \left\| R_{\theta,\beta} - \sum_{k=0}^{L} \frac{\beta^k}{k!} \big\langle H_k, T_k(\theta)\big\rangle \right\|_{L^p(\phi)}   \le C_{U,L,p}|\beta|^{L+1}.
\end{align}
Moreover,
\begin{align}
    \sup_{\theta\in U,\ \|h\|\le 1} \left\| \partial_\varepsilon R_{\theta+\varepsilon h,\beta}\big|_{\varepsilon=0} - \sum_{k=1}^{L} \frac{\beta^k}{k!} \big\langle H_k, DT_k(\theta)[h]\big\rangle \right\|_{L^p(\phi)} \le C_{U,L,p}|\beta|^{L+1}.
\end{align}
\end{lem}

\begin{proof}[Proof of Lemma~\ref{lem:uniform_Hermite_likelihood_ratio}]
Fix $U\subset\Theta$ compact, and define $M_U \triangleq a_{\max}\sup_{\theta\in U}\|\theta\|<\infty$, where $a_{\max}$ is the uniform bound from Assumption~\ref{ass:integrability_max}. Then, for $\mu$-a.e. $z$ and every $\theta\in U$, $\|A(z)\theta\|\le M_U$. Hence, uniformly over $\theta\in U$,
\begin{align}
    \|T_k(\theta)\| = \left\| \mathbb{E}_{Z\sim\mu}\big[(A(Z)\theta)^{\otimes k}\big] \right\| \le \mathbb{E}_{Z\sim\mu}\|A(Z)\theta\|^k \le M_U^k . \label{eq:Tk_uniform_bound}
\end{align}

For fixed $z$ and $\theta$, set $v=A(z)\theta$ and define $F_{v,\beta}(y) \triangleq \exp\!\left(\beta \langle y,v\rangle-\beta^2\|v\|^2/2\right)$. By the Hermite generating function (Lemma~\ref{lem:Wick_identity_prelim}), $F_{v,\beta}(y) = \sum_{k=0}^{\infty} \frac{\beta^k}{k!} \big\langle H_k(y),v^{\otimes k}\big\rangle$. Since $\|v\|\le M_U$, Taylor's theorem applied to $\beta \mapsto F_{v,\beta}(y)$ gives, uniformly over $\|v\|\le M_U$ and $|\beta|\le \beta_0$,
\begin{align}
    \left| F_{v,\beta}(y) - \sum_{k=0}^{L} \frac{\beta^k}{k!} \big\langle H_k(y),v^{\otimes k}\big\rangle \right| \le C_{U,L} |\beta|^{L+1} (1+\|y\|^{L+1}) \exp(c_U \beta_0\|y\|), \label{eq:Taylor_remainder_F_v_t}
\end{align}
for constants $C_{U,L},c_U<\infty$. The right-hand side belongs to $L^p(\phi)$ for every fixed $p<\infty$.

Now, since $R_{\theta,\beta}(y) = \mathbb{E}_{Z \sim \mu} [F_{A(Z)\theta,\beta}(y)]$, and $T_k(\theta) = \mathbb{E}_{Z \sim \mu}[(A(Z)\theta)^{\otimes k}]$, we may subtract the truncated series and write
\begin{align}
    &R_{\theta,\beta}(y)
    - \sum_{k=0}^{L} \frac{\beta^k}{k!} \big\langle H_k(y),T_k(\theta)\big\rangle
    \notag\\
    &\qquad =
    \mathbb{E}_{Z \sim \mu} \left[ F_{A(Z)\theta,\beta}(y) - \sum_{k=0}^{L} \frac{\beta^k}{k!} \big\langle H_k(y),(A(Z)\theta)^{\otimes k}\big\rangle \right]. \label{eq:R_truncated_remainder_as_expectation}
\end{align}
Therefore, by Minkowski's inequality and the bound in~\eqref{eq:Taylor_remainder_F_v_t},
\begin{align}
    \sup_{\theta\in U} \left\| R_{\theta,\beta} - \sum_{k=0}^{L} \frac{\beta^k}{k!} \big\langle H_k,T_k(\theta)\big\rangle \right\|_{L^p(\phi)} \le C_{U,L,p} |\beta|^{L+1}. \label{eq:uniform_Lp_likelihood_ratio_remainder}
\end{align}
This proves the first claim.

For the directional derivative, differentiate
$T_k(\theta) = \mathbb{E}_{Z\sim\mu}\big[(A(Z)\theta)^{\otimes k}\big]$
to obtain
\begin{align}
    DT_k(\theta)[h] = \sum_{j=1}^{k} \mathbb{E}_{Z\sim\mu}\Big[(A(Z)\theta)^{\otimes(j-1)} \otimes A(Z)h \otimes (A(Z)\theta)^{\otimes(k-j)}\Big].
\end{align}
Therefore, $\|DT_k(\theta)[h]\| \le k\,a_{\max}\,\|h\|\,M_U^{k-1}$,
uniformly in $\theta\in U$. Applying the same orthogonality argument to the derivative series gives the second estimate.
\end{proof}

\subsection{Finite-order expansion of the reciprocal likelihood ratio}
Recall the likelihood ratio $R_{\theta,\beta}(y)$ from~\eqref{eq:lr_def}, together with its Hermite expansion in~\eqref{eq:lr_hermite_series}, and recall the observed Fisher information from~\eqref{eq:def_I_obs}. It is convenient to introduce the directional observed score
\begin{align}
    \label{eq:score_dir2}
    S_{\theta,\beta}(y;h) \triangleq    \partial_\varepsilon\Big|_{\varepsilon=0} \log p_{\theta+\varepsilon h,\beta}(y),
\end{align}
where $y\in\mathbb{R}^d$ and $h\in\mathbb{R}^m$ is a parameter-space direction. This is the scalar score associated with the direction $h$, and its second moment yields the Fisher-information quadratic form in that direction. Indeed, for any $h\in\mathbb{R}^m$,
\begin{align}
    \mathbb{E}_{Y\sim p_{\theta^\star,\beta}}\!\left[S_{\theta^\star,\beta}(Y;h)^2 \right]
    &= \mathbb{E}_{Y\sim p_{\theta^\star,\beta}}\!\left[ \bigl(h^\top \nabla_\theta \log p_{\theta^\star,\beta}(Y)\bigr)^2\right] \\
    &=
    h^\top \mathcal{I}_{\mathrm{obs}}(\theta^\star;\beta)\, h.
\end{align}
Therefore,
\begin{align}
    \label{eqn:Fisher_matrix_to_directional_score_relation}
    h^\top \mathcal{I}_{\mathrm{obs}}(\theta^\star;\beta)\, h = \mathbb{E}_{Y\sim p_{\theta^\star,\beta}}\!\left[S_{\theta^\star,\beta}(Y;h)^2 \right],
    \qquad \forall h\in\mathbb{R}^m.
\end{align}

Since
$R_{\theta,\beta}(y)=p_{\theta,\beta}(y)/\phi(y)$ and $\phi$ does not depend on $\theta$, it follows that
\begin{align}
    \label{eq:score_lr_relation_sigma_x}
    S_{\theta,\beta}(y;h)
    &= \partial_\varepsilon\Big|_{\varepsilon=0} \log R_{\theta+\varepsilon h,\beta}(y) \nonumber\\
    &= \frac{\partial_\varepsilon\big|_{\varepsilon=0} R_{\theta+\varepsilon h,\beta}(y)}{R_{\theta,\beta}(y)} .
\end{align}

To expand the score, we will also need a finite-order expansion of the reciprocal factor $R_{\theta^\star,\beta}^{-1}$. The following lemma is a direct Taylor expansion of $(1+u)^{-1}$, applied to the Hermite-moment expansion of the likelihood ratio. 

\begin{lem}[Finite-order expansion of the reciprocal likelihood ratio]
\label{lem:lr_inverse_expansion}
Assume Assumption~\ref{ass:integrability_max}. Fix $\theta^\star\in\Theta$ and let $L\ge 1$. Then, there exists $\beta_0>0$ such that for all $|\beta|<\beta_0$,
\begin{align}
    R_{\theta^\star,\beta}(y)^{-1} = \sum_{m=0}^{L} \beta^m\,b_m(y) + O_{L^2(p_{\theta,\beta})}(\beta^{L+1}),
\end{align}
where $b_0\equiv 1$, and for each $m\ge 1$,
\begin{align}
    \label{eqn:def_b_m}
    b_m(y) = \sum_{r=1}^m (-1)^r \sum_{\substack{j_1+\cdots+j_r=m\\ j_i\ge 1}} c_{j_1}(y) \cdots c_{j_r}(y),
\end{align}
with
$c_m(y) \triangleq \frac{1}{m!}\,\big\langle H_m(y),T_m(\theta^\star)\big\rangle$.
In particular, each $b_m\in \bigoplus_{q=0}^m\mathcal{H}_q$ belongs to the direct sum of Hermite expansion of orders at most $m$, and the first-order term is $b_1(y) = -\big\langle H_1(y),T_1(\theta^\star)\big\rangle$.
\end{lem}

\begin{proof}[Proof of Lemma~\ref{lem:lr_inverse_expansion}]
By Lemma~\ref{lem:hermite_moment_lr_expansion},
\begin{align}
    R_{\theta^\star,\beta}(y) = 1+\sum_{m=1}^{L} \beta^m\,c_m(y) + O_{L^2(p_{\theta,\beta})}(\beta^{L+1}),
\end{align}
where $c_m(y) \triangleq \frac{1}{m!}\,\big\langle H_m(y),T_m(\theta^\star)\big\rangle$, so each $c_m \in \mathcal{H}_m$ lies in Hermite chaos $m$. Since $R_{\theta^\star,\beta}(y)=p_{\theta^\star,\beta}(y)/\phi(y)$, it is strictly positive for every $y\in\mathbb{R}^d$. Using the Taylor expansion $(1+u)^{-1}=\sum_{r=0}^{L}(-1)^r u^r+O(u^{L+1})$ and substituting $u=\sum_{m=1}^{L} \beta^m c_m(y)$, we obtain
\begin{align}
    R_{\theta^\star,\beta}(y)^{-1} = \sum_{m=0}^{L} \beta^m\,b_m(y) + O_{L^2(p_{\theta,\beta})}(\beta^{L+1}),
\end{align}
where $b_0\equiv 1$ and, for each $m\ge 1$, the coefficient of $\beta^m$ is
\begin{align}
    b_m(y) = \sum_{r=1}^m (-1)^r \sum_{\substack{j_1+\cdots+j_r=m\\ j_i\ge 1}} c_{j_1}(y)\cdots c_{j_r}(y).
\end{align}
Thus $b_m$ is a finite polynomial in $c_1,\ldots,c_m$. Since $c_j\in\mathcal{H}_j$, Lemma~\ref{lem:polynomial_degree_chaos} implies that every monomial $c_{j_1}\cdots c_{j_r}$ with $j_1+\cdots+j_r=m$ belongs to $\bigoplus_{q=0}^m\mathcal{H}_q$. Therefore, $b_m\in \bigoplus_{q=0}^m\mathcal{H}_q$.
Finally, $b_1(y) = -c_1(y) = -\big\langle H_1(y),T_1(\theta^\star)\big\rangle$.
This proves the claim.
\end{proof}

\subsection{Finite-order Hermite truncation of the score}

Recall the definition of the directional observed score from~\eqref{eq:score_lr_relation_sigma_x}. The next lemma gives a finite-order Hermite expansion of the score around the ground truth. 

\begin{lem}[Finite-order Hermite truncation of the score]
\label{lem:score_Hermite_truncation}
Assume Assumption~\ref{ass:integrability_max}, and let $L\ge 1$. Fix $\theta^\star\in\Theta$. Then, there exists $\beta_0>0$ such that for all $|\beta|<\beta_0$ and all $h\in W$,
\begin{align}
    S_{\theta^\star,\beta}(y;h) = \sum_{q=1}^{L} \beta^q \sum_{j=1}^{q} a_j(y;h)\,b_{q-j}(y) + r_{L,\beta}(y;h),
\end{align}
where
\begin{align}
    a_j(y;h) \triangleq \frac{1}{j!}\, \big\langle H_j(y),DT_j(\theta^\star)[h]\big\rangle,
    \qquad j\ge 1,
    \label{eqn_app_C30}
\end{align}
$b_0\equiv 1$, and $b_m$~\eqref{eqn:def_b_m} is the coefficient of $\beta^m$ in the expansion of $R_{\theta^\star,\beta}(y)^{-1}$ from Lemma~\ref{lem:lr_inverse_expansion}. In particular, for each $q$, the coefficient 
\begin{align}
    \sum_{j=1}^{q}a_j(y;h)b_{q-j}(y) \in \bigoplus_{r=0}^{q}\mathcal{H}_r.
\end{align}
depends linearly on $h$ and belongs to the direct sum of Hermite chaoses of orders at most $q$. Moreover, uniformly for $h$ in bounded subsets of $W$, $\|r_{L,\beta}(y;h)\|_{L^2(p_{\theta, \beta})} =  O\left(\beta^{L+1}\|h\|\right)$.
\end{lem}

\begin{proof}[Proof of Lemma~\ref{lem:score_Hermite_truncation}]
By the second part of Lemma~\ref{lem:uniform_Hermite_likelihood_ratio}, evaluated at $\theta=\theta^\star$, the directional derivative of the likelihood ratio admits the finite-order expansion
\begin{align}
    \partial_\varepsilon\Big|_{\varepsilon=0}R_{\theta^\star+\varepsilon h,\beta}(y) = \sum_{j=1}^{L} \beta^j\,a_j(y;h) + O_{L^2(p_{\theta,\beta})}(\beta^{L+1}\|h\|),
    \label{eqn:app_B40}
\end{align}
uniformly for $h$ in bounded subsets of $W$. Likewise, by Lemma~\ref{lem:lr_inverse_expansion},
\begin{align}
    R_{\theta^\star,\beta}(y)^{-1} = \sum_{m=0}^{L} \beta^m\,b_m(y) + O_{L^2(p_{\theta,\beta})}(\beta^{L+1}).
    \label{eqn:app_B41}
\end{align}

Multiplying the two truncated expansions in~\eqref{eq:score_lr_relation_sigma_x} and collecting powers of $\beta$ yields
\begin{align}
    S_{\theta^\star,\beta}(y;h) = \sum_{q=1}^{L} \beta^q \sum_{j=1}^{q} a_j(y;h)\,b_{q-j} + r_{L,\beta}(y;h),
\end{align}
where $r_{L,\beta}$ collects terms of total order at least $L+1$.

Because $a_j(y;h)\in\mathcal{H}_j$, it is a scalar Hermite polynomial of 
degree at most $j$. Moreover, by Lemma~\ref{lem:lr_inverse_expansion}, $b_{q-j}(y)\in\bigoplus_{r=0}^{q-j}\mathcal{H}_r$, and hence $b_{q-j}$ is a polynomial of degree at most $q-j$. Therefore $a_j(y;h)b_{q-j}(y)$ is a polynomial of degree at most $q$. By Lemma~\ref{lem:polynomial_degree_chaos}, $a_j(y;h)b_{q-j}(y) \in \bigoplus_{r=0}^{q}\mathcal{H}_r$.
Consequently, $\sum_{j=1}^{q}a_j(y;h)b_{q-j}(y) \in \bigoplus_{r=0}^{q}\mathcal{H}_r$.
Its dependence on $h$ is linear because each $a_j(\cdot;h)$ is linear in $h$, while
$b_{q-j}$ is independent of $h$.
Finally, the $L^2(p_{\theta, \beta})$ remainder follows immediately from combining the remainders in~\eqref{eqn:app_B40}--\eqref{eqn:app_B41}.
\end{proof}

\subsection{Proof of Proposition~\ref{prop:fisher_layer_block_asymptotics}}
\label{app:proof_fisher_layer_block_asymptotics}

\paragraph{Step 1: Likelihood ratio and Hermite-moment expansion.}
Recall the definition of the likelihood ratio $R_{\theta,\beta}(y) = p_{\theta,\beta}(y)/\phi(y)$ from~\eqref{eq:lr_def}-\eqref{eq:lr_exp_mixture_identity}. By Lemma~\ref{lem:hermite_moment_lr_expansion},
\begin{align}
    \label{eq:Rtheta_hermite_series}
    R_{\theta,\beta}(y) \;=\; \sum_{k=0}^\infty \frac{\beta^k}{k!}\, \big\langle H_k(y),\,T_k(\theta)\big\rangle.
\end{align}
Under Assumption~\ref{ass:integrability_max}, $\theta\mapsto T_k(\theta)$ is differentiable and differentiation may be interchanged with the expectation defining $T_k$ (by dominated convergence). Hence, for any direction $h\in\mathbb{R}^m$,
\begin{align}
    \label{eq:dR_series}   \partial_\varepsilon\Big|_{\varepsilon=0}R_{\theta^\star+\varepsilon h,\beta}(y) = \sum_{k=1}^{\infty}\frac{\beta^k}{k!}\, \big\langle H_k(y),\,DT_k(\theta^\star)[h]\big\rangle .
\end{align}
Moreover, by Lemma~\ref{lem:uniform_Hermite_likelihood_ratio}, the series in \eqref{eq:Rtheta_hermite_series} and \eqref{eq:dR_series} converge in $L^2(p_{\theta, \beta})$ for $\beta$ sufficiently small, and truncation remainders admit uniform $L^2(p_{\theta, \beta})$ control under Assumption~\ref{ass:integrability_max}.

Since $h\in U_k$ and $g\in U_\ell$, the layer definition gives
$DT_j(\theta^\star)[h]=0$ for every $j<k$, and $DT_j(\theta^\star)[g]=0$ for every $j<\ell$.
Hence, by Lemma~\ref{lem:uniform_Hermite_likelihood_ratio},
\begin{align}
    \label{eq:dR_h_leading_blockIobs}   \partial_\varepsilon\Big|_{\varepsilon=0}    R_{\theta^\star+\varepsilon h,\beta}(y)
    &=
    \frac{\beta^k}{k!}\, \big\langle H_k(y),DT_k(\theta^\star)[h]\big\rangle + \frac{\beta^{k+1}}{(k+1)!}\, \big\langle H_{k+1}(y),DT_{k+1}(\theta^\star)[h]\big\rangle + O_{L^2(p_{\theta,\beta})}(\beta^{k+2}),
    \\    
    \label{eq:dR_g_leading_blockIobs}   
    \partial_\varepsilon\Big|_{\varepsilon=0} R_{\theta^\star+\varepsilon g,\beta}(y)
    &=
    \frac{\beta^\ell}{\ell!}\, \big\langle H_\ell(y),DT_\ell(\theta^\star)[g]\big\rangle + \frac{\beta^{\ell+1}}{(\ell+1)!}\, \big\langle H_{\ell+1}(y),DT_{\ell+1}(\theta^\star)[g]\big\rangle + O_{L^2(p_{\theta,\beta})}(\beta^{\ell+2}).
\end{align}
By~\eqref{eq:Rtheta_hermite_series},
\begin{align}
    \label{eq:Rstar_expansion_blockIobs}
    R_{\theta^\star,\beta}(y) = 1+\beta\,\big\langle H_1(y),T_1(\theta^\star)\big\rangle
    +O_{L^2(p_{\theta,\beta})}(\beta^2),
\end{align}
hence, by Lemma~\ref{lem:lr_inverse_expansion},
\begin{align}
    \label{eq:Rinv_expansion_blockIobs}
    R_{\theta^\star,\beta}(y)^{-1} = 1-\beta\,\big\langle H_1(y),T_1(\theta^\star)\big\rangle +O_{L^2(p_{\theta,\beta})}(\beta^2).
\end{align}

\paragraph{Step 2: Score expansion.}
Recall the definition of the observed score in direction $h$~\eqref{eq:score_lr_relation_sigma_x},
\begin{align}
    \label{eq:score_dir}
    S_{\theta,\beta}(y;h) \;\triangleq\; \partial_\varepsilon\Big|_{\varepsilon=0}\log p_{\theta+\varepsilon h,\beta}(y) = \partial_\varepsilon\Big|_{\varepsilon=0}\log R_{\theta+\varepsilon h,\beta}(y),
\end{align}
since $\phi(y)$ does not depend on $\theta$.
Combining~\eqref{eq:dR_h_leading_blockIobs}-\eqref{eq:dR_g_leading_blockIobs}, together with~\eqref{eq:Rinv_expansion_blockIobs}--\eqref{eq:score_dir}, the preceding expansions gives
\begin{align}
    \label{eq:score_h_expansion_blockIobs}
    S_{\theta^\star,\beta}(y;h)
    &=
    \beta^k A_h(y) + \beta^{k+1} B_h(y) + O_{L^2(p_{\theta,\beta})}(\beta^{k+2}),
    \\    \label{eq:score_g_expansion_blockIobs}
    S_{\theta^\star,\beta}(y;g)
    &=
    \beta^\ell A_g(y) + \beta^{\ell+1} B_g(y) + O_{L^2(p_{\theta,\beta})}(\beta^{\ell+2}),
\end{align}
where
\begin{align}
    \label{eq:Ah_def_blockIobs}
    A_h(y) &\triangleq  \frac{1}{k!}\, \big\langle H_k(y),DT_k(\theta^\star)[h]\big\rangle,
    \\    \label{eq:Ag_def_blockIobs}
    A_g(y) &\triangleq  \frac{1}{\ell!}\, \big\langle H_\ell(y),DT_\ell(\theta^\star)[g]\big\rangle,
\end{align}
and
\begin{align}
    \label{eq:Bh_def_blockIobs}
    B_h(y)
    &\triangleq 
    \frac{1}{(k+1)!}\, \big\langle H_{k+1}(y),DT_{k+1}(\theta^\star)[h]\big\rangle - \big\langle H_1(y),T_1(\theta^\star)\big\rangle A_h(y),
    \\    \label{eq:Bg_def_blockIobs}
    B_g(y)
    &\triangleq 
    \frac{1}{(\ell+1)!}\, \big\langle H_{\ell+1}(y),DT_{\ell+1}(\theta^\star)[g]\big\rangle - \big\langle H_1(y),T_1(\theta^\star)\big\rangle A_g(y).
\end{align}
By the recurrence relation for Hermite polynomials, multiplication of a Hermite polynomial of order $r$ by a linear function yields a linear combination of Hermite polynomials of orders $r+1$ and $r-1$. Consequently, $B_h$ is a linear combination of Hermite polynomials of orders $k+1$ and $k-1$, while $B_g$ involves orders $\ell+1$ and $\ell-1$.

\paragraph{Step 3: Bilinear Fisher pairing.}
By definition of the observed Fisher information in Definition~\ref{def:I_obs_meanscaling}, and the observed score in~\eqref{eq:score_dir}
\begin{align}
    \label{eq:fisher_bilinear_blockIobs}
    \big\langle h,\mathcal{I}_{\mathrm{obs}}^{(W)}(\theta^\star;\beta)g\big\rangle
    &=     \mathbb{E}_{Y\sim p_{\theta^\star,\beta}}\!\big[S_{\theta^\star,\beta}(Y;h)\,S_{\theta^\star,\beta}(Y;g) \big]
    \\
    &=
     \int S_{\theta^\star,\beta}(y;h)\,S_{\theta^\star,\beta}(y;g)\,    \phi(y)R_{\theta^\star,\beta}(y)\,dy
\end{align}
Using the expansions~\eqref{eq:score_h_expansion_blockIobs}-\eqref{eq:score_g_expansion_blockIobs}, together with~\eqref{eq:Rstar_expansion_blockIobs}, yields
\begin{align}
    \label{eq:fisher_expand_blockIobs}
    S_{\theta^\star,\beta}(y;h)\,S_{\theta^\star,\beta}(y;g)\,R_{\theta^\star,\beta}(y)
    &=
    \beta^{k+\ell}A_h(y)A_g(y)
    \\
    &\quad
    +\beta^{k+\ell+1} \Big( A_h(y)B_g(y)+B_h(y)A_g(y) +\langle H_1(y),T_1(\theta^\star)\rangle A_h(y)A_g(y) \Big)
    \nonumber\\
    &\quad
    +O(\beta^{k+\ell+2}).
    \nonumber
\end{align}
Hence,
\begin{align}
    \label{eq:fisher_expand_expectation_blockIobs}
    \big\langle h,\mathcal{I}_{\mathrm{obs}}^{(W)}(\theta^\star;\beta)g\big\rangle
    &= \beta^{k+\ell}\,\mathbb{E}_\phi[A_h(y)A_g(y)]
    \\
    &\quad
    +\beta^{k+\ell+1}\, \mathbb{E}_\phi\!\Big[ A_h(y)B_g(y)+B_h(y)A_g(y) +\langle H_1(y),T_1(\theta^\star)\rangle A_h(y)A_g(y) \Big]
    \nonumber\\
    &\quad
    +O(\beta^{k+\ell+2}).
    \nonumber
\end{align}

\paragraph{Step 4: Leading term and diagonal case.}
By orthogonality of Hermite polynomials (Lemma~\ref{lem:Hermite_orthogonality_prelim}),
\begin{align}
    \label{eq:AhAg_eval_blockIobs}    \mathbb{E}_\phi[A_h(y)A_g(y)] =
    \begin{cases}
        0, & k\neq \ell,\\[1ex]
        \displaystyle
        \frac{1}{k!}\,
        \big\langle DT_k(\theta^\star)[h],DT_k(\theta^\star)[g]\big\rangle_F,
        & k=\ell.
    \end{cases}
\end{align}
If $k=\ell$, then the coefficient of order $\beta^{2k+1}$ in \eqref{eq:fisher_expand_expectation_blockIobs} vanishes. Indeed, $A_h$ is a linear combination of Hermite polynomials of total order $k$, whereas $B_g$ and $B_h$ are linear combinations of Hermite polynomials of total orders $k+1$ and $k-1$, so
\begin{align}
    \mathbb{E}_\phi[A_hB_g]=\mathbb{E}_\phi[B_hA_g]=0
\end{align}
by orthogonality. Moreover, since $A_h$ and $A_g$ are linear combinations of Hermite polynomials of the same total order $k$, their product $A_h A_g$ is an even function. Since $\langle H_1,T_1(\theta^\star)\rangle$ is an odd function, their product is odd, and its expectation under the standard Gaussian measure vanishes. Hence,
\begin{align}
    \label{eq:diag_case_blockIobs}
    \big\langle h,\mathcal{I}_{\mathrm{obs}}^{(W)}(\theta^\star;\beta)g\big\rangle =  \beta^{2k}\,\frac{1}{k!}\, \big\langle DT_k(\theta^\star)[h],DT_k(\theta^\star)[g]\big\rangle_F +O(\beta^{2k+2}),
    \qquad k=\ell.
\end{align}

\paragraph{Step 5: Off-diagonal orders.}
Assume $k\neq \ell$. Then the leading term
$\beta^{k+\ell}\mathbb{E}_\phi[A_h A_g]$ vanishes by orthogonality. It remains to analyze the coefficient of order $\beta^{k+\ell+1}$.
\begin{enumerate}
    \item $\mathbb{E}_\phi[A_hB_g]$ can be nonzero only if $k=\ell\pm1$;
    \item $\mathbb{E}_\phi[B_hA_g]$ can be nonzero only if $k=\ell\pm1$;
    \item $\mathbb{E}_\phi[\langle H_1,T_1(\theta^\star)\rangle A_h A_g]$ can be nonzero only if $|k-\ell|=1$, since the product of chaoses $1,k,\ell$ must contain the zero-th chaos.
\end{enumerate}
Therefore,
\begin{align}
    \label{eq:adjacent_case_blockIobs}
    \big\langle h,\mathcal{I}_{\mathrm{obs}}^{(W)}(\theta^\star;\beta)g\big\rangle = O(\beta^{k+\ell+1}),
    \qquad |k-\ell|=1,
\end{align}
while if $|k-\ell|\ge 2$, all order-$\beta^{k+\ell+1}$ terms vanish, so
\begin{align}
    \label{eq:separated_case_blockIobs}
    \big\langle h,\mathcal{I}_{\mathrm{obs}}^{(W)}(\theta^\star;\beta)g\big\rangle =  O(\beta^{k+\ell+2}),
    \qquad |k-\ell|\ge 2.
\end{align}

Combining \eqref{eq:diag_case_blockIobs}, \eqref{eq:adjacent_case_blockIobs}, and \eqref{eq:separated_case_blockIobs}, and using the arbitrariness of $h\in U_k$ and $g\in U_\ell$, yields \eqref{eq:block_fisher_layered}.

\subsection{Proof of Corollary~\ref{cor:dir_exponents_Iobs}}
\label{app:proof_dir_exponents_Iobs}

\paragraph{Step 1: Directional scaling.}
Write $h=\sum_{k=1}^{r_{\mathrm{loc}}} h_k$ with $h_k\in U_k$, and let $r=r(h)$~\eqref{eqn:def_r_h}. By definition of $r(h)$, we have $h\in V_r$ but $h\notin V_{r+1}$. Since the filtration decomposition gives $V_r=\bigoplus_{k\ge r} U_k$ and $V_{r+1}=\bigoplus_{k\ge r+1} U_k$, it follows that $h_k=0$ for $k<r$ and $h_r\neq 0$. Moreover, for every $k>r$, $h_k\in V_{r+1}$, and therefore $DT_r(\theta^\star)[h_k]=0$. Consequently,
\begin{align}
    DT_r(\theta^\star)[h]=DT_r(\theta^\star)[h_r].
\end{align}
By Proposition~\ref{prop:fisher_layer_block_asymptotics}, the only contribution at order $\beta^{2r}$ comes from the $(r,r)$-block, while all remaining terms are $O(\beta^{2r+2})$. Therefore,
\begin{align}
    h^\top \mathcal{I}_{\mathrm{obs}}^{(W)}(\theta^\star;\beta) h = \frac{\beta^{2r}}{r!}\, \|DT_r(\theta^\star)[h_r]\|_F^2 + O(\beta^{2r+2}).
\end{align}
Since $DT_r(\theta^\star)[h_r]=DT_r(\theta^\star)[h]$ and $\beta^2=\mathrm{SNR}$, this yields
\begin{align}
    h^\top \mathcal{I}_{\mathrm{obs}}^{(W)}(\theta^\star;\beta) h = \frac{\|DT_{r(h)}(\theta^\star)[h]\|_F^2}{r(h)!}\, \mathrm{SNR}^{\,r(h)} + O\bigl(\mathrm{SNR}^{\,r(h)+1}\bigr).
    \label{eqn:app_B72}
\end{align}

\paragraph{Step 2: Minimal eigenvalue scaling.}
Since $r_{\mathrm{loc}}<\infty$, by Definition~\ref{def:rloc}, there exists $\hat{h}\in W$ with $\|\hat{h}\|=1$ such that $r(\hat{h})=r_{\mathrm{loc}}$. By~\eqref{eqn:app_B72}, applied to $h=\hat{h}$, 
\begin{align}
    \label{eq:hhat_rayleigh}    \hat{h}^\top\mathcal{I}_{\mathrm{obs}}^{(W)}(\theta^\star;\beta)\hat{h} = \Theta(\mathrm{SNR}^{\,r(\hat{h})}) = \Theta(\mathrm{SNR}^{\,r_{\mathrm{loc}}}).
\end{align}
Since $\mathcal{I}_{\mathrm{obs}}(\theta^\star;\beta)\succeq 0$, the smallest eigenvalue of its restriction to $W$ satisfies
\begin{align}
    \label{eq:lmin_rayleigh_use}  \lambda_{\min}\!\Big(\mathcal{I}_{\mathrm{obs}}^{(W)}(\theta^\star;\beta)\Big) = \min_{\|h\|=1,\,h\in W} h^\top\mathcal{I}_{\mathrm{obs}}^{(W)}(\theta^\star;\beta)h \ \le\ \hat{h}^\top\mathcal{I}_{\mathrm{obs}}^{(W)}(\theta^\star;\beta)\hat{h}.
\end{align}

By the layer construction,
$DT_k(\theta^\star)$ is injective on $U_k$. Hence, since each $U_k$ is finite-dimensional,
there exists $c>0$ such that
\begin{align}
    \frac{1}{k!}\|DT_k(\theta^\star)[h_k]\|_F^2 \ge c\|h_k\|^2,
    \qquad h_k\in U_k,
\end{align}
for every $k=1,\ldots,r_{\mathrm{loc}}$

Using the uniform block expansion in Proposition~\ref{prop:fisher_layer_block_asymptotics}, the higher-order diagonal corrections and off-diagonal blocks can be absorbed into the leading diagonal terms for sufficiently small $\beta$. Therefore, for all $h\in W$,
\begin{align}
    h^\top \mathcal{I}_{\mathrm{obs}}^{(W)}(\theta^\star;\beta) h \ge c\sum_{k=1}^{r_{\mathrm{loc}}} \beta^{2k}\|h_k\|^2 .
\end{align}
Since $\beta^{2k}\ge \beta^{2r_{\mathrm{loc}}}$ for $0<\beta<1$ and $k\le r_{\mathrm{loc}}$, this gives
\begin{align}
    h^\top \mathcal{I}_{\mathrm{obs}}^{(W)}(\theta^\star;\beta) h \ge c\,\beta^{2r_{\mathrm{loc}}}\|h\|^2 = c\,\mathrm{SNR}^{\,r_{\mathrm{loc}}}\|h\|^2 .
\end{align}
Taking the minimum over all unit vectors $h\in W$ yields
\begin{align}
    \lambda_{\min}\!\big(\mathcal{I}_{\mathrm{obs}}^{(W)}(\theta^\star;\beta)\big) \ge c\,\mathrm{SNR}^{\,r_{\mathrm{loc}}}.    \label{eqn:minimial_eigen_lower_bound}
\end{align}
Combining \eqref{eq:lmin_rayleigh_use} with \eqref{eq:hhat_rayleigh} for the upper bound and~\eqref{eqn:minimial_eigen_lower_bound} for the lower bound yields $\lambda_{\min}\!\Big(\mathcal{I}_{\mathrm{obs}}^{(W)}(\theta^\star;\beta)\Big) = \Theta(\mathrm{SNR}^{\,r_{\mathrm{loc}}})$, as $\mathrm{SNR} \to 0$, which is \eqref{eq:lmin_Iobs_scale_main}.

\section{Efficiency of optimally weighted GMoM}
\label{app:gmom_efficiency}

\subsection{Noise-debiased moment features} \label{subsec:empirical_moments}

This subsection provides the Hermite-based construction underlying the empirical feature vector introduced in Section~\ref{subsec:empirical_features_main}. 
Let $Y$ denote an observation from \eqref{eq:meanscaling-model}. Recall that $T_j(\theta)$ denotes the order-$j$ signal moment tensor. The key identity is that Hermite tensors exactly remove the explicit Gaussian contribution of the additive noise.
Indeed, conditioning on $Z$, applying Lemma~\ref{lem:Wick_identity_prelim} to the Gaussian shift $Y=\beta A(Z)\theta+\xi$, and comparing coefficients in the generating-function expansion yields, for every $j\ge 1$,
\begin{align}
    \label{eq:Hermite_signal_identity} \mathbb{E}_{Y\sim p_{\theta,\beta}}\!\left[H_j\!\left(Y\right) \right] = \beta^j T_j(\theta).
\end{align}
Thus, $H_j(Y)$ is the canonical noise-debiased observable associated with the order-$j$ signal moment. This motivates the noise-debiased empirical moments used in the main text.

\begin{definition}[Noise-debiased empirical signal moments]
\label{def:empirical_signal_moments_hermite}
For each $j=1,\ldots,L$, define
\begin{align}
    \label{eq:Hermite_empirical_signal_moment_identity} 
    \beta^j \widehat{T}_{j,n} \triangleq \frac{1}{n}\sum_{i=1}^n H_j(y_i).
\end{align}
\end{definition}

To pass from the debiased tensors to the feature vector used in the GMoM criterion, we stack the Hermite observables up to order $L$,
\begin{align}
    \label{def:stacked_Hermite_feature_map}
    \psi(y)
    \triangleq
    \begin{bmatrix} H_1(y)\\ H_2(y)\\ \vdots\\    H_L(y)
    \end{bmatrix}.
\end{align}
Then, the empirical feature vector from \eqref{eq:def_Mn_maintext} is exactly the empirical mean of $\psi(y_i)$, and its population counterpart in \eqref{eq:def_Mtheta_maintext} is $ \mathbb{E}_{Y\sim p_{\theta,\beta}}[\psi(Y)]$. In particular, \eqref{eq:Hermite_signal_identity} yields
\begin{align}
     \mathbb{E}_{Y\sim p_{\theta,\beta}}[\psi(Y)] =
    \begin{bmatrix}
        \beta \,T_1(\theta)\\   \beta^2\,T_2(\theta)\\
        \vdots\\        \beta^L\,T_L(\theta)
    \end{bmatrix}.
\end{align}
This is precisely the normalized truncated moment map used throughout the GMoM construction.

\subsection{Low-SNR limiting covariance}
\label{subsec:low_snr_covariance_matrix}

The low-SNR comparison between the Fisher information and the GMoM information matrices is governed by the pure-noise limit of the moment covariance. Recall the definition of the moment covariance matrix $\Sigma_L(\theta^\star, \beta)$ from~\eqref{eqn:def_Sigma_L}. Then, its limit as $\beta \to 0$ is given explicitly by,
\begin{align}
    \label{eq:def_sigma_GMoM_updated}
    \Sigma_L^{(0)} & \triangleq \Cov\bigl(\psi(\xi)\bigr) 
    \\ & = \mathbb{E}\!\Big[\bigl(\psi(\xi)-\mathbb{E}[\psi(\xi)]\bigr)\bigl(\psi(\xi)-\mathbb{E}[\psi(\xi)]\bigr)^\top\Big] \in \mathbb{R}^{N_{\leq L}\times N_{\leq L}},
\end{align}
where $\xi\sim\mathcal{N}(0,I_d)$, and $\psi$ is the stacked Hermite feature map from~\eqref{def:stacked_Hermite_feature_map}. Since $\mathbb{E}[H_j(\xi)]=0$ for every $j\ge 1$, we have $\mathbb{E}[\psi(\xi)]=0$, and therefore
\begin{align}
    \Sigma_L^{(0)} = \mathbb{E}\big[\psi(\xi)\psi(\xi)^\top\big].
    \label{eqn:app_D7}
\end{align}

The next proposition gives a closed-form expression for $\Sigma_L^{(0)}$. 

\begin{proposition}[Closed form of the Gaussian limiting covariance]
\label{prop:sigma0_closed_form_updated}
Let $\xi\sim\mathcal{N}(0,I_d)$, and represent each order-$j$ tensor block in an orthonormal basis with respect to the Frobenius inner product, under the symmetric-tensor convention. Then the Gaussian limiting covariance $\Sigma_L^{(0)}$ of the stacked Hermite feature vector is block diagonal with respect to $\bigoplus_{j=1}^L\mathbb{R}^{N_j}$, and is given by
\begin{align}
    \Sigma_L^{(0)} = \operatorname{diag}\bigl( 1!I_{N_1},\, 2!I_{N_2},\, \ldots,\, L!I_{N_L} \bigr), \label{eq:sigma0_closed_form_orthonormal}
\end{align}
where $N_j=\binom{d+j-1}{j}$ is the dimension of the space of symmetric tensors of order $j$.
\end{proposition}

\begin{proof}[Proof of Proposition~\ref{prop:sigma0_closed_form_updated}]
For $u\in \mathbb{R}^{N_j}$ and $v\in \mathbb{R}^{N_k}$, the Gaussian orthogonality of Hermite tensors (Lemma~\ref{lem:Hermite_orthogonality_prelim}) gives
\begin{align}
    \mathbb{E}\!\left[ \langle H_j(\xi),u\rangle \langle H_k(\xi),v\rangle \right] = \delta_{jk}\,j!\,\langle u,v\rangle_F .
    \label{eqn:app_D9}
\end{align}
Combining Eqs.~\eqref{def:stacked_Hermite_feature_map}, ~\eqref{eqn:app_D7}, and~\eqref{eqn:app_D9} yields~\eqref{eq:sigma0_closed_form_orthonormal}.
\end{proof}

\begin{corollary}
\label{cor:sigma_positive_definite_updated}
The matrix $\Sigma_L^{(0)}$ is positive definite.
\end{corollary}

\begin{proof}[Proof of Corollary~\ref{cor:sigma_positive_definite_updated}]
Since $\Sigma_L^{(0)}$ is block diagonal with diagonal blocks $j!I_{N_j}$, for $j=1,\ldots,L$, and each block $j!I_{N_j}$ is positive definite, $\Sigma_L^{(0)}$ is positive definite.
\end{proof}

Since $\Sigma_L(\theta^\star;\beta)\to \Sigma_L^{(0)}$ as $\beta\to 0$, the positivity of $\Sigma_L^{(0)}$ implies that $\Sigma_L(\theta^\star;\beta)$ remains invertible for all sufficiently small $\beta$.

\begin{proof}[Proof of Lemma~\ref{lem:sigma_beta_positive_definite_small_beta}]
Since $\Sigma_L^{(0)}\succ 0$, one has $\lambda_{\min}(\Sigma_L^{(0)})>0$. By Weyl's inequality,
\begin{align}
    \lambda_{\min}\bigl(\Sigma_L(\theta^\star;\beta)\bigr) \ge \lambda_{\min}\bigl(\Sigma_L^{(0)}\bigr) - \bigl\|\Sigma_L(\theta^\star;\beta)-\Sigma_L^{(0)}\bigr\|_{\mathrm{op}}.
\end{align}
Because $\Sigma_L(\theta^\star;\beta)\to \Sigma_L^{(0)}$ in operator norm as $\beta\to 0$, the right-hand side is strictly positive for all sufficiently small $\beta$. Therefore $\Sigma_L(\theta^\star;\beta)\succ 0$ for all sufficiently small $\beta$.
\end{proof}

\subsection{Consistency and asymptotic normality}
\label{subsec:GMoM_consistency_iormality}

This covariance is also the asymptotic covariance of the empirical feature vector, as recorded in the following proposition.

\begin{proposition}[Asymptotic normality of the empirical feature vector]
\label{prop:empirical_feature_clt_main}
Fix $\beta>0$ and $L\ge 1$. Under Assumption~\ref{ass:integrability_max},
\begin{align}
    \sqrt{n}\,\bigl(M_n(\mathcal{Y})-M(\theta^\star;\beta)\bigr) \xrightarrow[]{\mathcal{D}} \mathcal{N}\bigl(0,\Sigma_L(\theta^\star;\beta)\bigr).
\end{align}
\end{proposition}

\begin{proof}
By \eqref{eq:def_Mn_maintext}, the vector $M_n(\mathcal{Y})$ is the empirical mean of the i.i.d.\ random vectors $\psi(y_i)$. Under Assumption~\ref{ass:integrability_max}, all moments of $Y$ are finite, and therefore $\psi(Y)$ has finite second moment. The conclusion follows from the multivariate central limit theorem.

\end{proof}

For each sufficiently small $\beta>0$, we study the estimator~\eqref{eq:def_GMoM_estimator_main} under the weighting regime $\Omega_n \xrightarrow[]{\mathbb{P}} \Sigma_L(\theta^\star;\beta)^{-1}$, where $\Sigma_L(\theta^\star;\beta)$ is defined in~\eqref{eqn:def_Sigma_L}. By Lemma~\ref{lem:sigma_beta_positive_definite_small_beta}, proved in Appendix~\ref{subsec:low_snr_covariance_matrix}, this covariance matrix is positive definite for all sufficiently small $\beta$, so the construction is well defined.

\begin{lem}[Positivity for sufficiently small $\beta$]
\label{lem:sigma_beta_positive_definite_small_beta}
There exists $\beta_0>0$ such that $\Sigma_L(\theta^\star;\beta)\succ 0$ for all $0<\beta<\beta_0$.
\end{lem}

Thus, for sufficiently small $\beta$, the covariance matrix $\Sigma_L(\theta^\star;\beta)$ is invertible, and the remaining requirement for consistency is identifiability by the selected moments. The next proposition distinguishes between global consistency, which requires the truncated moment map to identify the true equivalence class globally, and local consistency, which only requires local identifiability near $[\theta^\star]$. The proof is deferred to Appendix~\ref{sec:proof_prop_GMoM_model_consistency}.

\begin{proposition}[Consistency of the small-$\beta$ GMoM estimator]
\label{prop:GMoM_model_consistency}
Assume Assumption~\ref{ass:integrability_max} holds. Let $\widehat{\theta}_{\mathrm{GMoM}}$ be defined by~\eqref{eq:def_GMoM_estimator_main}, let $\beta>0$ be sufficiently small so that $\Sigma_L(\theta^\star;\beta)$ is invertible, and suppose that $\Omega_n\xrightarrow[]{\mathbb{P}}\Sigma_L(\theta^\star;\beta)^{-1}$. Then, the following statements hold. 

\begin{enumerate}
    \item \emph{(Global consistency).} If $L\ge r_{\mathrm{glob}}$, then as $n \to \infty$,
    \begin{align}
        d_{\mathrm{eq}}\bigl(\widehat\theta_{\mathrm{GMoM}},\theta^\star\bigr) \xrightarrow[]{\mathbb{P}}0 .
    \end{align}

    \item \emph{(Local consistency).} If $L\ge r_{\mathrm{loc}}(\theta^\star)$, then there exists a compact neighborhood $\mathcal{U}$ of $\theta^\star$  such that any local GMoM estimator $\widehat\theta_{\mathrm{GMoM}}^{\mathrm{loc}} \in \argmin_{\theta\in\mathcal{U}} Q_n(\theta; \mathcal{Y})$ satisfies
    \begin{align}
        d_{\mathrm{eq}}\bigl(\widehat\theta_{\mathrm{GMoM}}^{\mathrm{loc}},\theta^\star\bigr) \xrightarrow[]{\mathbb{P}}0 .
    \end{align}
\end{enumerate}
\end{proposition}

The next theorem states the asymptotic normality of the optimally weighted GMoM estimator on the identifiable normal space $W$. It is the quotient-adapted specialization of the standard GMoM asymptotic normality result, with the moment-condition covariance $\Sigma_L(\theta^\star;\beta)$ and the restricted Jacobian $D_WM(\theta^\star;\beta)$ defined above. The proof is deferred to Appendix~\ref{app:proof_GMoM_model_normality}.

\begin{thm}[Asymptotic normality of the low-SNR GMoM estimator]
\label{thm:GMoM_model_normality}
Assume the conditions of Proposition~\ref{prop:GMoM_model_consistency} hold, and assume that $L\ge r_{\mathrm{loc}}(\theta^\star)$. Then, after choosing a representative of $\widehat\theta_{\mathrm{GMoM}}$ that converges to $\theta^\star$,
\begin{align}
    \sqrt{n}\,\Pi_W\bigl(\widehat\theta_{\mathrm{GMoM}}-\theta^\star\bigr) \xrightarrow[]{\mathcal{D}} \mathcal{N}\bigl(0,\Sigma_{\mathrm{GMoM}}(\theta^\star;\beta)\bigr),
\end{align}
where $\Sigma_{\mathrm{GMoM}}(\theta^\star;\beta) \triangleq \bigl(\mathcal{I}_{\mathrm{GMoM}}^{(W)}(\theta^\star;\beta)\bigr)^{-1}$ is defined in~\eqref{eqn:asymptotic_covariance_GMoM}, and $\mathcal{I}_{\mathrm{GMoM}}^{(W)}(\theta^\star;\beta)$ is the restricted GMoM information matrix defined in~\eqref{eq:def_GMoM_information}.
\end{thm}

Thus, the asymptotic covariance is the inverse of the local quadratic information induced by the optimally weighted moment criterion on $W$. The remaining task is to compare this GMoM information, in the low-SNR regime, with the observed Fisher information restricted to $W$.

\subsubsection{Proof of Proposition~\ref{prop:GMoM_model_consistency}}
\label{sec:proof_prop_GMoM_model_consistency}

By construction, $M_n(\mathcal{Y})\xrightarrow[]{\mathbb{P}}M(\theta^\star;\beta)$. Under Assumption~\ref{ass:integrability_max}, the map $\theta\mapsto M(\theta;\beta)$ is continuous. Since $\Omega_n\xrightarrow[]{\mathbb{P}}\Sigma_L(\theta^\star;\beta)^{-1}$, the sample criterion converges uniformly in probability, on every compact subset of $\Theta$, to
\begin{align}
    Q(\theta) = \bigl(M(\theta^\star;\beta)-M(\theta;\beta)\bigr)^\top\Sigma_L(\theta^\star;\beta)^{-1} \bigl(M(\theta^\star;\beta)-M(\theta;\beta)\bigr). \label{eq:GMoM_population_criterion_consistency}
\end{align}
Since $\Sigma_L(\theta^\star;\beta)^{-1}$ is positive definite, $Q(\theta)=0$ if and only if $M(\theta;\beta)=M(\theta^\star;\beta)$.

For the global statement, assume $L\ge r_{\mathrm{glob}}$. Then the selected population moments identify the equivalence class of $\theta^\star$, so $M(\theta;\beta)=M(\theta^\star;\beta)$ implies $\theta\sim\theta^\star$. Hence the population criterion is minimized over $\Theta$ precisely on $[\theta^\star]$. Since $\Theta$ is compact and $Q_n$ converges uniformly in probability to $Q$, the standard extremum-estimator argument~\cite{van2000asymptotic} yields $d_{\mathrm{eq}}\bigl(\widehat\theta_{\mathrm{GMoM}},\theta^\star\bigr) \xrightarrow[]{\mathbb{P}}0$.

For the local statement, assume $L\ge r_{\mathrm{loc}}(\theta^\star)$. By definition of $r_{\mathrm{loc}}(\theta^\star)$, the derivative of the selected moment map is injective on the normal space $W$. Equivalently, after passing to a local transversal to the equivalence class, the differential of the moment map has full rank at $\theta^\star$. The inverse function theorem therefore implies that there exists a sufficiently small neighborhood $\mathcal{U}$ of $[\theta^\star]$ such that
\begin{align}
    M(\theta;\beta)=M(\theta^\star;\beta), \qquad \theta\in\mathcal{U},
    \quad\Longrightarrow\quad
    \theta\sim\theta^\star .
\end{align}
Thus, the population criterion $Q$ is minimized over $\mathcal{U}$ precisely on $[\theta^\star]\cap\mathcal{U}$. Applying the same extremum-estimator argument on the compact neighborhood $\mathcal{U}$ gives $d_{\mathrm{eq}}\bigl(\widehat\theta_{\mathrm{GMoM}}^{\mathrm{loc}},\theta^\star\bigr) \xrightarrow[]{\mathbb{P}}0$.
This proves both claims.

\subsubsection{Proof of Theorem~\ref{thm:GMoM_model_normality}}
\label{app:proof_GMoM_model_normality}

We begin with an auxiliary lemma.
\begin{lem}[Rate of the local GMoM estimator]
\label{lem:GMoM_root_n_rate}
Assume the conditions of Theorem~\ref{thm:GMoM_model_normality}. 
Then, 
\begin{align}
    \label{eq:GMoM_root_n_rate} \Pi_W(\widehat\theta_{\mathrm{GMoM}}-\theta^\star) = O_{\mathbb{P}}(n^{-1/2}).
\end{align}
\end{lem}

\begin{proof}[Proof of Lemma~\ref{lem:GMoM_root_n_rate}]
Since $\widehat\theta_{\mathrm{GMoM}}$ is a local minimizer of the GMoM target function $Q_n$~\eqref{eq:def_GMoM_criterion_main}, the first-order optimality condition along $W$ gives
\begin{align}
    DM(\widehat\theta_{\mathrm{GMoM}};\beta)^\top \Omega_n \bigl(M_n(\mathcal{Y})-M(\widehat\theta_{\mathrm{GMoM}};\beta)\bigr) = 0.
    \label{eq:GMoM_root_rate_FOC}
\end{align}
A Taylor expansion of $M(\cdot;\beta)$ around $\theta^\star$ along the normal directions $W$ gives
\begin{align}
    M(\widehat\theta_{\mathrm{GMoM}};\beta) = M(\theta^\star;\beta) + DM(\theta^\star;\beta) \Pi_W(\widehat\theta_{\mathrm{GMoM}}-\theta^\star) + o_{\mathbb{P}}\!\left( \big\| \Pi_W(\widehat\theta_{\mathrm{GMoM}}-\theta^\star)\big\| \right). \label{eq:GMoM_root_rate_Taylor}
\end{align}
Substituting~\eqref{eq:GMoM_root_rate_Taylor} into
\eqref{eq:GMoM_root_rate_FOC} gives
\begin{align}
    &DM(\widehat\theta_{\mathrm{GMoM}};\beta)^\top \Omega_n \Big[ M_n(\mathcal{Y})-M(\theta^\star;\beta) - DM(\theta^\star;\beta) \Pi_W(\widehat\theta_{\mathrm{GMoM}}-\theta^\star)
    \nonumber\\
    &\hspace{4.5cm}
    + o_{\mathbb{P}}\!\left( \big\|\Pi_W(\widehat\theta_{\mathrm{GMoM}}-\theta^\star)\big\| \right) \Big] = 0. \label{eq:GMoM_root_rate_after_substitution}
\end{align}
Rearranging~\eqref{eq:GMoM_root_rate_after_substitution}, we obtain
\begin{align}
    &DM(\widehat\theta_{\mathrm{GMoM}};\beta)^\top \Omega_n DM(\theta^\star;\beta) \Pi_W(\widehat\theta_{\mathrm{GMoM}}-\theta^\star)
    \nonumber\\
    &\qquad =    DM(\widehat\theta_{\mathrm{GMoM}};\beta)^\top \Omega_n \bigl(M_n(\mathcal{Y})-M(\theta^\star;\beta)\bigr) + o_{\mathbb{P}}\!\left( \big\|\Pi_W(\widehat\theta_{\mathrm{GMoM}}-\theta^\star)\big\| \right). \label{eq:GMoM_root_rate_rearranged}
\end{align}
By consistency and continuity of the Jacobian, $DM(\widehat\theta_{\mathrm{GMoM}};\beta) = DM(\theta^\star;\beta)+o_{\mathbb{P}}(1)$.
Therefore, the left-hand side of~\eqref{eq:GMoM_root_rate_rearranged} equals
\begin{align}
    DM(\theta^\star;\beta)^\top \Omega_n DM(\theta^\star;\beta) \Pi_W(\widehat\theta_{\mathrm{GMoM}}-\theta^\star) + o_{\mathbb{P}}\!\left( \big\|\Pi_W(\widehat\theta_{\mathrm{GMoM}}-\theta^\star)\big\| \right),
\end{align}
while, using $M_n(\mathcal{Y})-M(\theta^\star;\beta)=O_{\mathbb{P}}(n^{-1/2})$ (Proposition~\ref{prop:empirical_feature_clt_main}), the first term on the right-hand side of~\eqref{eq:GMoM_root_rate_rearranged} equals
\begin{align}
    DM(\theta^\star;\beta)^\top \Omega_n \bigl(M_n(\mathcal{Y})-M(\theta^\star;\beta)\bigr) + o_{\mathbb{P}}(n^{-1/2}).
\end{align}
Combining these identities gives
\begin{align}
    &DM(\theta^\star;\beta)^\top \Omega_n DM(\theta^\star;\beta)    \Pi_W(\widehat\theta_{\mathrm{GMoM}}-\theta^\star)
    \nonumber\\
    &\qquad = DM(\theta^\star;\beta)^\top \Omega_n \bigl(M_n(\mathcal{Y})-M(\theta^\star;\beta)\bigr) + o_{\mathbb{P}}(n^{-1/2}) + o_{\mathbb{P}}\!\left( \big\|\Pi_W(\widehat\theta_{\mathrm{GMoM}}-\theta^\star)\big\| \right). \label{eq:GMoM_root_rate_prelinear}
\end{align}
Moreover, as $\Omega_n \to \Sigma_L(\theta^\star;\beta)^{-1}$,
\begin{align}
    DM(\theta^\star;\beta)^\top \Omega_n DM(\theta^\star;\beta) \xrightarrow{\mathbb{P}} \mathcal{I}_{\mathrm{GMoM}}^{(W)}(\theta^\star;\beta),    \label{eq:GMoM_information_limit}
\end{align}
and the limiting matrix is positive definite on $W$. Hence the left-hand matrix in \eqref{eq:GMoM_root_rate_prelinear} is invertible with probability tending to one and has inverse bounded in probability. Therefore, using $M_n(\mathcal{Y})-M(\theta^\star;\beta)=O_{\mathbb{P}}(n^{-1/2})$,
\eqref{eq:GMoM_root_rate_prelinear} implies
\begin{align}
    \big\| \Pi_W(\widehat\theta_{\mathrm{GMoM}}-\theta^\star) \big\| \le  O_{\mathbb{P}}(n^{-1/2}) + o_{\mathbb{P}}\!\left( \big\| \Pi_W(\widehat\theta_{\mathrm{GMoM}}-\theta^\star) \big\| \right).
\end{align}
Absorbing the second term into the left-hand side gives $\Pi_W(\widehat\theta_{\mathrm{GMoM}}-\theta^\star) = O_{\mathbb{P}}(n^{-1/2})$, as claimed.
\end{proof}

Now we return to the proof of Theorem~\ref{thm:GMoM_model_normality}. Since
$L\ge r_{\mathrm{loc}}(\theta^\star)$, every nonzero direction $h\in W$ is detected by some
$DT_j(\theta^\star)[h]$ with $j\le L$. Hence
$DM(\theta^\star;\beta)$ is injective on $W$. Since
$\Sigma_L(\theta^\star;\beta)$ is positive definite, the matrix
\begin{align}
    \mathcal{I}_{\mathrm{GMoM}}^{(W)}(\theta^\star;\beta) = D_WM(\theta^\star;\beta)^\top \Sigma_L(\theta^\star;\beta)^{-1} D_WM(\theta^\star;\beta)
\end{align}
is positive definite on $W$, and therefore invertible.

By Lemma~\ref{lem:GMoM_root_n_rate}, after choosing a representative of
$\widehat\theta_{\mathrm{GMoM}}$ in its equivalence class that converges to
$\theta^\star$,
\begin{align}
    \label{eq:GMoM_root_n_rate_used}    \Pi_W(\widehat\theta_{\mathrm{GMoM}}-\theta^\star) = O_{\mathbb{P}}(n^{-1/2}).
\end{align}
Following~\eqref{eq:GMoM_root_rate_prelinear}, and using~\eqref{eq:GMoM_root_n_rate_used}, we obtain
\begin{align}
    \label{eq:GMoM_linearized_normal_equation}
    DM(\theta^\star;\beta)^\top & \Omega_n DM(\theta^\star;\beta)\,    \Pi_W(\widehat\theta_{\mathrm{GMoM}}-\theta^\star)
    \\ & =
    DM(\theta^\star;\beta)^\top \Omega_n \bigl(M_n(\mathcal{Y})-M(\theta^\star;\beta)\bigr) + o_{\mathbb{P}}(n^{-1/2}).
\end{align}
Multiplying~\eqref{eq:GMoM_linearized_normal_equation} by $\sqrt{n}$ and using
\eqref{eq:GMoM_information_limit} yields
\begin{align}
    \label{eq:GMoM_asymptotic_linear_representation}
    \sqrt{n}\, \Pi_W(\widehat\theta_{\mathrm{GMoM}}-\theta^\star) = \mathcal{I}_{\mathrm{GMoM}}^{(W)}(\theta^\star;\beta)^{-1} DM(\theta^\star;\beta)^\top \Sigma_L(\theta^\star;\beta)^{-1} \sqrt{n}\,\bigl(M_n(\mathcal{Y})-M(\theta^\star;\beta)\bigr) + o_{\mathbb{P}}(1).
\end{align}
The conclusion follows from Proposition~\ref{prop:empirical_feature_clt_main} and Slutsky's theorem.

\subsection{Score projection onto the selected GMoM span}
\label{subsec:score_projection_selected_features}

We now connect the low-SNR score expansion to the moment features used in the GMoM criterion. The key idea is that the selected Hermite feature vector captures the Hermite chaoses of the score up to order $L$. Thus, the component of the score visible to GMoM is obtained by orthogonally projecting the score onto the linear span of these selected features. 

Recall the definition of $\psi(y)$ from~\eqref{def:stacked_Hermite_feature_map} and of $ DM(\theta^\star;\beta) \triangleq DM(\theta^\star;\beta)\big|_W$ from~\eqref{eqn:def_Jacobian_moments}. 
Fix $L\triangleq  L_{\mathrm{mom}}\ge 1$, and let
\begin{align}
    g(Y,\theta^{\star}) \triangleq \psi(Y)-M(\theta^\star;\beta) \, \in \mathbb{R}^{N_L}
\end{align}
denote the centered Hermite feature vector.

\begin{lem}[Score projection onto the selected GMoM span]
\label{lem:score_projection_selected_moments}
Recall the definition of the directional score $S_{\theta^\star,\beta}(y;h)$ from~\eqref{eq:score_dir2}. Assume Assumption~\ref{ass:integrability_max}, let $L\ge 1$, and fix $\theta^\star\in\Theta$. Then, for every $h\in W$,
\begin{align}
    S_{\theta^\star,\beta}(Y;h) = \Big\langle \Sigma_L(\theta^\star;\beta)^{-1}DM(\theta^\star;\beta)h,\, g(Y, \theta^\star) \Big\rangle + r_\beta(Y;h),
    \label{eqn:app_D12}
\end{align}
and
\begin{align}
     \mathbb{E}_{Y\sim p_{\theta^\star,\beta}}\!\Big[\Big\langle \Sigma_L(\theta^\star;\beta)^{-1}DM(\theta^\star;\beta)h,\, g(Y, \theta^\star) \Big\rangle r_\beta(Y;h) \Big] =0.
\end{align}
Moreover, uniformly for $h$ in bounded subsets of $W$,
\begin{align}
    \mathbb{E}_{Y\sim p_{\theta^\star,\beta}}\!\left[r_\beta(Y;h)^2\right]^{1/2} = O\!\left(\beta^{L+1}\|h\|\right),
\end{align}
\end{lem}

\begin{proof}[Proof of Lemma~\ref{lem:score_projection_selected_moments}]
Fix $h\in W$. By Lemma~\ref{lem:score_Hermite_truncation}, applied with $Y=y$, we have
\begin{align}
    S_{\theta^\star,\beta}(Y;h) = Q_{L,\beta}(Y;h)+\widetilde{R}_{L,\beta}(Y;h),
\end{align}
where
\begin{align}
    Q_{L,\beta}(Y;h) \triangleq \sum_{q=1}^{L} \beta^q \sum_{j=1}^{q} a_j(Y;h)\,b_{q-j}(Y),
\end{align}
and $\|\widetilde{R}_{L,\beta}(\cdot;h)\|_{L^2(\phi)} = O\left(\beta^{L+1}\|h\|\right)$, uniformly for $h$ in bounded subsets of $W$.

\paragraph{Step 1: Finite-order Hermite approximation.}
By Lemma~\ref{lem:score_Hermite_truncation}, for each $q\le L$, the coefficient $\sum_{j=1}^{q} a_j(\cdot;h)\,b_{q-j} \in  \bigoplus_{r=0}^{q}\mathcal{H}_r$ belongs to the direct sum of Hermite chaoses of orders at most $q$. Hence $Q_{L,\beta}(Y;h) \in \bigoplus_{q=0}^{L}\mathcal{H}_q$ belongs to the direct sum of Hermite chaoses of orders at most $L$. Since $\psi(Y)$, as defined in~\eqref{def:stacked_Hermite_feature_map}, contains the full coordinate representations of the Hermite tensors $H_1(Y),\dots,H_L(Y)$, every scalar Hermite polynomial of orders at most $L$ can be written as a linear combination of the coordinates of $\psi(Y)$ (see Lemma~\ref{lem:L2_Hermite_expansion_prelim}). Therefore $Q_{L,\beta}(Y;h)$ belongs to the linear span of the coordinates of $\psi(Y)$, and hence also of $g(Y, \theta^\star)=\psi(Y)-M(\theta^\star;\beta)$.

\paragraph{Step 2: Identification of the orthogonal projection.}
We now compute the orthogonal projection of the full score onto the corresponding centered
feature span. Since this span is generated by the coordinates of $g(Y, \theta^\star)$, the projection has the form $\big\langle \alpha_\beta(h),g(Y, \theta^\star)\big\rangle$ for some coefficient vector $\alpha_\beta(h)$. By the normal equations for orthogonal
projection in the mean-square inner product under $Y\sim p_{\theta^\star,\beta}$, the
residual must be orthogonal to every coordinate of $g(Y, \theta^\star)$, that is,
\begin{align}
    \label{eq:score_projection_normal_equations}
    \mathbb{E}_{Y\sim p_{\theta^\star,\beta}}\!\left[g(Y, \theta^\star)\Big(S_{\theta^\star,\beta}(Y;h) - \big\langle \alpha_\beta(h),g(Y, \theta^\star)\big\rangle \Big) \right] = 0 .
\end{align}
Expanding \eqref{eq:score_projection_normal_equations} gives
\begin{align}
    \label{eq:score_projection_normal_equations_expanded}
    \mathbb{E}_{Y\sim p_{\theta^\star,\beta}}\!\big[g(Y, \theta^\star)S_{\theta^\star,\beta}(Y;h) \big] = \mathbb{E}_{Y\sim p_{\theta^\star,\beta}}\!\big[ g(Y, \theta^\star)g(Y, \theta^\star)^\top \big]\alpha_\beta(h).
\end{align}
By the definition of the feature covariance matrix~\eqref{eqn:def_Sigma_L},
\begin{align}
    \label{eq:Z_covariance_equals_Sigma}
    \mathbb{E}_{Y\sim p_{\theta^\star,\beta}}\!\big[ g(Y, \theta^\star)g(Y, \theta^\star)^\top \big] =\Sigma_L(\theta^\star;\beta).
\end{align}
Substituting \eqref{eq:Z_covariance_equals_Sigma} into
\eqref{eq:score_projection_normal_equations_expanded} yields
\begin{align}
    \label{eq:normal_equations_Sigma_alpha}    \Sigma_L(\theta^\star;\beta)\,\alpha_\beta(h) = \mathbb{E}_{Y\sim p_{\theta^\star,\beta}}\!\big[g(Y, \theta^\star)S_{\theta^\star,\beta}(Y;h) \big].
\end{align}

It remains to identify the right-hand side of
\eqref{eq:normal_equations_Sigma_alpha}. Since $M(\theta;\beta) = \mathbb{E}_{Y\sim p_{\theta,\beta}}[\psi(Y)]$, differentiating $M(\theta;\beta)$ in the direction $h \in W$ at
$\theta=\theta^\star$ gives the score identity
\begin{align}
    \label{eq:DM_score_identity_projection_proof}
    DM(\theta^\star;\beta)h = \mathbb{E}_{Y\sim p_{\theta^\star,\beta}}\!\big[ \psi(Y)S_{\theta^\star,\beta}(Y;h) \big].
\end{align}
On the other hand, using $g(Y, \theta^\star)=\psi(Y)-M(\theta^\star;\beta)$, we have
\begin{align}
    \label{eq:Z_score_expand_projection_proof}
    \mathbb{E}_{Y\sim p_{\theta^\star,\beta}}\!\big[g(Y, \theta^\star)S_{\theta^\star,\beta}(Y;h) \big]
    &=
    \mathbb{E}_{Y\sim p_{\theta^\star,\beta}}\!\big[\psi(Y)S_{\theta^\star,\beta}(Y;h)\big] \nonumber\\
    &\quad
    - M(\theta^\star;\beta) \mathbb{E}_{Y\sim p_{\theta^\star,\beta}}\!\big[ S_{\theta^\star,\beta}(Y;h)\big].
\end{align}
The score has mean zero $\mathbb{E}_{Y\sim p_{\theta^\star,\beta}}\![ S_{\theta^\star,\beta}(Y;h)] =  0$.
Combining with~\eqref{eq:DM_score_identity_projection_proof}--\eqref{eq:Z_score_expand_projection_proof}, we obtain
\begin{align}
    \label{eq:Z_score_equals_DM_projection_proof}
    \mathbb{E}_{Y\sim p_{\theta^\star,\beta}}\!\big[g(Y, \theta^\star)S_{\theta^\star,\beta}(Y;h) \big] = DM(\theta^\star;\beta)h.
\end{align}
Substituting \eqref{eq:Z_score_equals_DM_projection_proof} into
\eqref{eq:normal_equations_Sigma_alpha} gives
\begin{align}
    \label{eq:alpha_beta_solution_projection_proof}
    \alpha_\beta(h) = \Sigma_L(\theta^\star;\beta)^{-1} DM(\theta^\star;\beta)h,
\end{align}
which proves~\eqref{eqn:app_D12}.

\paragraph{Step 3: Bound on the projection error.}
Since orthogonal projection is the best approximation in the mean-square norm under the true observation law, we have
\begin{align}
    \mathbb{E}_{Y\sim p_{\theta^\star,\beta}}\!\left[r_\beta(Y;h)^2\right]^{1/2} \le \mathbb{E}_{Y\sim p_{\theta^\star,\beta}} \!\left[\widetilde{R}_{L,\beta}(Y;h)^2\right]^{1/2} = O\!\left(\beta^{L+1}\|h\|\right),
\end{align}
uniformly for $h$ in bounded subsets of $W$.

\end{proof}

\subsection{Proof of Theorem~\ref{thm:H_fixed_beta_fisher_moment_identity}}
\label{app:proof_H_fixed_beta_fisher_mo}

Define $R_\beta \triangleq \mathcal{I}_{\mathrm{obs}}^{(W)}(\theta^\star;\beta) -\mathcal{I}_{\mathrm{GMoM}}^{(W)}(\theta^\star;\beta)$.
Since both $\mathcal{I}_{\mathrm{obs}}^{(W)}(\theta^\star;\beta)$ and $\mathcal{I}_{\mathrm{GMoM}}^{(W)}(\theta^\star;\beta)$ are symmetric operators on $W$, so is $R_\beta$.

\paragraph{Step 1: Orthogonal decomposition of the score.}
Fix $h\in W$ and recall the definition of the directional score $S_{\theta^\star,\beta}(y;h)$ from~\eqref{eq:score_dir2}. By~\eqref{eqn:Fisher_matrix_to_directional_score_relation},
\begin{align}
    h^\top \mathcal{I}_{\mathrm{obs}}^{(W)}(\theta^\star;\beta)h =  \mathbb{E}_{Y\sim p_{\theta^\star,\beta}}\!\left[S_{\theta^\star,\beta}(Y;h)^2\right].
    \label{eqn:app_D47}
\end{align}
By the definition of the feature covariance matrix~\eqref{eqn:def_Sigma_L}, we have
\begin{align}
    \label{eqn:cov_relation}     \mathbb{E}_{Y\sim p_{\theta^\star,\beta}}\!\big[g(Y, \theta^\star)g(Y, \theta^\star)^\top\big] = \Sigma_L(\theta^\star;\beta).
\end{align}
where $g(Y, \theta^\star)\triangleq \psi(Y)-M(\theta^\star;\beta)$.

Set $a_h \triangleq \Sigma_L(\theta^\star;\beta)^{-1}DM(\theta^\star;\beta)h$. By Lemma~\ref{lem:score_projection_selected_moments},
\begin{align}
    S_{\theta^\star,\beta}(Y;h) = \langle a_h,\, g(Y, \theta^\star) \rangle + r_\beta(Y;h),
    \label{eqn:app_D49}
\end{align}
where $\mathbb{E}_{Y\sim p_{\theta^\star,\beta}}\!\big[\langle a_h,\, g(Y, \theta^\star) \rangle r_\beta(Y;h) \big] =0$, and $\|r_\beta(\cdot;h)\|_{L^2} = O\left(\beta^{L+1}\|h\|\right)$.

\paragraph{Step 2: Information discrepancy.}
Since $g(Y, \theta^\star)$ is centered and has covariance given by~\eqref{eqn:cov_relation}, the second moment of the scalar random variable $\langle a_h,g(Y, \theta^\star)\rangle$ is
\begin{align}
     \mathbb{E}_{Y\sim p_{\theta^\star,\beta}}\!\left[ \langle a_h,g(Y, \theta^\star)\rangle^2\right]
    &=
    a_h^\top \Sigma_L(\theta^\star;\beta)\,a_h
    \\
    &=
    h^\top DM(\theta^\star;\beta)^\top \Sigma_L(\theta^\star;\beta)^{-1} DM(\theta^\star;\beta)h
    \\
    &=
    h^\top \mathcal{I}_{\mathrm{GMoM}}^{(W)}(\theta^\star;\beta)h.
\end{align}
Combining~\eqref{eqn:app_D47} and~\eqref{eqn:app_D49}, and using the orthogonality of the decomposition,
\begin{align}
    h^\top \mathcal{I}_{\mathrm{obs}}^{(W)}(\theta^\star;\beta)h
    &= \mathbb{E}_{Y\sim p_{\theta^\star,\beta}}\!\left[ \langle a_h,g(Y, \theta^\star)\rangle^2 \right] +  \mathbb{E}_{Y\sim p_{\theta^\star,\beta}}\!\left[r_\beta(Y;h)^2 \right]
    \\
    &=
    h^\top \mathcal{I}_{\mathrm{GMoM}}^{(W)}(\theta^\star;\beta)h +  \mathbb{E}_{Y\sim p_{\theta^\star,\beta}}\!\left[ r_\beta(Y;h)^2 \right].
\end{align}
Hence
\begin{align}
    h^\top R_\beta h =  \mathbb{E}_{Y\sim p_{\theta^\star,\beta}}\!\left[ r_\beta(Y;h)^2 \right] = O\left(\beta^{2L+2}\|h\|^2\right).
\end{align}
In particular, $R_\beta$ is positive semidefinite. Since $R_\beta$ is symmetric, the operator norm gives
\begin{align}
    \|R_\beta\|_{\mathrm{op}} = \sup_{\|h\|=1} h^\top R_\beta h = O\left(\beta^{2L+2}\right).
\end{align}
This proves the first part of the claim.

\paragraph{Step 3: Relative inverse-information bound.}
It remains to prove the relative inverse-information statement. By the first part of the theorem, $\mathcal{I}_{\mathrm{GMoM}}^{(W)}(\theta^\star;\beta) = \mathcal{I}_{\mathrm{obs}}^{(W)}(\theta^\star;\beta) - R_\beta$. Multiplying from the left and right by $(\mathcal{I}_{\mathrm{obs}}^{(W)}(\theta^\star;\beta))^{-1/2}$ gives
\begin{align}
    & \notag \left(\mathcal{I}_{\mathrm{obs}}^{(W)}(\theta^\star;\beta)\right)^{-1/2} \left(\mathcal{I}_{\mathrm{GMoM}}^{(W)}(\theta^\star;\beta)\right) \left(\mathcal{I}_{\mathrm{obs}}^{(W)}(\theta^\star;\beta)\right)^{-1/2} 
    \\ & \qquad \qquad = I_W-\left(\mathcal{I}_{\mathrm{obs}}^{(W)}(\theta^\star;\beta)\right)^{-1/2} R_\beta \left(\mathcal{I}_{\mathrm{obs}}^{(W)}(\theta^\star;\beta)\right)^{-1/2}.
    \label{eqn:app_C59}
\end{align}
Since $\|R_\beta\|_{\mathrm{op}} = O\left(\mathrm{SNR}^{L+1}\right)$ and, by Corollary~\ref{cor:dir_exponents_Iobs} $\lambda_{\min} (\mathcal{I}_{\mathrm{obs}}^{(W)}(\theta^\star;\beta)) = \Theta ( \mathrm{SNR}^{r_{\mathrm{loc}}(\theta^\star)})$, we have
\begin{align}
    & \notag \left\|\left(\mathcal{I}_{\mathrm{obs}}^{(W)}(\theta^\star;\beta)\right)^{-1/2} R_\beta \left(\mathcal{I}_{\mathrm{obs}}^{(W)}(\theta^\star;\beta)\right)^{-1/2}\right\|_{\mathrm{op}} \\ & \qquad \qquad \le \left\| \left(\mathcal{I}_{\mathrm{obs}}^{(W)}(\theta^\star;\beta)\right)^{-1} \right\|_{\mathrm{op}} \|R_\beta\|_{\mathrm{op}} = O\left(\mathrm{SNR}^{L+1-r_{\mathrm{loc}}(\theta^\star)}\right).
    \label{eqn:app_C60}
\end{align}
If $L \ge r_{\mathrm{loc}}(\theta^\star)$, this bound is $o(1)$. Therefore, taking the inverse of~\eqref{eqn:app_C59}, and using the bound from~\eqref{eqn:app_C60}, together with the Neumann expansion, gives
\begin{align}
    \left(\mathcal{I}_{\mathrm{obs}}^{(W)}(\theta^\star;\beta)\right)^{1/2} \left(\mathcal{I}_{\mathrm{GMoM}}^{(W)}(\theta^\star;\beta)\right)^{-1} \left(\mathcal{I}_{\mathrm{obs}}^{(W)}(\theta^\star;\beta)\right)^{1/2} = I_W+\widetilde{R}_\beta.
\end{align}
where $\widetilde{R}_\beta$ is self-adjoint and satisfies $\|\widetilde{R}_\beta\|_{\mathrm{op}} = O(\mathrm{SNR}^{L+1-r_{\mathrm{loc}}(\theta^\star)}) = o(1)$, as claimed.

\subsection{Proof of Lemma~\ref{lem:weighting_matrix}}
\label{app:proof_lem_weighting_matrix}

Standard covariance concentration in fixed dimension gives  $\| \widehat{\Sigma}_{L,n} -\Sigma_L(\theta^\star;\beta) \|_{\mathrm{op}} =O_{\mathbb P}(n^{-1/2})$. By~\eqref{eq:zero_signal_covariance}, there exists $\kappa>0$ such that $\lambda_{\min}(\Sigma_L(\theta^\star;\beta))\ge\kappa$ for all sufficiently small $\beta$. Consequently, $\|(\widehat{\Sigma}_{L,n}+\lambda_n I)^{-1}\|_{\mathrm{op}}=O_{\mathbb P}(1)$. The resolvent identity then yields
\begin{align}
    \Omega_n-\Sigma_L^{-1} =-\Omega_n \left(\widehat{\Sigma}_{L,n}-\Sigma_L+\lambda_n I\right) \Sigma_L^{-1},
\end{align}
where $\Sigma_L=\Sigma_L(\theta^\star;\beta)$, proving~\eqref{eq:empirical_weighting_rate}.

\section{Efficiency loss of non-optimal weighted GMoM}
\label{app:eff_loss}

\subsection{Proof of Proposition~\ref{prop:block_isotropic_weighting}}
\label{subsec:proof_of_block_isotropic_weighting}

\paragraph{Step 1: GMoM covariance comparison.}
Recall the covariance-loss decomposition~\eqref{eq:nonoptimal_total_loss}. By the standard GMoM covariance comparison~\cite{hansen1982large,newey1994large}, the additional loss due to nonoptimal weighting satisfies
\begin{align}
    \notag \Delta_\Omega & = \Sigma_{\mathrm{GMoM}}^{(\Omega)}-\Sigma_{\mathrm{GMoM}} ^{(\mathrm{opt})}
    \\ &=(K_\Omega-K_{\mathrm{opt}})\Sigma_L(K_\Omega-K_{\mathrm{opt}})^\top \succeq 0, \label{eq:nonoptimal_gap_identity}
\end{align}
where, with $D_WM=D_WM(\theta^\star;\beta)$ and $\Sigma_L=\Sigma_L(\theta^\star;\beta)$,
\begin{align}
    K_\Omega
    &=(D_WM^\top\Omega D_WM)^{-1}D_WM^\top\Omega,\\
    K_{\mathrm{opt}}
    &=(D_WM^\top\Sigma_L^{-1}D_WM)^{-1} D_WM^\top\Sigma_L^{-1}.
\end{align}
We show that the two linear estimators agree to leading order after rescaling the informative layers. Equation~\eqref{eq:nonoptimal_gap_identity} will then give a quadratic bound on the weighting loss.

\paragraph{Step 2: Informative-layer decomposition.}
Write $r=r_{\mathrm{loc}}(\theta^\star)$. By~\eqref{eq:def_Mtheta_maintext} and the informative-layer definition~\ref{def:residual_subspaces_layers}, $DT_j(\theta^\star)|_{U_k}=0$ whenever $j<k$. Hence
\begin{align}
    D_WM(\theta^\star;\beta)
    \left(\sum_{k=1}^{r}\beta^{-k}\Pi_{U_k}\right)
    &=B+O(\beta),
    \label{eq:weighted_proof_Ht}
\end{align}
where $B:W\to\bigoplus_{j=1}^L\mathbb{R}^{N_j}$ has the $k$-th diagonal block $B_k=DT_k(\theta^\star)|_{U_k}$ for $1\le k\le r$, and all other blocks zero. Each $B_k$ is injective on $U_k$, so $B$ has full column rank. 

\paragraph{Step 3: Leading-order equivalence of the weighted estimators.}
By  condition~\eqref{eq:weight_invariance_condition} and the block-diagonal form of $\Omega$, both $\Omega$ and $(\Sigma_L^{(0)})^{-1}$ preserve $\operatorname{Im}(B)$, namely, $\Omega\operatorname{Im}(B)\subseteq\operatorname{Im}(B)$, and $(\Sigma_L^{(0)})^{-1}\operatorname{Im}(B) \subseteq\operatorname{Im}(B)$.
Since they are symmetric, they also preserve $\operatorname{Im}(B)^\perp$. In particular,
\begin{align}
    B^\top\Omega(I-BB^\dagger)=0, \qquad B^\top(\Sigma_L^{(0)})^{-1}(I-BB^\dagger)=0
\end{align}
where $B^\dagger=(B^\top B)^{-1}B^\top$ denotes the Moore--Penrose pseudoinverse of $B$, and $BB^\dagger$ is the orthogonal projector onto $\operatorname{Im}(B)$. This implies
\begin{align}
    (B^\top\Omega B)^{-1}B^\top\Omega
    &=(B^\top(\Sigma_L^{(0)})^{-1}B)^{-1} B^\top(\Sigma_L^{(0)})^{-1} = B^\dagger.
    \label{eqn:app_E_6}
\end{align}

Using the definitions of $K_\Omega$ and $K_{\mathrm{opt}}$ above, together with~\eqref{eq:weighted_proof_Ht} and~\eqref{eqn:app_E_6}, we obtain
\begin{align}
    \left(\sum_{k=1}^{r}\beta^k\Pi_{U_k}\right)K_\Omega =B^\dagger+O(\beta), \qquad \left(\sum_{k=1}^{r}\beta^k\Pi_{U_k}\right)K_{\mathrm{opt}}
    =B^\dagger+O(\beta).
\end{align}
Therefore, $\left\|\left(\sum_{k=1}^{r}\beta^k\Pi_{U_k}\right) (K_\Omega-K_{\mathrm{opt}})\right\|_{\mathrm{op}} =O(\beta)$.
Since $\|\Sigma_L\|_{\mathrm{op}}=O(1)$, the quadratic identity~\eqref{eq:nonoptimal_gap_identity} now gives
\begin{align}
    \left\| \left(\sum_{k=1}^{r}\beta^k\Pi_{U_k}\right) \Delta_\Omega \left(\sum_{k=1}^{r}\beta^k\Pi_{U_k}\right) \right\|_{\mathrm{op}}
    &=O(\beta^2). \label{eq:weighted_proof_scaled_loss}
\end{align}

\paragraph{Step 4: Fisher normalization}
By Proposition~\ref{prop:fisher_layer_block_asymptotics}, the observed Fisher information $\mathcal{I}_\beta$ satisfies,
\begin{align}
    \left(\sum_{k=1}^{r}\beta^{-k}\Pi_{U_k}\right) \mathcal{I}_\beta \left(\sum_{k=1}^{r}\beta^{-k}\Pi_{U_k}\right) &=B^\top(\Sigma_L^{(0)})^{-1}B+O(\beta) = O(1).
\end{align}
Combining this bound with~\eqref{eq:weighted_proof_scaled_loss}, by submultiplicativity of the operator norm,
\begin{align}
    &\left\|\mathcal{I}_\beta^{1/2}\Delta_\Omega \mathcal{I}_\beta^{1/2}\right\|_{\mathrm{op}}
    \notag\\
    &\quad\le
    \left\| \left(\sum_{k=1}^{r}\beta^{-k}\Pi_{U_k}\right) \mathcal{I}_\beta \left(\sum_{k=1}^{r}\beta^{-k}\Pi_{U_k}\right) \right\|_{\mathrm{op}} \left\| \left(\sum_{k=1}^{r}\beta^k\Pi_{U_k}\right) \Delta_\Omega \left(\sum_{k=1}^{r}\beta^k\Pi_{U_k}\right) \right\|_{\mathrm{op}}
    \notag\\    &\quad=O(\beta^2)=O(\mathrm{SNR}).
\end{align}
Finally, Theorem~\ref{thm:H_fixed_beta_fisher_moment_identity} bounds the moment-truncation loss by $O(\mathrm{SNR}^{L+1-r})=O(\mathrm{SNR})$, since $L\ge r$. Together with the decomposition~\eqref{eq:nonoptimal_total_loss} and the efficiency definitions~\eqref{eq:relative_efficiency_op}--\eqref{eq:relative_efficiency_comparison}, this proves $0 \le \mathcal{E}_{\mathrm{tr}}(\Sigma_{\mathrm{GMoM}}^{(\Omega)};\beta) \le\mathcal{E}_{\mathrm{op}}(\Sigma_{\mathrm{GMoM}}^{(\Omega)};\beta) =O(\mathrm{SNR})$. Every block-isotropic $\Omega=\operatorname{diag}(c_1I_{N_1},\ldots,c_LI_{N_L})$ with $c_k>0$ satisfies~\eqref{eq:weight_invariance_condition}, completing the proof.

\subsection{Anisotropic weighting in a three-component Gaussian mixture}
\label{app:anisotropic_GMM_example}

We derive the limiting efficiency loss in~\eqref{eq:anisotropic_GMM_efficiency_limit}. We consider an equal-weight Gaussian mixture with component means $\mu_1=(-1,0)$, $\mu_2=(0,1)$, and $\mu_3=(1,-1)$. The first two signal moments are $T_1(\theta)=\frac13\sum_{i=1}^3\mu_i$ and $T_2(\theta)=\frac13\sum_{i=1}^3\mu_i^{\otimes2}$. Solving $DT_1(\theta^\star)[h]=DT_2(\theta^\star)[h]=0$ gives the one-dimensional third informative layer $U_3=\operatorname{span}\{h\}$, where $h=(1,-2,-2,1,1,1)^\top$. Thus, $r_{\mathrm{loc}}=3$.

In orthonormal symmetric-tensor coordinates, the third-moment derivative along this direction is $v=DT_3(\theta^\star)[h]=(2,-\sqrt3,-\sqrt3,2)^\top$. Since the zero-signal covariance of the third-order Hermite features is $\Sigma_3^{(0)}=3!I_4$, the limiting variance inflation on this one-dimensional layer under weighting $\Omega_{A,3}$ is $(v^\top v)(v^\top\Omega_{A,3}^2v)/(v^\top\Omega_{A,3}v)^2-1$. The lower informative layers incur no limiting loss under the identity weighting blocks.

For $\Omega_{A,3}=\operatorname{diag}(1,10,1,1)$, we obtain $v^\top v=14$, $v^\top\Omega_{A,3}v=41$, and $v^\top\Omega_{A,3}^2v=311$. Consequently,
\begin{align}
    \lim_{\beta\to 0} \mathcal{E}_{\mathrm{op}}    (\Sigma_{\mathrm{GMoM}}^{(\Omega_A)};\beta)=\frac{14\cdot311}{41^2}-1 = \frac{2673}{1681}>0.
\end{align}
The positive limiting loss arises because $\Omega_{A,3}v$ is not proportional to $v$, so the weighting matrix does not preserve the third-order informative subspace.

\subsection{Efficiency loss of frequency marching in cyclic MRA}
\label{app:FM_efficiency_derivation}

Recall that we consider signals of odd dimension $d=2m+1\geq 7$.
We compare optimally weighted GMoM using all moments through order $L=3$ with GMoM using only the Fourier mean, power spectrum, and the frequency-marching bispectrum entries
\begin{align}
    B(1,k),\quad 1\le k\le m-1,
    \qquad B(m,m),  \label{eq:FM_app_selected_entries}
\end{align}
where $B(k,\ell)=\widehat\theta_k\widehat\theta_\ell \overline{\widehat\theta_{k+\ell}}$. Let $S$ denote the linear map extracting these selected Fourier coordinates from the stacked Hermite features. The corresponding information matrices in the Fourier-phase coordinates are
\begin{align}
    \mathcal{I}_{\phi,\mathrm{full}} &=(D_\phi M)^\top\Sigma_3^{-1}D_\phi M,
    \notag\\ \mathcal{I}_{\phi,\mathrm{FM}}
    &=(D_\phi M)^\top S^\top  (S\Sigma_3S^\top)^{-1}S D_\phi M, \label{eq:FM_app_information_matrices}
\end{align}
where $D_\phi M$ is the Jacobian of the moment map with respect to the independent Fourier phases. The following lemmas characterize the two information matrices.

\begin{lem}[Leading phase information from the full third moment]
\label{lem:FM_full_phase_information}
For the centered flat-spectrum model, the leading phase-information matrix of full GMoM satisfies
\begin{align}
    \mathcal{I}_{\phi,\mathrm{full}}(\beta) &=a^6\beta^6 H_{\mathrm{full}}+O(\beta^8),    \label{eq:FM_app_full_information}
\end{align}
where $H_{\mathrm{full}}=dI_m$. Thus, the full third moment provides equal leading-order information in all independent Fourier-phase directions.
\end{lem}

\begin{proof}
Let $u=(u_1,\ldots,u_m)^\top$ denote a phase perturbation, extended to $\mathbb{Z}_d$ by $u_0=0$ and $u_{d-k}=-u_k$. Since the first two signal moments are independent of the Fourier phases, Corollary~\ref{cor:dir_exponents_Iobs} applied to the GMoM information matrix gives,
\begin{align}
    u^\top\mathcal{I}_{\phi,\mathrm{full}}(\beta)u
    &=\frac{\beta^6}{3!} \|D_\phi T_3(\theta^\star)[u]\|_F^2 +O(\beta^8), \label{eq:FM_app_information_derivative}
\end{align}
where the remainder follows from $\Sigma_3=\Sigma_3^{(0)}+O(\beta^2)$ for the centered model.

The three-fold Fourier transform of the third signal moment satisfies $(\mathcal{F}^{\otimes3}T_3(\theta^\star))_{k,\ell,q} =\widehat\theta_k^\star \widehat\theta_\ell^\star \widehat\theta_q^\star \mathbf{1}_{\{k+\ell+q=0\}}$. Differentiating with respect to the phases and using Parseval's identity,
\begin{align}
    \|D_\phi T_3(\theta^\star)[u]\|_F^2 =a^6\sum_{k,\ell\in\mathbb{Z}_d} (u_k+u_\ell-u_{k+\ell})^2.
\end{align}
Consequently, by~\eqref{eq:FM_app_information_derivative}, $u^\top H_{\mathrm{full}}u =\frac{1}{3!}\sum_{k,\ell\in\mathbb{Z}_d} (u_k+u_\ell-u_{k+\ell})^2$. Since $\sum_{k\in\mathbb{Z}_d}u_k=0$, expanding the square eliminates all cross terms, yielding
\begin{align}
    u^\top H_{\mathrm{full}}u &=\frac d2\sum_{k\in\mathbb{Z}_d}u_k^2 =d\sum_{k=1}^{m}u_k^2.
\end{align}
Hence $H_{\mathrm{full}}=dI_m$.
\end{proof}

\begin{lem}[Leading phase information from frequency marching]
\label{lem:FM_selected_phase_information}
For optimally weighted GMoM using the selected moments in~\eqref{eq:FM_app_selected_entries}, the leading phase-information matrix satisfies
\begin{align}
    \mathcal{I}_{\phi,\mathrm{FM}}(\beta) =a^6 \beta^6 H_{\mathrm{FM}}+O(\beta^8),
    \label{eq:FM_app_selected_information}
\end{align}
where

\begin{align}
    H_{\mathrm{FM}}=
    \begin{pmatrix}
        d & 0 & & & \\
        0 & 3 & -2 & & \\
        & -2 & 4 & -2 & \\
        & & -2 & \ddots & \ddots \\
        & & & \ddots & 4 & -2 \\
        & & & & -2 & 6
    \end{pmatrix}.
    \label{eq:FM_app_tridiagonal}
\end{align}
The diagonal entries are $d,3, 4,\ldots,4,6$. In particular, $H_{\mathrm{FM}}\succ 0$, and the selected moments identify all phase directions locally.
\end{lem}

\begin{proof}
Let $u=(u_1,\ldots,u_m)^\top$ be a phase perturbation, extended to $\mathbb{Z}_d$ by $u_0=0$ and $u_{d-k}=-u_k$. We first express the information in terms of the selected moment derivatives. Applying \eqref{eq:def_GMoM_information} to the reduced moment map $S M$ gives
\begin{align}
    \mathcal{I}_{\phi,\mathrm{FM}}(\beta) =(S D_\phi M)^\top (S \Sigma_3 S^\top)^{-1} (S D_\phi M).
\end{align}
The first two moments have zero phase derivatives. Let $S_3$ denote the selection map restricted to the third-order features. By Proposition~\ref{prop:sigma0_closed_form_updated}, their zero-signal covariance after selection is $3! S_3 S_3^\top$. Since the signal is centered, $\Sigma_3(\beta)=\Sigma_3^{(0)}+O(\beta^2)$. Consequently,
\begin{align}
    u^\top\mathcal{I}_{\phi,\mathrm{FM}}(\beta)u = \frac{\beta^6}{3!} \|D_\phi T_3(\theta^\star)[u]\|_{\mathrm{sel}}^2 +O(\beta^8),    
    \label{eq:FM_app_selected_quadratic}
\end{align}
where $\|\cdot\|_{\mathrm{sel}}$ is the Frobenius norm of the orthogonal projection onto the selected symmetric Fourier-tensor subspace.

It remains to compute this derivative norm. Since $B(k,\ell)=\widehat\theta_k\widehat\theta_\ell \overline{\widehat\theta_{k+\ell}}$, we have
\begin{align}
    D_\phi B(k,\ell)[u] = iB(k,\ell)(u_k+u_\ell-u_{k+\ell}).
\end{align}
For the selected bispectrum entries, the corresponding phase derivatives are $v_k^\top u$, with
\begin{align}
    v_1&=2e_1-e_2,\notag\\
    v_k&=e_1+e_k-e_{k+1},
    \quad 2\le k\le m-1,\notag\\
    v_m&=e_1+2e_m,
\end{align}
where $e_1,\ldots,e_m$ are the standard basis vectors of $\mathbb{R}^m$.

All selected entries have magnitude $a^3$. The corresponding Fourier triples have multiplicities $6$, $12$, and $6$, respectively, counting distinct permutations and conjugates. Therefore,
\begin{align}
    \|D_\phi T_3(\theta^\star)[u]\|_{\mathrm{sel}}^2 = a^6\Bigg[ 6(v_1^\top u)^2 +12\sum_{k=2}^{m-1}(v_k^\top u)^2 + 6(v_m^\top u)^2 \Bigg].
\end{align}
Substituting into~\eqref{eq:FM_app_selected_quadratic} and comparing with~\eqref{eq:FM_app_selected_information} yields
\begin{align}
    H_{\mathrm{FM}} = v_1 v_1^\top  + 2\sum_{k=2}^{m-1}v_kv_k^\top + v_mv_m^\top.
\end{align}
Expanding this expression gives the stated matrix.

Finally, if $v_k^\top u=0$ for every $k$, the first $m-1$ relations imply $u_k=ku_1$, while the last gives $du_1=0$. Thus $u=0$, proving $H_{\mathrm{FM}}\succ 0$.
\end{proof}

We now quantify the resulting statistical efficiency loss and its dependence on the signal dimension.

\begin{lem}[Relative efficiency loss of frequency marching]
\label{lem:FM_efficiency_loss}
Let $d=2m+1\ge 7$, and let $T_{m-1}$ denote the lower-right $(m-1)\times(m-1)$ block of $H_{\mathrm{FM}}$ in~\eqref{eq:FM_app_tridiagonal}. Then
\begin{align}
    \lim_{\beta\to 0} \mathcal{E}_{\mathrm{op}}(\Sigma_{\mathrm{FM}};\beta)
    &=\frac{d}{\lambda_{\min}(T_{m-1})}-1,
    \label{eq:FM_app_op_gap}\\
    \lim_{\beta\to 0} \mathcal{E}_{\mathrm{tr}}(\Sigma_{\mathrm{FM}};\beta)
    &=\frac{(d-5)(d-3)(d+8)}{24(d-1)}.
    \label{eq:FM_app_trace_closed_form}
\end{align}
Moreover, taking $\beta\to 0$ first and then $d\to\infty$, these losses satisfy
\begin{align}
    \lim_{\beta\to 0} \mathcal{E}_{\mathrm{op}}(\Sigma_{\mathrm{FM}};\beta)
    &\sim\frac{d^3}{8\pi^2},
    \notag\\
    \lim_{\beta\to 0} \mathcal{E}_{\mathrm{tr}}(\Sigma_{\mathrm{FM}};\beta)
    &\sim\frac{d^2}{24}.
    \label{eq:FM_app_dimension_scaling}
\end{align}
\end{lem}

\begin{proof}
By Lemmas~\ref{lem:FM_full_phase_information} and~\ref{lem:FM_selected_phase_information}, $H_{\mathrm{full}}=dI_m$ and $H_{\mathrm{FM}}=\operatorname{diag}(d,T_{m-1})$. Consequently, by the definition of $\mathcal{E}_{\mathrm{op}}$ and $\mathcal{E}_{\mathrm{tr}}$ from~\eqref{eq:relative_efficiency_op}--~\eqref{eq:relative_efficiency_tr},
\begin{align}
    \lim_{\beta\to 0} \mathcal{E}_{\mathrm{op}}(\Sigma_{\mathrm{FM}};\beta)
    &=\lambda_{\max} (dH_{\mathrm{FM}}^{-1})-1 = \frac{d}{\lambda_{\min}(T_{m-1})}-1,
    \notag\\
    \lim_{\beta\to 0} \mathcal{E}_{\mathrm{tr}}(\Sigma_{\mathrm{FM}};\beta)
    &=\frac{d}{m} \left(\frac1d+ \operatorname{tr}(T_{m-1}^{-1})\right)-1.
    \label{eq:FM_app_trace_gap}
\end{align}

To evaluate the trace, direct inversion of the tridiagonal matrix gives
\begin{align}
    [T_{m-1}^{-1}]_{jj}
    &=\frac{(j+1)(2m-2j-1)}{2(2m+1)},
    \quad 1\le j\le m-1,
    \label{eq:FM_app_inverse_diagonal}
\end{align}
and hence
\begin{align}
    \operatorname{tr}(T_{m-1}^{-1}) =\frac{(m-1)(2m^2+5m-6)}{12(2m+1)}.
\end{align}
Substitution into~\eqref{eq:FM_app_trace_gap} yields~\eqref{eq:FM_app_trace_closed_form}.

For the operator-norm loss, we analyze the smallest eigenvalue of the matrix $T_{m-1}$. Let $w=(w_1,\ldots,w_{m-1})^\top$ be an eigenvector with eigenvalue $\lambda$. The eigenvalue equation can be written as
\begin{align}
    -2w_{j-1}+4w_j-2w_{j+1} = \lambda w_j,
    \qquad 1\le j\le m-1,
\end{align}
with boundary conditions $w_0=w_1/2$, and $w_m=-w_{m-1}$, which follow from the first and last diagonal entries of $T_{m-1}$.
The recurrence admits solutions of the form
\begin{align}
    w_j=\sin(j\tau+\delta),
    \qquad
    \lambda=4-4\cos\tau,
\end{align}
where $\tau\in(0,\pi)$.
Substituting the boundary conditions yields
\begin{align}
    \tan\delta
    &=\frac{\sin\tau}{2-\cos\tau},
    \label{eq:FM_app_phase_boundary}\\
    \left(m-\frac12\right)\tau+\delta
    &=\pi,
    \label{eq:FM_app_eigenvalue_condition}
\end{align}
where the second equation corresponds to the
smallest eigenvalue.

As $m\to\infty$, the smallest positive solution satisfies $\tau\to 0$. Expanding the first boundary condition gives $\delta=\tau+O(\tau^3)$. Substituting into \eqref{eq:FM_app_eigenvalue_condition}, we obtain
\begin{align}
    \left(m+\frac12\right)\tau
    +O(\tau^3)=\pi.
\end{align}
Since $d=2m+1$, it follows that
\begin{align}
    \tau\sim\frac{\pi}{m+1/2}
    =\frac{2\pi}{d}.
\end{align}
Consequently,
\begin{align}
    \lambda_{\min}(T_{m-1}) =4-4\cos\tau =2\tau^2+O(\tau^4) \sim\frac{8\pi^2}{d^2}. \label{eq:FM_app_smallest_eigenvalue}
\end{align}
Combining this asymptotic with \eqref{eq:FM_app_op_gap} and \eqref{eq:FM_app_trace_closed_form} establishes~\eqref{eq:FM_app_dimension_scaling}.
\end{proof}

\section{Numerical simulation details}
\label{app:numerical_details}

We provide additional implementation details for the
numerical experiments in the main text.

\paragraph{Population-level experiments.}
For Figure~\ref{fig:1}, the ground-truth parameters are fixed generic configurations, specified in the publicly available MATLAB implementation. The remaining numerical settings, including quadrature orders and finite-difference increments, are given in the figure caption.

\paragraph{Finite-sample experiments.}
The nonlinear estimators are initialized at
\begin{align}
    \theta^{(0)}=\theta^\star+\tau\xi,
    \qquad \xi\sim\mathcal{N}(0,I),
\end{align}
with the same initialization used for all nonlinear estimators within each Monte Carlo trial. The perturbation magnitudes are $\tau=0.12,0.10,0.04,$ and $0.06$ for Figures~\ref{fig:mainFigure}(a), ~\ref{fig:mainFigure}(b), \ref{fig:5}, and \ref{fig:4},  respectively.

For Figure~\ref{fig:mainFigure}(a), the component means of the three-component Gaussian mixture are
\begin{align}
    \begin{pmatrix}
    1.2 & 0.1 & 0 & 0.3\\
    -0.7 & 0.8 & 0.2 & -0.5\\
    -0.3 & -1.1 & -0.4 & 0.9
    \end{pmatrix}.
\end{align}
The anisotropic weighting is $\Omega_A=I_{34}+34e_{144}e_{144}^{\top}$ in orthonormal symmetric-tensor coordinates, where $e_{144}$ denotes the basis vector associated with the symmetric third-order tensor entry $(1,4,4)$. 

For Figure~\ref{fig:mainFigure}(b), the MRA signal is $\theta^\star=(1.2,-0.4,-1.1,0.7,0.3)^\top$. For Figure~\ref{fig:5}, the MRA signals are $\theta^\star= \left(\sqrt{d-1},-\frac{1}{\sqrt{d-1}},\ldots, -\frac{1}{\sqrt{d-1}}\right)^\top$. Frequency marching is computed directly from empirical Fourier invariants, without oracle initialization.
For Figure~\ref{fig:4}, the MRA signal is $\theta^\star=(1.2,-0.4,-1.1)^\top$, and moments through order $L=3$ are used. 

\end{appendices}

\end{document}